\documentclass[12pt]{amsart}
\usepackage{cases}
\usepackage{amsthm}
\usepackage{amsmath}
\usepackage{amscd}
\usepackage{graphicx}
\usepackage{float}
\usepackage[mathscr]{eucal}
\usepackage[colorlinks,linkcolor=blue,citecolor=blue, pdfstartview=FitH]{hyperref}

\input xy
\xyoption{all} \numberwithin{equation}{section}
\begin{document}
\title[On the asymptotic expansions I]{On the asymptotic expansions of various quantum invariants I: the colored Jones polynomial of twist knots at the root of unity $e^{\frac{2\pi\sqrt{-1}}{N+\frac{1}{2}}}$}

\author[Qingtao Chen and Shengmao Zhu]{Qingtao Chen and
Shengmao Zhu}

\address{Department of Pure Mathematics \\
Xi'an Jiaotong-Liverpool University \\
Suzhou Jiangsu \\
China}
\email{Qingtao.Chen@xjtlu.edu.cn}

\address{Department of Mathematics \\
Zhejiang Normal University,  \\
Jinhua Zhejiang,  321004, China }
\email{szhu@zju.edu.cn}

\begin{abstract}
This is the first article in a series devoted to the study of the asymptotic expansions of various quantum invariants related to the twist knots. In this paper,  by using the saddle point method developed by Ohtsuki, we obtain an asymptotic expansion formula for the colored Jones polynomial of twist knots $\mathcal{K}_p$ with $p\geq 6$ at the root of unity $e^{\frac{2\pi\sqrt{-1}}{N+\frac{1}{2}}}$. 
\end{abstract}
\maketitle

\theoremstyle{plain} \newtheorem{thm}{Theorem}[section] \newtheorem{theorem}[%
thm]{Theorem} \newtheorem{lemma}[thm]{Lemma} \newtheorem{corollary}[thm]{%
Corollary} \newtheorem{proposition}[thm]{Proposition} \newtheorem{conjecture}%
[thm]{Conjecture} \theoremstyle{definition}
\newtheorem{remark}[thm]{Remark}
\newtheorem{remarks}[thm]{Remarks} \newtheorem{definition}[thm]{Definition}
\newtheorem{example}[thm]{Example}
\tableofcontents
\newpage

\section{Introduction}
In this series of articles, we study the asymptotic expansions of various quantum invariants at different roots of unit and make a connection between them. 

This work is motivated by the volume conjectures, let us briefly review the background.   In \cite{Kash95}, by using the quantum dilogarithm function, R. Kashaev defined a link invariant $\langle \mathcal{L} \rangle_{N}$  for a link $\mathcal{L}$, which depends on a positive integer $N$. Furthermore, in \cite{Kash97}, he conjectured that for any hyperbolic link $\mathcal{L}$, the asymptotics at $N\rightarrow \infty$ of $|\langle \mathcal{L}\rangle_N|$ gives its volume, i.e. 
 \begin{align}
     2\pi \lim_{N\rightarrow \infty}\frac{\log |\langle\mathcal{L}\rangle_N|}{N}=vol(S^3\setminus \mathcal{L})
 \end{align}
where $vol(S^3\setminus \mathcal{L})$ denotes the hyperbolic volume of the complement of $\mathcal{L}$ in $S^3$, and gave evidence for the conjecture. In  \cite{MuMu01}, H. Murakami and J. Murakami proved that for any link $\mathcal{L}$, Kashaev's invariant 
$\langle \mathcal{L}\rangle_N$ is equal to $N$-th normalized colored Jones polynomial  evaluated at the root of unity $e^{\frac{2\pi \sqrt{-1}}{N}}$, which is written as $J_{N}(\mathcal{L};e^{\frac{2\pi \sqrt{-1}}{N}})$. Further, as an extension of  Kashaev's conjecture, they conjectured that, for any knot $\mathcal{K}$,  
\begin{align} \label{formula-original-volume}
    2\pi \lim_{N\rightarrow \infty} \frac{\log |J_N(\mathcal{K}; e^{\frac{2\pi \sqrt{-1}}{N}})|}{N}=vol(S^3 \setminus \mathcal{K}),
\end{align}
where $vol(S^3\setminus \mathcal{K})$ denotes the simplicial volume (normalized by multiplying by the hyperbolic volume of the regular ideal tetrahedron) of the complement of $\mathcal{K}$ in $S^3$. This is usually called the (Kashaev-Murakami-Murakami) volume conjecture. As a complexification of the volume conjecture, it is conjectured in \cite{MMOTY02} that, for a hyperbolic link $\mathcal{L}$, 
    \begin{align}
    2\pi \lim_{N\rightarrow \infty} \frac{\log J_N(\mathcal{L};e^{\frac{2\pi \sqrt{-1}}{N}})}{N}=vol(S^3\setminus \mathcal{L})+\sqrt{-1}cs(S^3\setminus \mathcal{L}),
\end{align}
for an appropriate choice of a branch of the logarithm, where $``cs"$ denotes the Chern-Simons invariant \cite{Mey86}. 
From the viewpoint of the $SL(2,\mathbb{C})$ Chern-Simons theory, S. Gukov conjectured in \cite{Guk05} that the asymptotic expansion of $J_{N}(\mathcal{K};e^{\frac{2\pi\sqrt{-1}}{k}})$ of a hyperbolic knot $\mathcal{K}$ as $N,k\rightarrow \infty$ fixing $u=\frac{N}{k}$ is presented by the following form, 
\begin{align} \label{formula-generalvolume}
    J_{N}(\mathcal{K};e^{\frac{2\pi\sqrt{-1}}{k}})\sim e^{N\zeta}N^{\frac{3}{2}}\omega\left(1+\sum_{i=1}^{\infty}\kappa_i\left(\frac{2\pi\sqrt{-1}}{N}\right)^i\right)
\end{align}
for some scalars $\zeta,\omega,\kappa_i$ depending on $\mathcal{K}$ and $u$, also see \cite{DGLZ09,GH08}. Moreover, T. Ohtuski showed when $\mathcal{K}$ is a hyperbolic knot with up to 7 crossings \cite{Oht16,OhtYok18,Oht17},  the asymptotic expansion of the Kashaev invariant is presented by the following form 
\begin{align} \label{formula-Kashaevexpansion}
    \langle \mathcal{K}\rangle_{N}=e^{N\zeta}N^{\frac{3}{2}}\omega(\mathcal{K})\left(1+\sum_{i=1}^d\kappa_i(\mathcal{K})\left(\frac{2\pi\sqrt{-1}}{N}\right)^i+O\left(\frac{1}{N^{d+1}}\right)\right),
\end{align}
for any $d$, where $\omega(\mathcal{K})$ and $\kappa_i(\mathcal{K})$'s are some scalars. 

The volume conjecture has been rigorously proved for some particular knots and links such as torus knots \cite{KT00,DKash07}, the figure-eight knot \cite{AndHan06}, Whitehead doubles of $(2,p)$-torus knots \cite{Zheng07}, positive iterated torus knots \cite{Van08}, the $5_2$ knot \cite{Oht16}, the knots with 6 crossings \cite{OhtYok18}, the knots with 7 crossings \cite{Oht17} and some links \cite{GL05,Van08,Van08-2,YY10,Zheng07}, see also \cite{Mur10} for a review. 

On the other hand, it is known that the quantum invariant of a closed 3-manifold at $q=e^{\frac{2\pi\sqrt{-1}}{N}}$ is of polynomial order as $N\rightarrow \infty$. However, the first author and T. Yang  \cite{CY18} observed that  Reshetikhin-Turaev invariants and Turaev-Viro invariants at $q=e^{\frac{4\pi\sqrt{-1}}{r}}$, for $r\geq 3$ an odd, are of exponential order as $r\rightarrow \infty$. Furthermore,  they proposed the volume conjecture for Reshetikhin-Turaev invariants and Turaev-Viro invariants. 

In \cite{DKY18},  Detcherry, Kalfagianni and Yang  gave a formula relating the Turaev-Viro invariants of the
complement of a link $\mathcal{L}$ in $S^3$ to the values of the colored Jones polynomials of $\mathcal{L}$. By using this formula, they proved Chen-Yang's volume conjecture \cite{CY18} for the figure-eight knot and Borromean rings.
In addition, they proposed the following 
\begin{conjecture}[\cite{DKY18}, Question 1.7] \label{conjecture-DKY}
    For a hyperbolic link $\mathcal{L}$ in $S^3$, we have
   \begin{align}
       \lim_{N\rightarrow \infty}\frac{2\pi}{N} \log|J_{N}(\mathcal{L}; e^{\frac{2\pi\sqrt{-1}}{N+\frac{1}{2}}})|=vol(S^3\setminus \mathcal{L}).
   \end{align} 
\end{conjecture}

The asymptotic behavior of  $J_{N}(\mathcal{L}; e^{\frac{2\pi\sqrt{-1}}{N+\frac{1}{2}}})$  is not predicted
either by the original volume conjecture (\ref{formula-original-volume})  or by its generalizations (\ref{formula-generalvolume}).  Moreover, Conjecture \ref{conjecture-DKY}  seems somewhat surprising, since a result in \cite{GL11,CLZ15} has stated that for any positive integer $k$, $J_{N}(\mathcal{L};e^{\frac{2\pi\sqrt{-1}}{N+k}})$ grows only polynomially in $N$. Conjecture \ref{conjecture-DKY} has been proved for figure-eight knot and Borromean ring in \cite{DKY18},  we also refer to \cite{Wong19} for an extended version of Conjecture \ref{conjecture-DKY}. 

The purpose of this paper is to study Conjecture \ref{conjecture-DKY} for the twist knot $\mathcal{K}_p$. We investigate the asymptotic expansion for the normalized $N$-th colored Jones polynomial  $J_N(\mathcal{K}_p;e^{\frac{2\pi \sqrt{-1}}{N+\frac{1}{2}}})$ instead. Furthermore, in a subsequent paper \cite{CZ23-2}, we present two asymptotic expansions  for the normalized $N$-th colored Jones polynomials of twist knots at the more general root unity $e^{\frac{2\pi \sqrt{-1}}{N+\frac{1}{M}}}$ with $M\geq 2$, and at the root of unity  $e^{\frac{2\pi \sqrt{-1}}{N}}$ respectively. Moreover, the asymptotic expansion for  Reshetikhin-Turaev invariants of closed hyperbolic 3-manifolds obtained by integral surgery along the twist knot at the root of unity   $e^{\frac{4\pi\sqrt{-1}}{r}}$ will be given in \cite{CZ23-3}. Finally, the last article \cite{CZ23-3} in this series is devoted to the asymptotic expansion for the Turaev-Viro invariants of the complements of the twist knots in $S^3$.

Let $V(p,t,s)$ be the potential function of the colored Jones polynomial for the twist knot $\mathcal{K}_p$ given by formula (\ref{formula-potentialfunction00}). 
By Proposition \ref{prop-critical}, there exists a unique critical point $(t_0,s_0)$ of $V(p,t,s)$. Let $x_0=e^{2\pi\sqrt{-1}t_0}$ and $y_0=e^{2\pi\sqrt{-1}s_0}$, we put
\begin{align}
  \zeta(p)&=V(p,t_0,s_0)\\\nonumber
  &=\pi \sqrt{-1}\left((2p+1)s_0^2-(2p+3)s_0-2t_0\right)\\\nonumber
    &+\frac{1}{2\pi\sqrt{-1}}\left(\text{Li}_2(x_0y_0)+\text{Li}_2(x_0/y_0)-3\text{Li}_2(x_0)+\frac{\pi^2}{6}\right)  
\end{align}
and 
\begin{align} \label{formula-omegap}
    \omega(p)&=\frac{\sin (2\pi s_0)e^{2\pi\sqrt{-1}t_0}}{(1-e^{2\pi\sqrt{-1}t_0})^{\frac{3}{2}}\sqrt{\det Hess(V)(t_0,s_0)}}\\\nonumber
    &=\frac{(y_0-y_0^{-1})x_0}{-4\pi (1-x_0)^\frac{3}{2}\sqrt{H(p,x_0,y_0)}}
\end{align}
with
\begin{align}
    H(p,x_0,y_0)&=\left(\frac{-3(2p+1)}{\frac{1}{x_0}-1}+\frac{2p+1}{\frac{1}{x_0y_0}-1}+\frac{2p+1}{\frac{1}{x_0/y_0}-1}-\frac{3}{(\frac{1}{x_0}-1)(\frac{1}{x_0y_0}-1)}\right.\\\nonumber
    &\left.-\frac{3}{(\frac{1}{x_0}-1)(\frac{1}{x_0/y_0}-1)}+\frac{4}{(\frac{1}{x_0y_0}-1)(\frac{1}{x_0/y_0}-1)}\right).
\end{align}

Then, we have
\begin{theorem}  \label{theorem-main}
  For $p\geq 6$, the asymptotic expansion of the colored Jones polynomial $J_N(\mathcal{K}_p;e^{\frac{2\pi \sqrt{-1}}{N+\frac{1}{2}}})$ is given by the following form
    \begin{align}
       J_N(\mathcal{K}_p;e^{\frac{2\pi \sqrt{-1}}{N+\frac{1}{2}}})&=(-1)^{p+1}\frac{4\pi e^{\frac{1}{4}\pi\sqrt{-1}}(N+\frac{1}{2})^{\frac{1}{2}}}{\sin\frac{\pi}{2N+1}}\omega(p)e^{(N+\frac{1}{2})\zeta(p)}\\\nonumber
       &\cdot\left(1+\sum_{i=1}^d\kappa_i(p)\left(\frac{2\pi\sqrt{-1}}{N+\frac{1}{2}}\right)^i+O\left(\frac{1}{(N+\frac{1}{2})^{d+1}}\right)\right),
    \end{align}
    for $d\geq 1$, where $\kappa_i(p)$ are constants determined by $\mathcal{K}_p$.
\end{theorem}
By Theorem  \ref{theorem-zeta=volume}, we know that
\begin{align}
     2\pi \zeta(p)= vol(S^3\setminus \mathcal{K}_p)+\sqrt{-1}cs(S^3\setminus \mathcal{K}_p) \mod \pi^2\sqrt{-1}\mathbb{Z}.
\end{align}
It implies that
\begin{corollary}
    For $p\geq 6$, we have
    \begin{align}
        \lim_{N\rightarrow \infty}\frac{2\pi}{N}\log J_N(\mathcal{K}_p;e^{\frac{2\pi \sqrt{-1}}{N+\frac{1}{2}}})=vol(S^3\setminus \mathcal{K}_p)+\sqrt{-1}cs(S^3\setminus \mathcal{K}_p) \mod \pi^2\sqrt{-1}\mathbb{Z}.
    \end{align}
\end{corollary}
Hence we proved Conjecture \ref{conjecture-DKY} for twist knot $\mathcal{K}_p$ with $p\geq 6$.  

Moreover, by Theorem \ref{theorem-omega=Rtorsion}, we know that $\omega(p)$ is related to the Reidemeister torsion of $\mathcal{K}_p$.

\begin{example}
    For $p=100$, we compute that 
    \begin{align}
           t_0&=0.8237997818-0.1280592525\sqrt{-1}, \\\nonumber
           s_0&=0.5050124998-0.00001256317546\sqrt{-1}, \\\nonumber
           x_0&=1.000001243-1.999752031\sqrt{-1}, \\\nonumber
           y_0&=-0.9995829910-0.03149174478\sqrt{-1}, \\\nonumber
           2\pi \zeta(100)&=3.6636144-1043.809608\sqrt{-1}.
    \end{align}
\end{example}
\begin{remark}
    We need the condition $p\geq 6$ in Theorem \ref{theorem-main} which makes sure the volume $vol(S^3\setminus \mathcal{K}_p)$ is not too small, so we can construct the required homotopy and verify the assumptions of the saddle point method successfully. We remark that our method can also work for the cases of $p\leq -1$ with some exceptions.     
\end{remark}

We use the saddle point method developed by Ohtsuki in a series papers  \cite{Oht16,  Oht17, Oht18, OhtYok18}  to prove Theorem \ref{theorem-main}.   An outline of the proof is follows.  First, we write the colored Jones polynomial of the twist knot $J_N(\mathcal{K}_p;e^{\frac{2\pi\sqrt{-1}}{N+\frac{1}{2}}})$  as a summation of Fourier coefficients with the help of  quantum dilogarithm function and the Poisson summation formula. Next, we will show that infinite terms of these Fourier coefficients can be neglected in the sense that they can be sufficiently small order at $N\rightarrow \infty$, we obtain formula (\ref{formula-Poission-after}). Then we estimate the remained finite terms of Fourier coefficients by using the saddle point method, and we find that only two main Fourier coefficients will contribute to the asymptotic expansion formula. Finally, we use the Neumann-Zagier's theory to prove that the critical value of the potential function gives the complex volume of the twist knot.  Hence we finish the proof Theorem \ref{theorem-main}.

The paper is organized as follows. In Section \ref{section-preliminaries}, we fix the notations and review the related materials that will be used in this paper. In Section \ref{section-potentialfunction}, we compute the potential function for the colored Jones polynomials of the twist knot $\mathcal{K}_p$ and obtain Proposition \ref{prop-coloredJonespotential}. In Section  \ref{Section-poissonsummation}, we prove Proposition \ref{prop-fouriercoeff}  which expresses the colored Jones polynomial of the twist knot 
$J_N(\mathcal{K}_p;e^{\frac{2\pi\sqrt{-1}}{N+\frac{1}{2}}})$  as a summation of Fourier coefficients by Poisson summation formula.
In Section \ref{Section-Asympoticexpansion}, 
we first show that infinite terms of the Fourier coefficients can be neglected. Then we estimate the remained finite terms of Fourier coefficients by using the saddle point method, we obtain that only two main Fourier coefficients will contribute to the final form of the  asymptotic expansion.  Hence we finish the proof Theorem  \ref{theorem-main}.  In Section \ref{Section-geometry}, we establish the equivalence of the critical equation from the potential function and the geometric equation from Thurstion's hyperbolic geometry. Then we prove that the potential function and determinant of Hessian matrix at the critical point will give the complex volume and adjoint twisted Reidemeister torsion of hyperbolic 3-manifold respectively.   The final  Section \ref{Section-App} is devoted to the proof of several lemmas which will be used in previous sections.

\textbf{Acknowledgements.} 

The first author would like to thank Nicolai Reshetikhin, Kefeng Liu and Weiping Zhang for bringing him to this area and a lot of discussions during his career, thank Francis Bonahon,   Giovanni Felder and Shing-Tung Yau for their continuous encouragement, support and discussions, and thank Jun Murakami and Tomotada Ohtsuki for their helpful discussions and support. He also want to thank Jørgen Ellegaard Andersen, Sergei Gukov, Thang Le, Gregor Masbaum,  Rinat Kashaev, Vladimir Turaev and Hiraku Nakajima for their support, discussions and interests, and thank Yunlong Yao who built him solid analysis foundation twenty years ago.

The second author would like to thank Kefeng Liu and Hao Xu for bringing him to this area when he was a graduate student at CMS of Zhejiang University, and for their constant encouragement and helpful discussions since then.

\section{Preliminaries} \label{section-preliminaries}
\subsection{Colored Jones polynomials}
In this subsection, we review the definition of the colored Jones polynomials and fix the notations.  
 Let $M$ be an oriented 3-manifold, the Kauffman bracket skeim module $\mathcal{K}(M)$ is the free $\mathbb{Z}[A^{\pm 1}]$-module generated by isotopy classes of framed links in $M$ modulo the submodule generated by  the Kauffman bracket skein relation:

(1) Kauffman bracket skein relation: 
\begin{figure}[!htb] 
\begin{align}
\raisebox{-15pt}{
\includegraphics[width=65 pt]{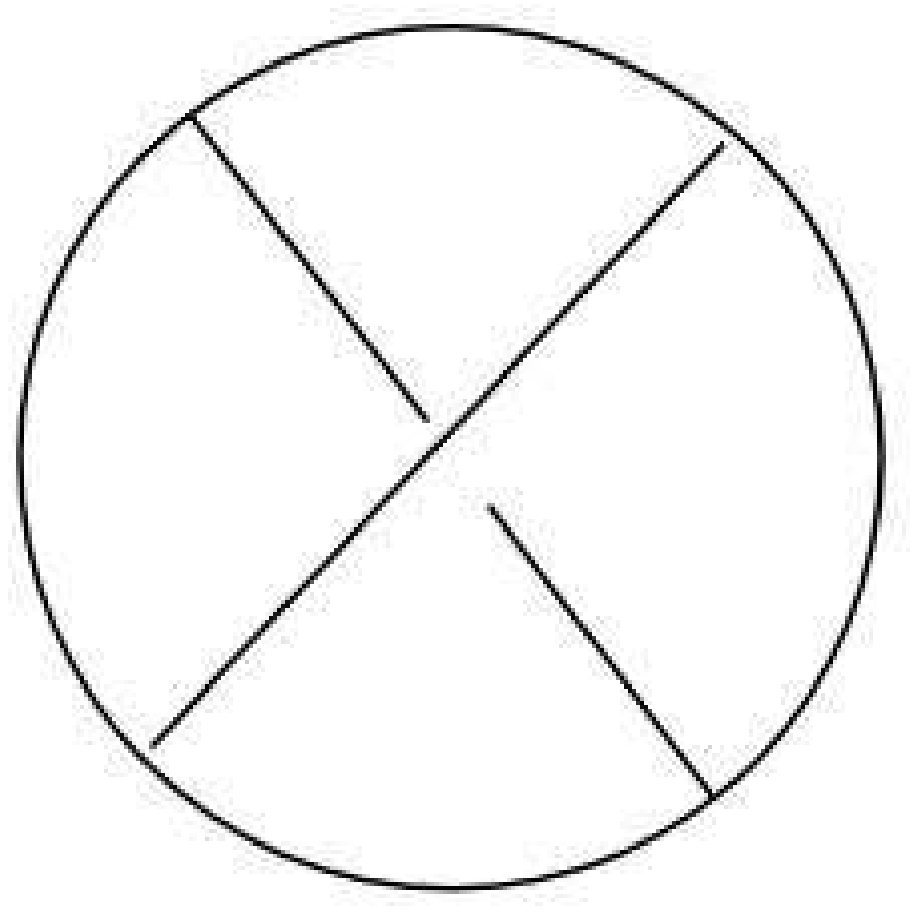}}=A\raisebox{-13pt}{ \includegraphics[width=60 pt]{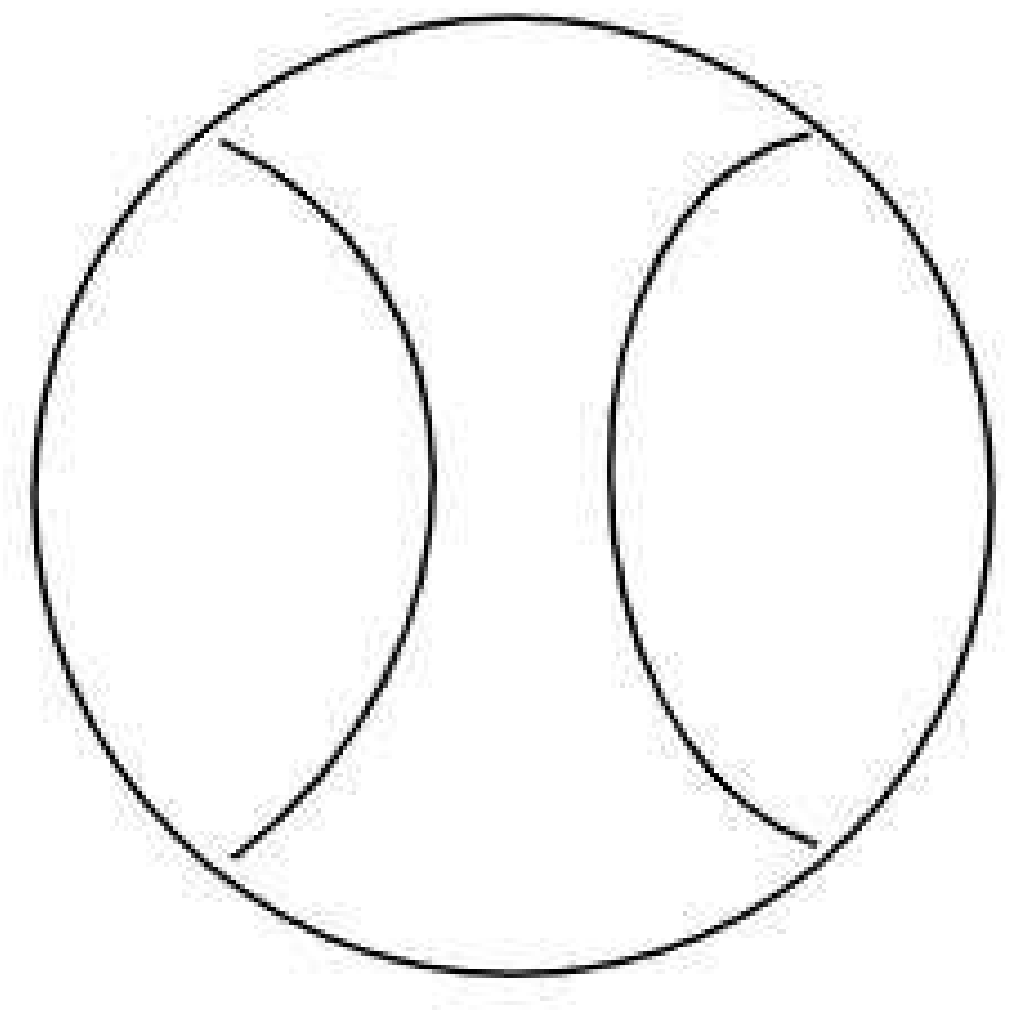}}+A^{-1}\raisebox{-15pt}{ \includegraphics[width=60 pt]{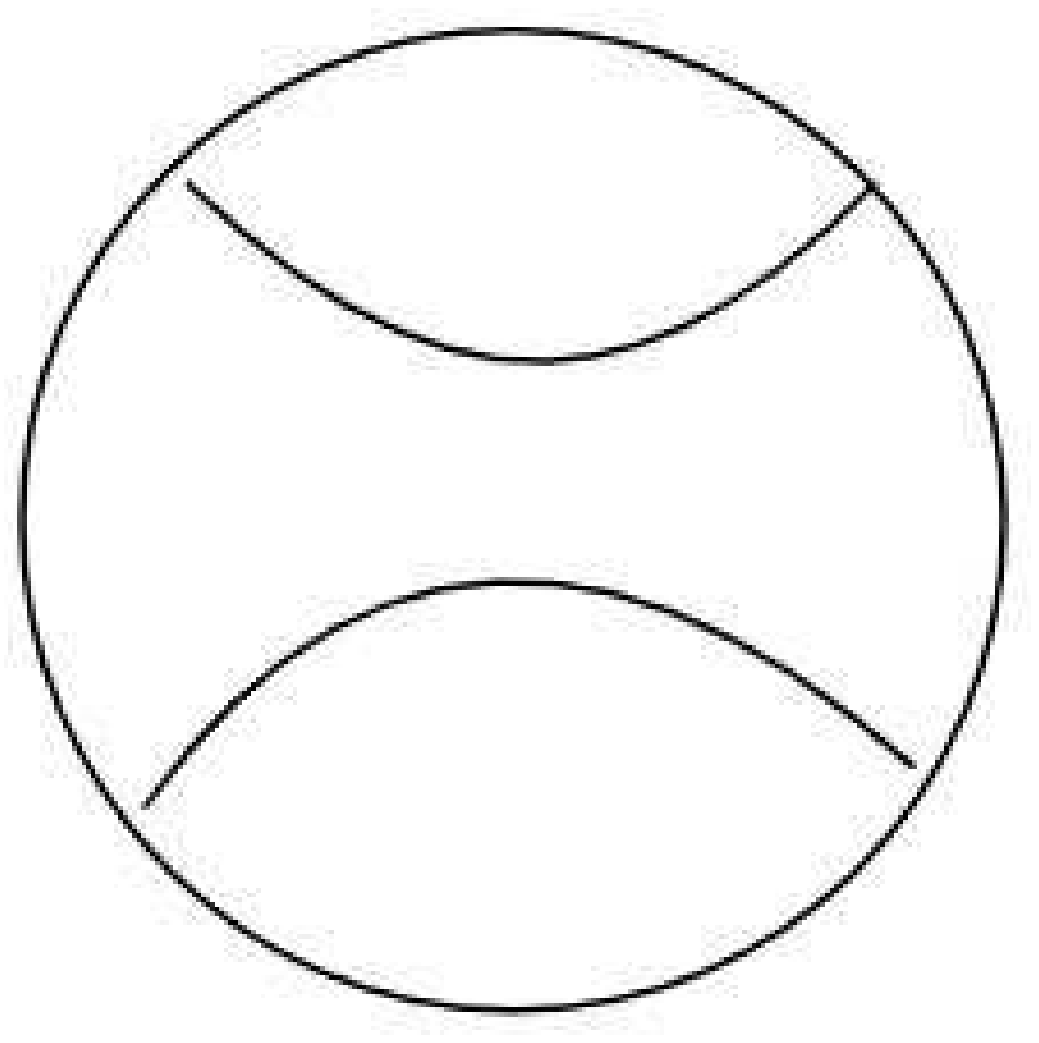}}  
\end{align}
\end{figure}

(2) Framing relation: 
\begin{figure}[!htb] 
\begin{align}
\mathcal{L}\cup \raisebox{-15pt}{
\includegraphics[width=60 pt]{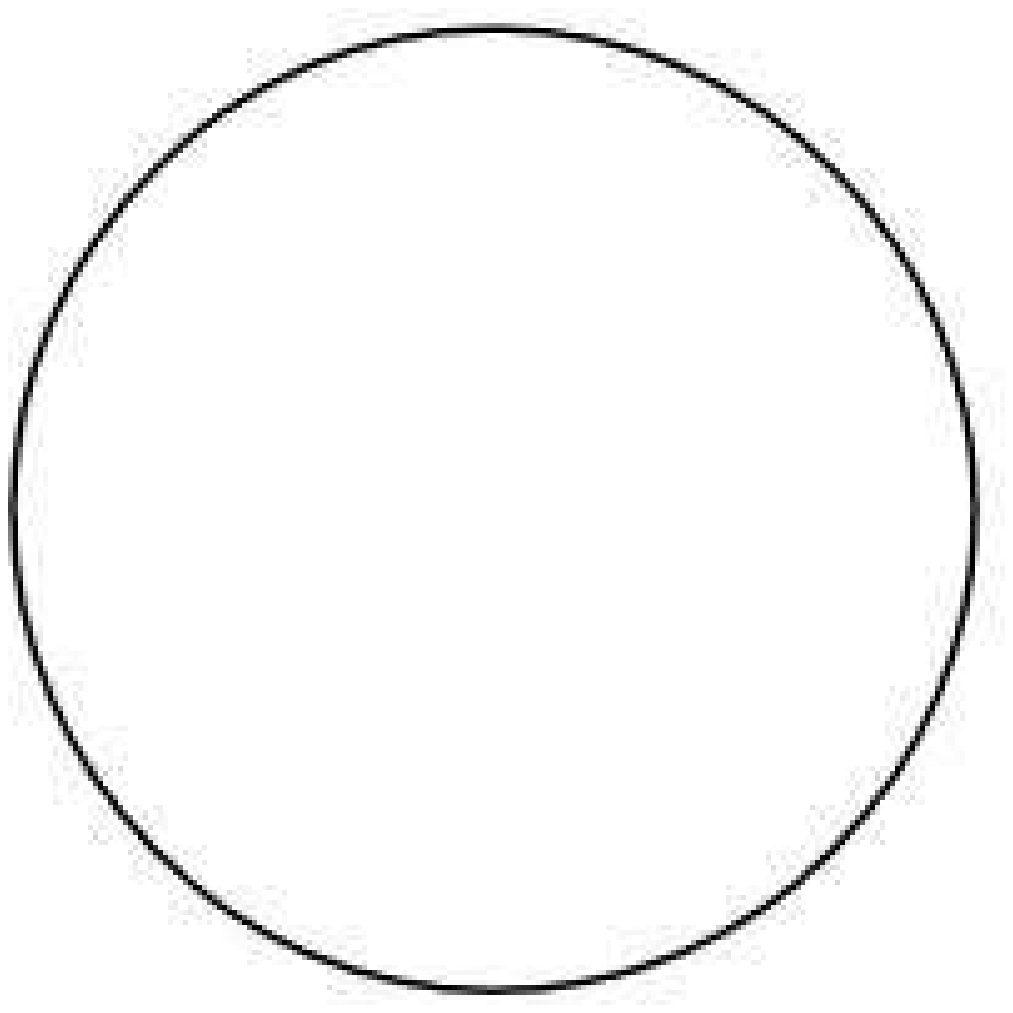}}=(-A^2-A^{-2})\mathcal{L}.
\end{align}
\end{figure}

 The Kauffman bracket $\langle \mathcal{L}\ \rangle$ of a framed link $\mathcal{L}$ in $S^3$ gives a map from $\mathcal{K}(S^3)$ to $\mathbb{Z}[A^{\pm 1}]$. We use the normalization that the bracket of the empty link is 1. 

The Kauffman bracket skein module of the solid torus $S^1\times D^2$ is given by $\mathbb{Z}[A^{\pm 1}][z]$. Usually, we denote this skein module by $\mathcal{B}$. Here $z$ is given by the framed link $S^1\times J$, where $J$ is a small arc lies in the interior of $D^2$, and $z^n$ means $n$-parallel copies of $z$. 

The twist map $t: \mathcal{B}\rightarrow \mathcal{B}$ is a map induced by a full right handed twist on the solid torus. There exists a basis $\{e_i\}_{i\geq 0}$ of $\mathcal{B}$, which are eigenvectors of the twist map $t$ (see e.g \cite{BHMV92}).  $\{e_i\}_{i\geq 0}$ can be defined recursively by 
\begin{align}
    e_0=1, \ e_1=z,  \ e_i=ze_{i-1}-e_{i-2}. 
\end{align}
Moreover, the $e_i$ satisfies
\begin{align}
\langle e_i \rangle&=(-1)^i\frac{A^{2(i+1)}-A^{-2(i+1)}}{A^2-A^{-2}}  \\
t(e_i)&=\mu_i e_i
\end{align}
where $\mu_i=(-1)^iA^{i^2+2i}$ is also called the framing factor. Throughout this paper, we make the  convention  
\begin{align} \label{formula-qconvention}
    q=A^{4}, \ \{n\}=q^{\frac{n}{2}}-q^{-\frac{n}{2}} \ \text{for an integer } \ n. 
\end{align}

\begin{definition}
Given a knot $\mathcal{K}$ with zero framing, the {\em $N$-th colored Jones polynomial} $J_{N}(\mathcal{K};q)$ of $\mathcal{K}$ is defined to be the Kauffman bracket of $\mathcal{K}$ cabled by $(-1)^{N-1}e_{N-1}$, i.e.
\begin{align}
    \bar{J}_{N}(\mathcal{K};q)=(-1)^{N-1}\langle\mathcal{K}(e_{N-1}) \rangle
\end{align}
where the factor of $(-1)^N$ is included such that for the unknot $U$, $\bar{J}_{N}(U;q)=[N]$. Furthermore, the {\em normalized $N$-th colored Jones polynomial} of $\mathcal{K}$ is defined as
\begin{align}
    J_{N}(\mathcal{K};q)=\frac{\langle \mathcal{K}(e_{N-1})\rangle}{\langle e_{N-1}\rangle}.
\end{align}
\end{definition}

By using the Kauffman bracket skein theory \cite{BHMV92,MV94}, Masbaum \cite{Mas03} rederived the cyclotomic expansion formula for the colored Jones polynomial of the twist knot $\mathcal{K}_p$ due to Habiro \cite{Hab08}.    

\begin{proposition}
The normalized $N$-th colored Jones polynomial of the twist knot $\mathcal{K}_p$ is given by 
\begin{align} \label{formula-coloredJonestwist}
J_N(\mathcal{K}_{p};q)=\sum_{k=0}^{N-1}\sum_{l=0}^{k}(-1)^lq^{\frac{k(k+3)}{4}+pl(l+1)}\frac{\{k\}!\{2l+1\}}{\{k+l+1\}!\{k-l\}!}\prod_{i=1}^k(\{N+i\}
\{N-i\}).
\end{align}
\end{proposition}

\subsection{Dilogarithm and Lobachevsky functions}
Let $\log: \mathbb{C}\setminus (-\infty,0]\rightarrow \mathbb{C}$ be the standard logarithm function defined by 
\begin{align}
    \log z=\log |z|+\sqrt{-1}\arg z
\end{align}
with $-\pi <\arg z<\pi$. 

The dilogarithm function $\text{Li}_2: \mathbb{C}\setminus (1,\infty)\rightarrow \mathbb{C}$ is defined by 
\begin{align}
    \text{Li}_2(z)=-\int_0^{z}\frac{\log(1-x)}{x}dx
\end{align}
where the integral is along any path in $\mathbb{C}\setminus (1,\infty)$ connecting $0$ and $z$, which is holomorphic in $\mathbb{C}\setminus [1,\infty)$ and continuous in $\mathbb{C}\setminus (1,\infty)$. 

The dilogarithm function satisfies the following properties 
\begin{align}
    \text{Li}_2\left(\frac{1}{z}\right)=-\text{Li}_2(z)-\frac{\pi^2}{6}-\frac{1}{2}(\log(-z) )^2.
\end{align}
In the unit disk $\{z\in \mathbb{C}| |z|<1\}$,  $\text{Li}_2(z)=\sum_{n=1}^{\infty}\frac{z^n}{n^2}$, and on the unit circle 
\begin{align}
 \{z=e^{2\pi \sqrt{-1}t}|0 \leq t\leq 1\},    
\end{align}
we have
\begin{align}
    \text{Li}_2(e^{2\pi\sqrt{-1} t})=\frac{\pi^2}{6}+\pi^2t(t-1)+2\pi \sqrt{-1}\Lambda(t)
\end{align}
where 
\begin{align} \label{formula-Lambda(t)}
\Lambda(t)=\text{Re}\left(\frac{\text{Li}_2(e^{2\pi \sqrt{-1}t})}{2\pi \sqrt{-1}}\right)=-\int_{0}^{t}\log|2 \sin \pi t|d t 
\end{align}
for $t\in \mathbb{R}$. The function $\Lambda(t)$ is an odd function which has period $1$ and satisfies 
$
\Lambda(1)=\Lambda(\frac{1}{2})=0.
$

Furthermore, we have the following estimation for the function $$\text{Re}\left(\frac{1}{2\pi\sqrt{-1}}\text{Li}_2\left(e^{2\pi\sqrt{-1}(t+X\sqrt{-1})}\right)\right)$$ with $t,X\in \mathbb{R}$.   
\begin{lemma} (see Lemma 2.2 in \cite{OhtYok18}) \label{lemma-Li2}
    Let $t$ be a real number with $0<t<1$, and let
 \begin{align}
    F(t,X)=\left\{ \begin{aligned}
         &0  &  \ (\text{if} \ X\geq 0) \\
         &2\pi \left(t-\frac{1}{2}\right)X & \ (\text{if} \ X<0)
                          \end{aligned}\right.
                      \end{align}
    Then there exists a constant $C>0$ such that 
\begin{align}
    F(t,X)-C<\text{Re}\left(\frac{1}{2\pi\sqrt{-1}}\text{Li}_2\left(e^{2\pi\sqrt{-1}(t+X\sqrt{-1})}\right)\right)<
   F(t,X)+C.
\end{align}
Moreover, 
\begin{align}
\lim_{|X|\rightarrow +\infty}|\text{Re}\left(\frac{1}{2\pi\sqrt{-1}}\text{Li}_2\left(e^{2\pi\sqrt{-1}(t+X\sqrt{-1})}\right)\right)-F(t,X)|=0.    
\end{align}
\end{lemma}

\subsection{Quantum dilogrithm functions}
For a positive integer $N$, we set $\xi_N=e^{\frac{2\pi\sqrt{-1}}{N+\frac{1}{2}}}$. We introduce the holomorphic function $\varphi_N(t)$ for $\{t\in
\mathbb{C}| 0< \text{Re}(t) < 1\}$, by the following integral
\begin{align}
\varphi_N(t)=\int_{-\infty}^{+\infty}\frac{e^{(2t-1)x}dx}{4x \sinh x
\sinh\frac{x}{N+\frac{1}{2}}}.
\end{align}
Noting that this integrand has poles at $n\pi \sqrt{-1} (n\in
\mathbb{Z})$, where, to avoid the poles at $0$, we choose the
following contour of the integral
\begin{align}
\gamma=(-\infty,-1]\cup \{z\in \mathbb{C}||z|=1, \text{Im} z\geq 0\}
\cup [1,\infty).
\end{align}

\begin{lemma}  \label{lemma-varphixi}
The function $\varphi_N(t)$ satisfies 
\begin{align}
    (\xi_{N})_n&=\exp \left(\varphi_N\left(\frac{1}{2N+1}\right)-\varphi_N\left(\frac{2n+1}{2N+1}\right)\right)   \  \   \left(0\leq n\leq N\right), \\
    (\xi_{N})_n&=\exp \left(\varphi_N\left(\frac{1}{2N+1}\right)-\varphi_N\left(\frac{2n+1}{2N+1}-1\right)+\log 2\right)   \  \   \left(N< n\leq 2N\right).
\end{align}
\end{lemma}

\begin{lemma} \label{lemma-varphixi2}
    We have the following identities:
\begin{align}
    \varphi_N(t)+\varphi_N(1-t)&=2\pi \sqrt{-1}\left(-\frac{2N+1}{4}\left(t^2-t+\frac{1}{6}\right)+\frac{1}{12(2N+1)}\right),\\ 
    \varphi_N\left(\frac{1}{2N+1}\right)&=\frac{2N+1}{4\pi\sqrt{-1}}\frac{\pi^2}{6}+\frac{1}{2}\log \left(\frac{2N+1}{2}\right)+\frac{\pi \sqrt{-1}}{4}-\frac{\pi \sqrt{-1}}{6(2N+1)},\\
    \varphi_N\left(1-\frac{1}{2N+1}\right)&=\frac{2N+1}{4\pi\sqrt{-1}}\frac{\pi^2}{6}-\frac{1}{2}\log \left(\frac{2N+1}{2}\right)+\frac{\pi \sqrt{-1}}{4}-\frac{\pi \sqrt{-1}}{6(2N+1)}.
\end{align}
\end{lemma}
The function $\varphi_N (t)$ is closely related to the dilogarithm function as follows.
\begin{lemma} \label{lemma-varphixi3}
    (1)For every $t$ with $0<\text{Re}(t)<1$, 
    \begin{align}
        \varphi_N(t)=\frac{N+\frac{1}{2}}{2\pi \sqrt{-1}}\text{Li}_2(e^{2\pi\sqrt{-1}t})
 -\frac{\pi \sqrt{-1}e^{2\pi\sqrt{-1}t}}{6(1-e^{2\pi\sqrt{-1}t})}\frac{1}{2N+1}+O\left(\frac{1}{(N+\frac{1}{2})^3}\right).
    \end{align}
    (2) For every $t$ with $0<\text{Re}(t)<1$, 
    \begin{align}
        \varphi_N'(t)=-\frac{2N+1}{2}\log(1-e^{2\pi\sqrt{-1}t})+O\left(\frac{1}{N+\frac{1}{2}}\right)
    \end{align}
    (3) As $N\rightarrow \infty$, $\frac{1}{N+\frac{1}{2}}\varphi_N(t)$ uniformly converges to $\frac{1}{2\pi\sqrt{-1}}\text{Li}_2(e^{2\pi\sqrt{-1}t})$ and $\frac{1}{N+\frac{1}{2}}\varphi'_N(t)$ uniformly converges to $-\log(1-e^{2\pi\sqrt{-1}t})$ on any compact subset of $\{t\in \mathbb{C}|0<\text{Re}(t)<1\}$. 
\end{lemma}
See the literature, such as \cite{Oht16,CJ17,WongYang20-1} for the proof of Lemma \ref{lemma-varphixi}, \ref{lemma-varphixi2}, \ref{lemma-varphixi3}.

 \subsection{Saddle point method}
We need to use the following version of saddle point method as illustrated in \cite{Oht18}.
\begin{proposition}[\cite{Oht18}, Proposition 3.1]  \label{proposition-saddlemethod}
   Let $A$ be a non-singular symmetric complex $2\times 2$ matrix, and let $\Psi(z_1,z_2)$ and $r(z_1,z_2)$ be holomorphic functions of the forms, 
   \begin{align}
    \Psi(z_1,z_2)&=\mathbf{z}^{T}A\mathbf{z}+r(z_1,z_2), \\\nonumber
    r(z_1,z_2)&=\sum_{i,j,k}b_{ijk}z_iz_jz_k+\sum_{i,j,k,l}c_{ijkl}z_iz_jz_kz_l+\cdots
   \end{align}
   defined in a neighborhood of $\mathbf{0}\in \mathbb{C}^2$. The restriction of the domain 
   \begin{align} \label{formula-domain0}
       \{(z_1,z_2)\in \mathbb{C}^2| \text{Re} \Psi(z_1,z_2)<0\}  
   \end{align}
   to a neighborhood of $\mathbf{0}\in \mathbb{C}^2$ is homotopy equivalent to $S^1$. Let $D$ be an oriented disk embeded in $\mathbb{C}^2$ such that $\partial D$ is included in the domain (\ref{formula-domain0}) whose inclusion is homotopic to a homotopy equivalence to the above $S^1$ in the domain (\ref{formula-domain0}). Then we have the following asymptotic expansion
\begin{align}
    \int_{D}e^{N\psi(z_1,z_2)}dz_1dz_2=\frac{\pi}{N\sqrt{\det(-A)}}\left(1+\sum_{i=1}^d\frac{\lambda_i}{N^i}+O(\frac{1}{N^{d+1}})\right),
\end{align}
   for any $d$, where we choose the sign of $\sqrt{\det{(-A)}}$ as explained in Proposition \cite{Oht16}, and $\lambda_i$'s are constants presented by using coefficients of the expansion $\Psi(z_1,z_2)$, such presentations are obtained by formally expanding the following formula, 
\begin{align}
    1+\sum_{i=1}^{\infty}\frac{\lambda_i}{N^i}=\exp\left(Nr\left(\frac{\partial }{\partial w_1},\frac{\partial }{\partial w_2}\right)\right)\exp\left(-\frac{1}{4N}(w_1,w_2)A^{-1}\binom{w_1}{w_2}\right)|_{w_1=w_2=0}.
\end{align}
\end{proposition}
For a proof of the Proposition \ref{proposition-saddlemethod},  see \cite{Oht16}. 
\begin{remark}[\cite{Oht18}, Remark 3.2]  \label{remark-saddle}
    As mentioned in Remark 3.6 of \cite{Oht16}, we can extend Proposition \ref{proposition-saddlemethod} to the case where $\Psi(z_1,z_2)$ depends on $N$ in such a way that $\Psi(z_1,z_2)$ is of the form 
    \begin{align}
        \Psi(z_1,z_2)=\Psi_0(z_1,z_2)+\Psi_1(z_1,z_2)\frac{1}{N}+R(z_1,z_2)\frac{1}{N^2}. 
    \end{align}
    where $\Psi_i(z_1,z_2)$'s are holomorphic functions independent of $N$, and we assume that $\Psi_0(z_1,z_2)$ satisfies the assumption of the Proposition and $|R(z_1,z_2)|$ is bounded by a constant which is independent of $N$.  
\end{remark}

\section{Computation of the potential function} \label{section-potentialfunction}
This section is devoted to the computation of potential function for the colored 
Jones polynomial $J_{N}(\mathcal{K}_p;q)$ at the root of unity $\xi_N$. 

We introduce the following $q$-Pochhammer symbol
\begin{align}
    (q)_n=\prod_{i=1}^{n}(1-q^i).
\end{align}
Then we have
\begin{align}
\ \{n\}!=(-1)^nq^{\frac{-n(n+1)}{4}}(q)_n.
\end{align}
By formula (\ref{formula-coloredJonestwist}),  we obtain
\begin{align}
J_{N}(\mathcal{K}_p;q)
=&\sum_{k=0}^{N-1}\sum_{l=0}^k(-1)^{k+l}q^{pl(l+1)+\frac{l(l-1)}{2}-Nk+\frac{k(k+1)}{2}+k}\\\nonumber
&\cdot\frac{(1-q^{2l+1})}{(1-q^N)}\frac{(q)_k(q)_{N+k}}{(q)_{k+l+1}(q)_{k-l}(q)_{N-k-1}}. 
\end{align}
Computing at the root of unity $\xi_N$, we get   
\begin{align}  \label{formula-sumkl}
 J_{N}(\mathcal{K}_p;\xi_N)=&\sum_{k=0}^{N-1}\sum_{l=0}^k\frac{(-1)^{k+l+1}\sin \frac{2\pi(2l+1)}{2N+1}}{\sin \frac{\pi }{2N+1}}\\\nonumber
 &\cdot q^{(p+\frac{1}{2})l^2+(p+\frac{1}{2})l+\frac{k^2}{2}+2k+\frac{3}{4}}\frac{(q)_k(q)_{N+k}}{(q)_{k+l+1}(q)_{k-l}(q)_{N-k-1}}|_{q=\xi_N}. 
\end{align}

Now, we study the following term
\begin{align}
   &(-1)^{l-k-1}q^{(p+\frac{1}{2})l^2+(p+\frac{1}{2})l+\frac{k^2}{2}+2k+\frac{3}{4}}\frac{(q)_k(q)_{N+k}}{(q)_{k+l+1}(q)_{k-l}(q)_{N-k-1}}|_{q=\xi_N} .
\end{align}

By using Lemma \ref{lemma-varphixi},  we obtain 
\begin{align}
    &\frac{(\xi_N)_k(\xi_N)_{N+k}}{(\xi_N)_{k+l+1}(\xi_N)_{k-l}(\xi_N)_{N-k-1}}\\\nonumber
    &=\exp\left(\varphi_N\left(\frac{2(k+l+1)+1}{2N+1}\right)+\varphi_N\left(\frac{2(k-l)+1}{2N+1}\right)+\varphi_N\left(1-\frac{2(k+1)}{2N+1}\right)\right.\\\nonumber
    &\left.-\varphi_N\left(\frac{2k+1}{2N+1}\right)-\varphi_N\left(\frac{2k}{2N+1}\right)-\varphi_N\left(\frac{1}{2N+1}\right)+\log 2\right) \\\nonumber
    &=\exp\left(\varphi_N\left(\frac{2k+2l+3}{2N+1}\right)+\varphi_N\left(\frac{2k-2l+1}{2N+1}\right)-\varphi_N\left(\frac{2k+1}{2N+1}\right)\right.\\\nonumber
    &\left.-\varphi_N\left(\frac{2k}{2N+1}\right)-\varphi_N\left(\frac{2k+2}{2N+1}\right)+\frac{(2N+1)\pi\sqrt{-1}}{24 }-\frac{\pi \sqrt{-1}}{4}-\frac{1}{2}\log \frac{2N+1}{2}\right.\\\nonumber
    &\left.+\pi \sqrt{-1}\left(\frac{1}{3(2N+1)}-\frac{2(k+1)^2}{2N+1}+(k+1)-\frac{2N+1}{12}\right)+\log 2\right),
\end{align}
for $0<k+l+1\leq N$, and
\begin{align}
    &\frac{(\xi_N)_k(\xi_N)_{N+k}}{(\xi_N)_{k+l+1}(\xi_N)_{k-l}(\xi_N)_{N-k-1}}\\\nonumber
    &=\exp\left(\varphi_N\left(\frac{2(k+l+1)+1}{2N+1}-1\right)+\varphi_N\left(\frac{2(k-l)+1}{2N+1}\right)\right.\\\nonumber
    &\left.+\varphi_N\left(1-\frac{2(k+1)}{2N+1}\right)-\varphi_N\left(\frac{2k+1}{2N+1}\right)-\varphi_N\left(\frac{2k}{2N+1}\right)-\varphi_N\left(\frac{1}{2N+1}\right)\right)\\\nonumber
    &=\exp\left(\varphi_N\left(\frac{2k+2l+3}{2N+1}-1\right)+\varphi_N\left(\frac{2k-2l+1}{2N+1}\right)-\varphi_N\left(\frac{2k+1}{2N+1}\right)\right.\\\nonumber
    &\left.-\varphi_N\left(\frac{2k}{2N+1}\right)-\varphi_N\left(\frac{2k+2}{2N+1}\right)+\frac{(2N+1)\pi\sqrt{-1}}{24 }-\frac{\pi \sqrt{-1}}{4}-\frac{1}{2}\log \frac{2N+1}{2} \right.\\\nonumber
    &\left.+\pi \sqrt{-1}\left(\frac{1}{3(2N+1)}-\frac{2(k+1)^2}{2N+1}+(k+1)-\frac{2N+1}{12}\right)\right),
\end{align}
 for $N<k+l+1\leq 2N$. 

Therefore, we obtain 
\begin{align}
    &(-1)^{l-k-1}\xi_N^{(p+\frac{1}{2})(l^2+l)+\frac{k^2}{2}+2k+\frac{3}{4}}\frac{(\xi_N)_k(\xi_N)_{N+k}}{(\xi_N)_{k+l+1}(\xi_N)_{k-l}(\xi_N)_{N-k-1}}\\\nonumber
    &=2\exp (N+\frac{1}{2})\left(\pi \sqrt{-1}\left((2p+1)\left(\frac{2l+1}{2N+1}\right)^2+\frac{4(2k+1)}{(2N+1)^2}-\frac{6p+7}{3(2N+1)^2}\right.\right.\\\nonumber&\left.\left.+\frac{2l+1}{2N+1}-\frac{3}{2(2N+1)}\right)+\frac{1}{N+\frac{1}{2}}\varphi_N\left(\frac{2k+2l+3}{2N+1}\right)+\frac{1}{N+\frac{1}{2}}\varphi_N\left(\frac{2k-2l+1}{2N+1}\right)\right.\\\nonumber
    &\left.-\frac{1}{N+\frac{1}{2}}\varphi_N\left(\frac{2k+1}{2N+1}\right)-\frac{1}{N+\frac{1}{2}}\varphi_N\left(\frac{2k}{2N+1}\right)-\frac{1}{N+\frac{1}{2}}\varphi_N\left(\frac{2k+2}{2N+1}\right)\right.\\\nonumber
    &\left.-\frac{\pi\sqrt{-1}}{12 }-\frac{1}{2N+1}\log \frac{2N+1}{2} \right)
\end{align}
for $0<k+l+1\leq N$, and
\begin{align}
    &(-1)^{l-k-1}\xi_N^{(p+\frac{1}{2})(l^2+l)+\frac{k^2}{2}+2k+\frac{3}{4}}\frac{(\xi_N)_k(\xi_N)_{N+k}}{(\xi_N)_{k+l+1}(\xi_N)_{k-l}(\xi_N)_{N-k-1}}\\\nonumber
    &=\exp (N+\frac{1}{2})\left(\pi \sqrt{-1}\left((2p+1)\left(\frac{2l+1}{2N+1}\right)^2+\frac{4(2k+1)}{(2N+1)^2}-\frac{6p+7}{3(2N+1)^2}\right.\right.\\\nonumber&\left.\left.+\frac{2l+1}{2N+1}-\frac{3}{2(2N+1)}\right)+\frac{1}{N+\frac{1}{2}}\varphi_N\left(\frac{2k+2l+3}{2N+1}-1\right)\right.\\\nonumber
    &\left.+\frac{1}{N+\frac{1}{2}}\varphi_N\left(\frac{2k-2l+1}{2N+1}\right)-\frac{1}{N+\frac{1}{2}}\varphi_N\left(\frac{2k+1}{2N+1}\right)-\frac{1}{N+\frac{1}{2}}\varphi_N\left(\frac{2k}{2N+1}\right)\right.\\\nonumber
    &\left.-\frac{1}{N+\frac{1}{2}}\varphi_N\left(\frac{2k+2}{2N+1}\right)-\frac{\pi\sqrt{-1}}{12 }-\frac{1}{2N+1}\log \frac{2N+1}{2} \right)
\end{align}
for $N<k+l+1\leq 2N$. 

Now we set
\begin{align}
    t=\frac{2k+1}{2N+1}, s=\frac{2l+1}{2N+1}
\end{align}
and define the functions $\tilde{V}_N(p,t,s)$ and $\delta(t,s)$ as follows. 

(1) If $0<s<1$, $0<t\pm s<1$, then 
\begin{align*}
    \tilde{V}_N(p,t,s)&=\pi \sqrt{-1}\left((2p+1)s^2+s+\frac{4}{2N+1}t-\frac{6p+7}{3(2N+1)^2}-\frac{3}{2(2N+1)}\right)\\\nonumber
    &+\frac{1}{N+\frac{1}{2}}\varphi_N\left(t+s+\frac{1}{2N+1}\right)+\frac{1}{N+\frac{1}{2}}\varphi_N\left(t-s+\frac{1}{2N+1}\right)\\\nonumber
    &-\frac{1}{N+\frac{1}{2}}\varphi_N\left(t\right)-\frac{1}{N+\frac{1}{2}}\varphi_N\left(t-\frac{1}{2N+1}\right)-\frac{1}{N+\frac{1}{2}}\varphi_N\left(t+\frac{1}{2N+1}\right)\\\nonumber
    &-\frac{\pi\sqrt{-1}}{12 }-\frac{1}{2N+1}\log \frac{2N+1}{2},
\end{align*}
and $\delta(t,s)=2$. 

(2) If $0<t<1$, $0<t-s<1$ and $1<t+s<2$,  then 
\begin{align*}
    \tilde{V}_N(p,t,s)&=\pi \sqrt{-1}\left((2p+1)s^2+s+\frac{4}{2N+1}t-\frac{6p+7}{3(2N+1)^2}-\frac{3}{2(2N+1)}\right)\\\nonumber
    &+\frac{1}{N+\frac{1}{2}}\varphi_N\left(t+s+\frac{1}{2N+1}-1\right)+\frac{1}{N+\frac{1}{2}}\varphi_N\left(t-s+\frac{1}{2N+1}\right)\\\nonumber
    &-\frac{1}{N+\frac{1}{2}}\varphi_N\left(t\right)-\frac{1}{N+\frac{1}{2}}\varphi_N\left(t-\frac{1}{2N+1}\right)-\frac{1}{N+\frac{1}{2}}\varphi_N\left(t+\frac{1}{2N+1}\right)\\\nonumber
    &-\frac{\pi\sqrt{-1}}{12 }-\frac{1}{2N+1}\log \frac{2N+1}{2},
\end{align*}
and $\delta(t,s)=1$. 

Based on the above calculations, we obtain
\begin{align} \label{formula-coloredJonesPotential1}
    J_{N}(\mathcal{K}_p;\xi_N)&=\sum_{k=0}^{N-1}\sum_{l=0}^k\frac{\sin \frac{2\pi (2l+1)}{2N+1}}{\sin\frac{\pi}{2N+1}}\delta\left(\frac{2k+1}{2N+1},\frac{2l+1}{2N+1}\right)e^{(N+\frac{1}{2})\tilde{V}_N\left(\frac{2k+1}{2N+1},\frac{2l+1}{2N+1}\right)}\\\nonumber
    &=\sum_{k=0}^{N-1}\sum_{l=0}^k\frac{\sin \frac{2\pi (2l+1)}{2N+1}}{\sin\frac{\pi}{2N+1}}\delta\left(\frac{2k+1}{2N+1},\frac{2l+1}{2N+1}\right)\\\nonumber
    &\cdot e^{(N+\frac{1}{2})\left(\tilde{V}_N\left(\frac{2k+1}{2N+1},\frac{2l+1}{2N+1}\right)-2\pi\sqrt{-1}\frac{2k}{2N+1}-2(p+2)\pi\sqrt{-1}\frac{2l}{2N+1}\right)}. 
\end{align}
For convenience, we introduce the function $V_N(p,t,s)$ which is determined by the following formula  
\begin{align}
    &\tilde{V}_N(p,t,s)-2\pi\sqrt{-1}(t-\frac{1}{2N+1})-2(p+2)\pi\sqrt{-1}(s-\frac{1}{2N+1})\\\nonumber
    &=V_N(p,t,s)+\pi\sqrt{-1}\frac{4p+9}{2(2N+1)}-\frac{1}{2N+1}\log \frac{2N+1}{2}. 
\end{align}
Note that the functions $\tilde{V}_N(p,t,s)$ and $V_N(p,t,s)$ are defined on the region 
    \begin{align}
    D=\{(t,s)\in \mathbb{R}^2| 0<t<1, 0<s<1,  0< t-s<1\}.
\end{align}

From formula (\ref{formula-coloredJonesPotential1}),  we finally obtain
\begin{proposition}   \label{prop-coloredJonespotential}
The normalized $N$-th colored Jones polynomial of the twist $\mathcal{K}_p$ at the root of unity $\xi_N$ can be computed as
\begin{align} \label{formula-coloredJonespotential2}
     J_N(\mathcal{K}_{p};\xi_N)=\sum_{k=0}^{N-1}\sum_{l=0}^k g_N(k,l)
\end{align}
with
\begin{align}
    g_N(k,l)&=(-1)^pe^{\frac{\pi\sqrt{-1}}{4}}\frac{1}{\sqrt{(N+\frac{1}{2})}\sin\frac{\pi}{2N+1}}\sin \frac{2\pi (2l+1)}{2N+1}\\\nonumber
    &\cdot \delta\left(\frac{2k+1}{2N+1},\frac{2l+1}{2N+1}\right)e^{(N+\frac{1}{2})V_N\left(p,\frac{2k+1}{2N+1},\frac{2l+1}{2N+1}\right)},
\end{align}
where the function $\delta(t,s)$ and the function $V_N(p,t,s)$ are  given by  

(1) If $0<t<1$, $0<t\pm s<1$, then $\delta(t,s)=2$ and
\begin{align}
    V_N(p,t,s)&=\pi \sqrt{-1}\left((2p+1)s^2-(2p+3)s+\left(\frac{4}{2N+1}-2\right)t-\frac{6p+7}{3(2N+1)^2}\right)\\\nonumber
    &+\frac{1}{N+\frac{1}{2}}\varphi_N\left(t+s+\frac{1}{2N+1}\right)+\frac{1}{N+\frac{1}{2}}\varphi_N\left(t-s+\frac{1}{2N+1}\right)\\\nonumber
    &-\frac{1}{N+\frac{1}{2}}\varphi_N\left(t\right)-\frac{1}{N+\frac{1}{2}}\varphi_N\left(t-\frac{1}{2N+1}\right)\\\nonumber
    &-\frac{1}{N+\frac{1}{2}}\varphi_N\left(t+\frac{1}{2N+1}\right)-\frac{\pi\sqrt{-1}}{12 }.
\end{align}
(2)  If $0<t<1$, $0<t-s<1$ and $1<t+s<2$, then $\delta(t,s)=1$ and 
\begin{align}
    V_N(p,t,s)&=\pi \sqrt{-1}\left((2p+1)s^2-(2p+3)s+\left(\frac{4}{2N+1}-2\right)t-\frac{6p+7}{3(2N+1)^2}\right)\\\nonumber
    &+\frac{1}{N+\frac{1}{2}}\varphi_N\left(t+s+\frac{1}{2N+1}-1\right)+\frac{1}{N+\frac{1}{2}}\varphi_N\left(t-s+\frac{1}{2N+1}\right)\\\nonumber
    &-\frac{1}{N+\frac{1}{2}}\varphi_N\left(t\right)-\frac{1}{N+\frac{1}{2}}\varphi_N\left(t-\frac{1}{2N+1}\right)\\\nonumber
    &-\frac{1}{N+\frac{1}{2}}\varphi_N\left(t+\frac{1}{2N+1}\right)-\frac{\pi\sqrt{-1}}{12 }.
\end{align}
\end{proposition}

We introduce the potential function for the twist knot $\mathcal{K}_p$ as follows
\begin{align} \label{formula-potentialfunction00}
    &V(p,t,s)=\lim_{N\rightarrow\infty}V_{N}(p,t,s)=\pi \sqrt{-1}\left((2p+1)s^2-(2p+3)s-2t\right)\\\nonumber
    &+\frac{1}{2\pi\sqrt{-1}}\left(\text{Li}_2(e^{2\pi\sqrt{-1}(t+s)})+\text{Li}_2(e^{2\pi\sqrt{-1}(t-s)})-3\text{Li}_2(e^{2\pi\sqrt{-1}t})+\frac{\pi^2}{6}\right). 
\end{align}

\section{Poisson summation formula} \label{Section-poissonsummation}
In this section, with the help of Poisson summation formula, we write the formula (\ref{formula-coloredJonespotential2}) as a sum of integrals.  First, according to formulas (\ref{formula-coloredJonestwist}) and (\ref{formula-coloredJonespotential2}), we have
\begin{align}
g_{N}(k,l)&=(-1)^lq^{\frac{k(k+3)}{4}+pl(l+1)}\frac{\{k\}!\{2l+1\}}{\{k+l+1\}!\{k-l\}!}\prod_{i=1}^k(\{N+i\}
\{N-i\})\\\nonumber
&=(-1)^lq^{\frac{k(k+3)}{4}+pl(l+1)}\frac{\{2l+1\}}{\{N\}}\frac{\{k\}!\{N+k\}!}{\{k+l+1\}!\{k-l\}!\{N-k-1\}!}. 
\end{align}

By Lemmas \ref{lemma-varphixi}, \ref{lemma-varphixi2}, \ref{lemma-varphixi3} and formula (\ref{formula-Lambda(t)}), we obtain 
\begin{align}
    \log \left|\{n\}!\right|=-(N+\frac{1}{2})\Lambda\left(\frac{2n+1}{2N+1}\right)+O(\log (2N+1))
\end{align}  
for any integer $0<n<2N+1$ and at $q=\xi_N=e^{\frac{2\pi \sqrt{-1}}{N+\frac{1}{2}}}$.
So we have 
\begin{align}
    &\log|g_N(k,l)|\\\nonumber
    &=-(N+\frac{1}{2})\Lambda\left(\frac{2k+1}{2N+1}\right)-(N+\frac{1}{2})\Lambda\left(\frac{2N+2k+1}{2N+1}\right)+(N+\frac{1}{2})\Lambda\left(\frac{2(k+l+1)+1}{2N+1}\right)\\\nonumber
    &+(N+\frac{1}{2})\Lambda\left(\frac{2(k-l)+1}{2N+1}\right)+(N+\frac{1}{2})\Lambda\left(\frac{2(N-k-1)+1}{2N+1}\right)+O(\log (2N+1)) \\\nonumber
    &=-(N+\frac{1}{2})\Lambda\left(\frac{2k+1}{2N+1}\right)-(N+\frac{1}{2})\Lambda\left(\frac{2k}{2N+1}\right)+(N+\frac{1}{2})\Lambda\left(\frac{2k+2l+3}{2N+1}\right)\\\nonumber
    &+(N+\frac{1}{2})\Lambda\left(\frac{2k-2l+1}{2N+1}\right)-(N+\frac{1}{2})\Lambda\left(\frac{2k+2}{2N+1}\right)+O(\log (2N+1)), 
\end{align}
where in the second ``=" we have used the properties of the function $\Lambda(t)$. 
We put
\begin{align}
    v_N(t,s)&=\Lambda\left(t+s+\frac{1}{2N+1}\right)+\Lambda\left(t-s+\frac{1}{2N+1}\right)\\\nonumber
    &-\Lambda\left(t-\frac{1}{2N+1}\right)-\Lambda\left(t\right)-\Lambda\left(t+\frac{1}{2N+1}\right), 
\end{align}
then we obtain 
\begin{align}
    |g_N(k,l)|=e^{(N+\frac{1}{2})v_N\left(\frac{2k+1}{2N+1},\frac{2l+1}{2N+1}\right)+O(\log (2N+1))}. 
\end{align}

We define the function
\begin{align} \label{formula-vts}
    v(t,s)=\Lambda(t+s)+\Lambda(t-s)-3\Lambda\left(t\right).
\end{align}
Note that $\left(\frac{2k+1}{2N+1},\frac{2l+1}{2N+1}\right)\in D=\{(t,s)\in \mathbb{R}^2| 1< t+s< 2, 0< t-s<1, \frac{1}{2}< t<1\}$ for $0\leq k,l\leq N-1$. So we may assume that
the function $v(t,s)$ is defined in the region $D$. Set
\begin{align}
    D_0=\{0.02< t-s, 1.02 < t+s, 0.25 < s< 0.75, 0.5\leq t\leq 0.909\}.
\end{align}
See Figure \ref{figure-D0} for the pictures of the regions $D$ and $D_0$, note that $D_0$ lies in $D$.

\begin{figure}[!htb] 
\begin{align*} 
\raisebox{-15pt}{
\includegraphics[width=230 pt]{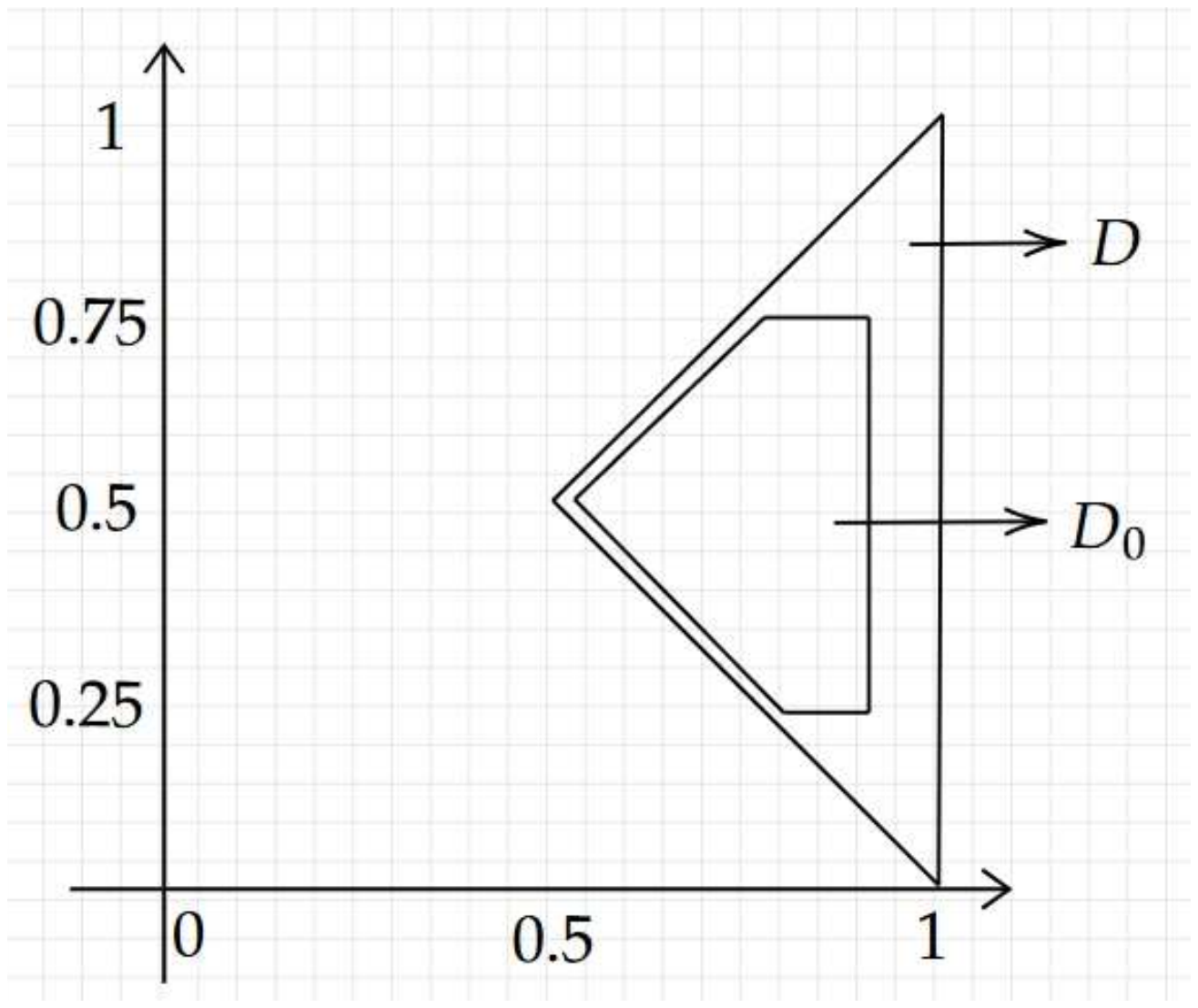}}.
\end{align*}
\caption{The region $D_0\subset D$}
\label{figure-D0} 
\end{figure}

Let $\zeta_{\mathbb{R}}(p)$ be the real part of the critical value $V(p,t_0,s_0)$, see formula (\ref{formula-zetaR(p)}) for its precise definition.  

Then we have
\begin{lemma} \label{lemma-regionD'0}
The following domain
    \begin{align} \label{formula-domain}
        \left\{(t,s)\in D| v(t,s)> \frac{3.56337}{2\pi }\right\}
    \end{align}
    is included in the region $D_0$.
\end{lemma}
\begin{proof}
    See Appendix \ref{appendix-0} for a proof. 
\end{proof}

\begin{remark}
We can take $\delta>0$ small enough (such as $\delta=0.00001$), and set 
\begin{align}
    D_\delta&=\left\{0.02+\delta < t-s, 1.02+\delta < t+s, \right.\\\nonumber &\left.   0.25+\delta < s< 0.75-\delta, 
 0.5+\delta< t< 0.909-\delta\right\},
\end{align}
then the region (\ref{formula-domain}) can also be included in the region $D_{\delta}$. 
\end{remark}

\begin{proposition} \label{prop-gkl}
For $p\geq 6$ and $N$ large enough, if  $(\frac{2k+1}{2N+1},\frac{2l+1}{2N+1})\in D\setminus D_0$,  we have
\begin{align}
    |g_{N}(k,l)|<O\left(e^{(N+\frac{1}{2})\left(\zeta_{\mathbb{R}}(p)-\epsilon\right)}\right)
\end{align}
for some sufficiently small $\epsilon>0$.
\end{proposition}
\begin{proof}
 Since $v_{N}(t,s)$ converges uniformly to $v(t,s)$, for $(\frac{2k+1}{2N+1},\frac{2l+1}{2N+1})\in D\setminus D_0$ and $N$ large enough, we have
\begin{align} \label{formula-gkl}
    |v_N\left(\frac{2k+1}{2N+1},\frac{2l+1}{2N+1}\right)|\leq \frac{3.56337}{2\pi}
    \end{align}
    by Lemma \ref{lemma-regionD'0}.
Then for $p\geq 9$, we have
  $
    |v_N\left(\frac{2k+1}{2N+1},\frac{2l+1}{2N+1}\right)| <\zeta_{\mathbb{R}}(p)-\epsilon
$
    for some small $\epsilon>0$ by Lemma \ref{lemma-volumeestimate}. 
    
In particular, for $p=6,7,8$ we also have 
$|v_N\left(\frac{2k+1}{2N+1},\frac{2l+1}{2N+1}\right)|\leq \frac{3.56337}{2\pi}<\zeta_{\mathbb{R}}(p)-\epsilon$
    for some small $\epsilon>0$, since straightforward computations give $\zeta_{\mathbb{R}}(6)=\frac{3.5889}{2\pi }$, $\zeta_{\mathbb{R}}(7)=\frac{3.6095}{2\pi }$ and $\zeta_{\mathbb{R}}(8)=\frac{3.6227}{2\pi }$. 
\end{proof}

For a sufficiently small $\delta$, we take a smooth bump function $\psi$ on $\mathbb{R}^2$ such that
$\psi(t,s)=1$ on $(t,s)\in \overline{D_{\delta}}$,  $0<\psi(t,s)<1$ on $(t,s)\in D_0\setminus \overline{D_{\delta}}$, $\psi(t,s)=0$ for $(t,s)\notin D_0$.  
Let 
\begin{align}
    h_N(k,l)=\psi\left(\frac{2k+1}{2N+1},\frac{2l+1}{2N+1}\right)g_N(k,l).
\end{align}
Then  by Proposition \ref{prop-gkl},  for $p\geq 6$, we have 
\begin{align}  \label{formula-JN}
    J_N(\mathcal{K}_p;\xi_N)=\sum_{(k,l)\in \mathbb{Z}^2}h_N(k,l)+O\left(e^{(N+\frac{1}{2})\left(\zeta_{\mathbb{R}}(p)-\epsilon\right)}\right).
\end{align}

  We recall the Poisson summation formula \cite{SS03} in 2-dimensional case which states that for any function $h$ in the Schwartz space on $\mathbb{R}^2$, we have 
\begin{align} \label{formula-Poisson}
    \sum_{(k,l)\in \mathbb{Z}^2}h(k,l)=\sum_{(m,n)\in \mathbb{Z}^2}\hat{h}(m,n)
\end{align}
where 
\begin{align}
    \hat{h}(m,n)=\int_{\mathbb{R}^2}h(u,v)e^{-2\pi \sqrt{-1}mu-2\pi \sqrt{-1}nv}dudv.
\end{align}

Note that $h_N$ is $C^{\infty}$-smooth and equals zero outside $D_0$, it is in the Schwartz space on $\mathbb{R}^2$. The Poisson summation formula (\ref{formula-Poisson}) holds for $h_N$. 

By introducing the variables $t=\frac{2k+1}{2N+1}, s=\frac{2l+1}{2N+1}$, 
we compute the Fourier coefficient $\hat{h}_{N}(m,n)$ as follows
\begin{align}
    &\hat{h}_{N}(m,n)=\int_{\mathbb{R}^2}h_N(k,l)e^{-2\pi \sqrt{-1}mk-2\pi \sqrt{-1}nl}dkdl\\\nonumber
    &=(-1)^{m+n}\left(N+\frac{1}{2}\right)^2\\\nonumber
    &\cdot\int_{\mathbb{R}^2}h_N\left((N+\frac{1}{2})t-\frac{1}{2},(N+\frac{1}{2})s-\frac{1}{2}\right)e^{-2\pi \sqrt{-1}\frac{(2N+1)mt}{2}-2\pi \sqrt{-1}\frac{(2N+1)ns}{2}}dtds\\\nonumber
    &=(-1)^{m+n}\left(N+\frac{1}{2}\right)^2\frac{(-1)^pe^{\frac{\pi\sqrt{-1}}{4}}}{\sqrt{N+\frac{1}{2}}\sin\frac{\pi}{2N+1}}\\\nonumber
    &\cdot\int_{D_0}\psi(t,s)\sin(2\pi s)e^{(N+\frac{1}{2})\left(V_N\left(p,t,s\right)-2\pi\sqrt{-1}mt-2\pi\sqrt{-1}ns\right)}dtds,
\end{align}
where
\begin{align}
    &V_N(p,t,s)\\\nonumber
    &=\pi \sqrt{-1}\left((2p+1)s^2-(2p+3)s+\left(\frac{4}{2N+1}-2\right)t-\frac{6p+7}{3(2N+1)^2}-\frac{1}{12}\right)\\\nonumber
    &+\frac{1}{2N+1}\varphi_N\left(t+s+\frac{1}{2N+1}-1\right)+\frac{1}{2N+1}\varphi_N\left(t-s+\frac{1}{2N+1}\right)\\\nonumber
    &-\frac{1}{2N+1}\varphi_N\left(t\right)-\frac{1}{2N+1}\varphi_N\left(t-\frac{1}{2N+1}\right)-\frac{1}{2N+1}\varphi_N\left(t+\frac{1}{2N+1}\right).
\end{align}

Therefore, applying the Poisson summation formula (\ref{formula-Poisson}) to (\ref{formula-JN}), we obtain
\begin{proposition} \label{prop-fouriercoeff}
For $p\geq 6$, the normalized $N$-th colored Jones polynomial of the twist knot $\mathcal{K}_{p}$ is given by 
\begin{align} \label{formula-fouriercoeff}
      J_N(\mathcal{K}_p;\xi_N)=\sum_{(m,n)\in \mathbb{Z}^2}\hat{h}_N(m,n)+O\left(e^{(N+\frac{1}{2})\left(\zeta_{\mathbb{R}}(p)-\epsilon\right)}\right),
\end{align}
where
  \begin{align}
      \hat{h}_N(m,n)=&(-1)^{p+m+n}e^{\frac{\pi\sqrt{-1}}{4}}\frac{\left(N+\frac{1}{2}\right)^{\frac{3}{2}}}{\sin\frac{\pi}{2N+1}} \int_{D_0}\psi(t,s)\sin(2\pi s)e^{(N+\frac{1}{2})\left(V_N\left(p,t,s;m,n\right)\right)}dtds
\end{align}
with
\begin{align}
V_N\left(p,t,s;m,n\right)=V_N\left(p,t,s\right)-2\pi\sqrt{-1}mt-2\pi\sqrt{-1}ns.
\end{align}
\end{proposition}

\begin{lemma}
The following identity holds
    \begin{align} \label{formula-potientalsym}
          V_{N}(p,t,1-s;m,n)
          =V_{N}(p,t,s;m,-n-2)-2\pi \sqrt{-1}(n+1). 
    \end{align}
\end{lemma}
\begin{proof}
    By a straightforward computation, we obtain the following identity
    \begin{align}
        &\pi\sqrt{-1}\left((2p+1)(1-s)^2-(2p+2n+3)(1-s)+\left(\frac{4}{2N+1}-2m-2\right)t-\frac{1}{12}\right)\\\nonumber
        &=\pi\sqrt{-1}\left((2p+1)s^2-(2p+2(-n-2)+3)s)+\left(\frac{4}{2N+1}-2m-2\right)t-\frac{1}{12}\right)\\\nonumber
        &-2\pi \sqrt{-1}(n+1).
    \end{align}
    which immediately gives the formula (\ref{formula-potientalsym}).
\end{proof}

\begin{proposition} \label{prop-hmn}
    For any $m,n\in \mathbb{Z}$, we have
    \begin{align}
        \hat{h}_{N}(m,-n-2)=(-1)^{n}\hat{h}_{N}(m,n).
    \end{align}
\end{proposition}
\begin{proof}
    Since 
    \begin{align}
      &\psi(t,s)\sin(2\pi s)\exp\left((N+\frac{1}{2})V_N\left(p,t,1-s;m,n\right)\right) \\\nonumber
      &=\psi(t,s)\sin(2\pi s)\exp\left((N+\frac{1}{2})\left(V_{N}\left(p,t,s;m,-n-2\right)-2\pi\sqrt{-1}(n+1)\right)\right)\\\nonumber
      &=\psi(t,s)\sin(2\pi s)\exp\left((N+\frac{1}{2})V_N\left(p,t,s;m,-n-2\right)\right)(-1)^{n+1},
    \end{align}
 and   we remark that we can choose the bump function $\psi(t,s)$ satisfying $\psi(t,s)=\psi(t,1-s)$ since the region $D_0$ is symmetric with respect to the line $s=\frac{1}{2}$. 

Then we have
\begin{align}
    &\int_{D_0}\psi(t,s)\sin(2\pi s)\exp\left((N+\frac{1}{2})V_N\left(p,t,s;m,-n-2\right)\right)dtds\\\nonumber
    &=(-1)^{n+1}\int_{D_0}\psi(t,s)\sin(2\pi s)\exp\left((N+\frac{1}{2})V_N\left(p,t,1-s;m,n\right)\right)dtds\\\nonumber
    &=(-1)^{n}\int_{D_0}\psi(t,\tilde{s})\sin(2\pi \tilde{s})\exp\left((N+\frac{1}{2})V_N\left(p,t,\tilde{s};m,n\right)\right)dtd\tilde{s},
\end{align}
where in the third ``=", we have let $\tilde{s}=1-s$.  It follows that 
\begin{align}
        \hat{h}_{N}(m,-n-2)=(-1)^{n}\hat{h}_{N}(m,n).
    \end{align}
\end{proof}
\begin{corollary}
    We have  
    \begin{align} \label{formula-bigcancel}
        \hat{h}_{N}(m,-1)=0,
    \end{align}
    and 
    \begin{align}
    \hat{h}_{N}(m,-2)=\hat{h}_{N}(m,0).    
    \end{align}
\end{corollary}

\begin{remark}
The case (\ref{formula-bigcancel}) is called the ``Big Cancellation". The first situation of such a phenomenon of big cancellation happened in quantum invariants was discovered in the volume conjecture of the Turaev-Viro invariants by Chen-Yang \cite{CY18}. 
The hidden reason behind that was found and described as a precise statement of symmetric property of asymptotics of the quantum $6j$-symbol which is on the Poisson summation level by Chen-Murakami which is Conjecture 3 in \cite{CJ17}.
To the best of our knowledge, this is the first time that such a phenomenon of big cancellation on the Poisson summation level (on the case of colored Jones polynomial) is proved.
\end{remark}

\section{Asympototic expansions} \label{Section-Asympoticexpansion}


The goal of this section is to estimate each Fourier coefficient $\hat{h}_N(m,n)$ appearing in Proposition \ref{prop-fouriercoeff}. In Section \ref{subsection-preparation}, we establish some results which will be used in the following subsections.  In Section \ref{subsection-mneq0} 
we estimate the Fourier coefficients $\hat{h}_N(m,n)$ that can be neglected. We remark that this subsection actually is equivalent to the corresponding subsection devoted to the verification of the assumption of the Poisson summation formula in Ohtsuki's original papers such as \cite{Oht16,Oht18}. The final result we obtain in this subsection is the formula (\ref{formula-Poission-after}). In Sections \ref{subsection-m=0np}, we estimate the remaining Fourier coefficients and find that only two terms will contribute to the final form of the asymptotic expansion. Finally, we finish the proof of Theorem \ref{theorem-main} in Section \ref{subsection-final}.

\subsection{Some preparations} \label{subsection-preparation}
The aim of this subsection is to make some preparations that will be used in the following subsections.  
We consider the following potential function for the twist knot $\mathcal{K}_p$
\begin{align}
    &V(p,t,s; m,n)=\pi \sqrt{-1}\left((2p+1)s^2-(2p+3+2n)s-(2m+2)t\right)\\\nonumber
    &+\frac{1}{2\pi\sqrt{-1}}\left(\text{Li}_2(e^{2\pi\sqrt{-1}(t+s)})+\text{Li}_2(e^{2\pi\sqrt{-1}(t-s)})-3\text{Li}_2(e^{2\pi\sqrt{-1}t})+\frac{\pi^2}{6}\right).
\end{align}
Suppose $t=t_R+X\sqrt{-1}$ and $s=s_R+Y\sqrt{-1}$, we introduce the function
\begin{align}
    &f(p,t,X,s,Y;m,n)=\text{Re} V(p,t_R+X\sqrt{-1},s_R+Y\sqrt{-1};m,n),
\end{align}
which will also be denoted by $f(X,Y;m,n)$ for brevity in the following. 

We have
\begin{align} \label{formula-partialfX}
    \frac{\partial f}{\partial X}&=\text{Re}\left(\sqrt{-1}\frac{\partial }{\partial t}V(p,t_R+X\sqrt{-1},s_R+Y\sqrt{-1};m,n)\right)\\\nonumber
    &=-\text{Im}(-(2m+2)\pi \sqrt{-1}+3\log(1-x)-\log(1-xy)-\log(1-x/y))\\\nonumber
    &=-3\text{arg}(1-x)+\text{arg}(1-xy)+\text{arg}(1-x/y)+(2m+2)\pi,
\end{align}
and
\begin{align}
    \frac{\partial f}{\partial Y}&=\text{Re}\left(\sqrt{-1}\frac{\partial }{\partial s}V(p,t_R+X\sqrt{-1},s_R+Y\sqrt{-1};m,n)\right)\\\nonumber
    &=-\text{Im}(-(2p+3+2n)\pi \sqrt{-1}+(4p+2)\pi\sqrt{-1}s\\\nonumber
    &-\log(1-xy)+\log(1-x/y))\\\nonumber
    &=\text{arg}(1-xy)-\text{arg}(1-x/y)+(2p+3+2n)\pi-(4p+2)\pi s,
\end{align}
where we put $x=e^{2\pi\sqrt{-1}(t_R+X\sqrt{-1})}$ and $y=e^{2\pi\sqrt{-1}(s_R+\sqrt{-1}Y)}$.

Since $\frac{dx}{dX}=-2\pi x$, we compute
\begin{align} \label{formula-partialfXX}
    \frac{\partial^2 f}{\partial X^2}&=-\text{Im}\frac{\partial }{\partial X}\left(3\log(1-x)-\log(1-xy)-\log(1-x/y)\right)\\\nonumber
    &=-\text{Im}\left(\left(-\frac{3}{1-x}+\frac{y}{1-xy}+\frac{1/y}{1-x/y}\right)\frac{\partial x}{\partial X}\right)\\\nonumber
    &=2\pi \text{Im}\left(-\frac{3x}{1-x}+\frac{xy}{1-xy}+\frac{x/y}{1-x/y}\right)\\\nonumber
    &=2\pi \text{Im}\left(-\frac{3}{1-x}+\frac{1}{1-xy}+\frac{1}{1-x/y}\right).
\end{align}

Furthermore, we have
\begin{align}
     \frac{\partial^2 f}{\partial X\partial Y}&=2\pi \text{Im}\left(\frac{1}{1-xy}-\frac{1}{1-x/y}\right)
\end{align}
and 
\begin{align}
    \frac{\partial^2 f}{\partial Y^2}&=2\pi \text{Im}\left(\frac{1}{1-xy}+\frac{1}{1-x/y}\right).
\end{align}
Therefore, the Hessian matrix of $f(X,Y;m,n)$ is presented by 

\begin{align}
    2\pi \begin{pmatrix}
     3a+b+c & b-c \\
     b-c & b+c
    \end{pmatrix}
\end{align}

where we put 
\begin{align}
    a=-\text{Im} \frac{1}{1-x}, \ b=\text{Im}\frac{1}{1-xy}, \ c=\text{Im} \frac{1}{1-x/y}. 
\end{align}
More precisely, by direct computations, we obtain
\begin{align}
    3a&=-\frac{3\sin 2\pi t_R}{e^{2\pi X}+e^{-2\pi X}-2\cos(2\pi t_R)}, \\ \nonumber
    b&=\frac{\sin(2\pi (t_R+s_R))}{e^{2\pi(X+Y)}+e^{-2\pi(X+Y)}-2\cos (2\pi(t_R+s_R))}, \\\nonumber
    c&=\frac{\sin(2\pi (t_R-s_R))}{e^{2\pi(X-Y)}+e^{-2\pi(X-Y)}-2\cos (2\pi(t_R-s_R))}.
\end{align}
So if $\frac{1}{2}<t_R<1$, $1<t_R+s_R<\frac{3}{2}$ and $0<t_R-s_R<\frac{1}{2}$, we have that 
\begin{align}
    a>0, \ b>0, c>0,
\end{align}
which implies that the Hessian matrix of $f$ with respect to $X,Y$ is positive definite, i.e. we obtain 
\begin{lemma} \label{lemma-HessXY}
    In the region $D_{H}=\{(t,s)\in \mathbb{R}^2|\frac{1}{2}<t<1, 1<t+s<\frac{3}{2}, 0<t-s<\frac{1}{2}\}$, the Hessian matrix $\begin{pmatrix}
     \frac{\partial^2 f}{\partial X^2} & \frac{\partial^2 f}{\partial X\partial Y} \\
     \frac{\partial^2 f}{\partial X\partial Y} & \frac{\partial^2 f}{\partial Y^2}
    \end{pmatrix}$  is positive definite. 
\end{lemma}

\begin{figure}[!htb] \label{figure1} 
\begin{align*} 
\raisebox{-15pt}{
\includegraphics[width=220 pt]{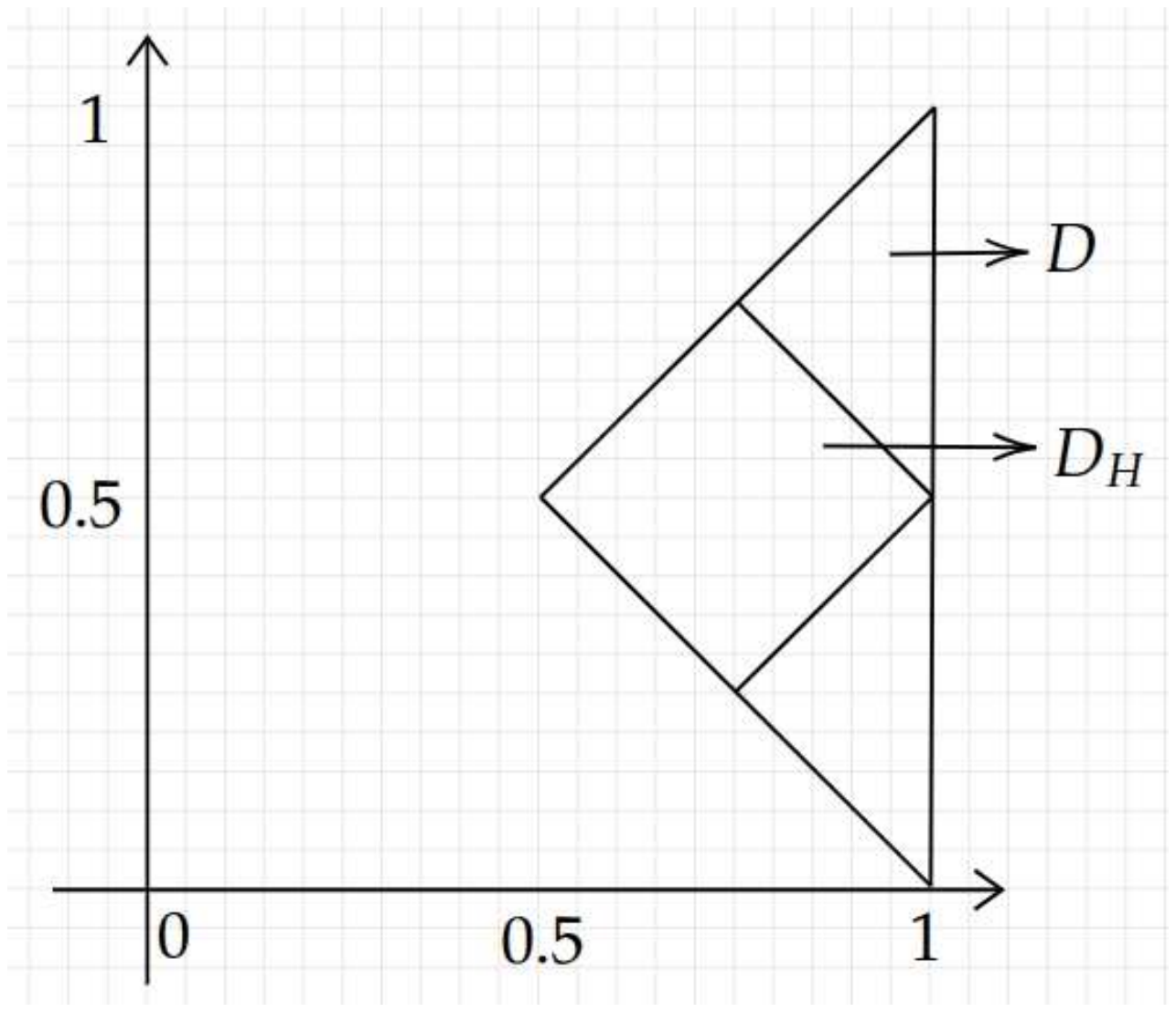}}.
\end{align*}
\caption{The region $D_H\subset D$}
\end{figure}

\begin{lemma} \label{lemma-Vr}
    For any $L>0$, in the region 
    \begin{align}
        \{(t,s)\in  \mathbb{C}^2| (\text{Re}(t),\text{Re}(s))\in D_0, |\text{Im}\  t|<L, |\text{Im} \ s|<L\} 
    \end{align}
    we have
    \begin{align}
        V_N(p,t,s;m,n)
        &=V(p,t,s;m,n)-\frac{1}{2N+1}\left(\log(1-e^{2\pi\sqrt{-1}(t+s)})\right.\\\nonumber
   &\left.+\log(1-e^{2\pi\sqrt{-1}(t-s)})-4\pi\sqrt{-1}t\right)+\frac{w_N(t,s)}{(2N+1)^2},
    \end{align}
    with $|w_N(t,s)|$ bounded from above by a constant independent of $N$. 
\end{lemma}
\begin{proof}
By using Taylor expansion, together with Lemma \ref{lemma-varphixi3}, we have 
\begin{align}
    &\varphi_N\left(t+s-1+\frac{1}{2N+1}\right)\\\nonumber
    &=\varphi_N(t+s-1)+\varphi'_N(t+s-1)\frac{1}{2N+1}\\\nonumber&+\frac{\varphi''_{N}(t+s-1)}{2}\frac{1}{(2N+1)^2}+O\left(\frac{1}{(2N+1)^2}\right)\\\nonumber
    &=\frac{N+\frac{1}{2}}{2\pi\sqrt{-1}}\text{Li}_2(e^{2\pi\sqrt{-1}(t+s)})-\frac{\pi\sqrt{-1}}{6(2N+1)}\frac{e^{2\pi\sqrt{-1}(t+s)}}{1-e^{2\pi\sqrt{-1}(t+s)}}\\\nonumber
    &-\frac{1}{2}\log(1-e^{2\pi\sqrt{-1}(t+s)})+\frac{\pi\sqrt{-1}}{2(2N+1)}\frac{e^{2\pi\sqrt{-1}(t+s)}}{1-e^{2\pi\sqrt{-1}(t+s)}}+O\left(\frac{1}{(2N+1)^2}\right)
    \end{align}

Then, we expand $\varphi_N\left(t-s+\frac{1}{2N+1}\right)$, $\varphi_N\left(t-\frac{1}{2N+1}\right)$ and $\varphi_N\left(t+\frac{1}{2N+1}\right)$ similarly,  we obtain 
\begin{align}
   &V_N(p,t,s;m,n)=V(p,t,s;m,n)\\\nonumber
   &-\frac{1}{2N+1}\left(\log(1-e^{2\pi\sqrt{-1}(t+s)})+\log(1-e^{2\pi\sqrt{-1}(t-s)})-4\pi\sqrt{-1}t\right)\\\nonumber
   &-\frac{\pi\sqrt{-1}}{3(2N+1)^2}\left(-2\frac{e^{2\pi\sqrt{-1}(t+s)}}{1-e^{2\pi\sqrt{-1}(t+s)}}-2\frac{e^{2\pi\sqrt{-1}(t-s)}}{1-e^{2\pi\sqrt{-1}(t-s)}}+3\frac{e^{2\pi\sqrt{-1}t}}{1-e^{2\pi\sqrt{-1}t}}+6p+7\right)\\\nonumber
   &+O\left(\frac{1}{(2N+1)^3}\right). 
\end{align}
Finally, we let 
\begin{align}
    &w_N(t,s)\\\nonumber
    &=-\frac{\pi\sqrt{-1}}{3}\left(-2\frac{e^{2\pi\sqrt{-1}(t+s)}}{1-e^{2\pi\sqrt{-1}(t+s)}}-2\frac{e^{2\pi\sqrt{-1}(t-s)}}{1-e^{2\pi\sqrt{-1}(t-s)}}+3\frac{e^{2\pi\sqrt{-1}t}}{1-e^{2\pi\sqrt{-1}t}}+6p+7\right)\\\nonumber
    &+O\left(\frac{1}{(2N+1)}\right),
\end{align}
then the proof of Lemma \ref{lemma-Vr} is complete. 
\end{proof}
We consider the critical point of $V(p,t,s)$, which is given by the solution of the following equations
\begin{align}  \label{equation-critical1}
    \frac{\partial V(p,t,s)}{\partial t}&=-2\pi\sqrt{-1}+3\log(1-e^{2\pi\sqrt{-1}t})\\\nonumber
    &-\log(1-e^{2\pi\sqrt{-1}(t+s)})-\log(1-e^{2\pi\sqrt{-1}(t-s)})=0,
\end{align}
\begin{align} \label{equation-critical2}
    \frac{\partial V(p,t,s)}{\partial s}&=(4p+2)\pi\sqrt{-1}s-(2p+3)\pi\sqrt{-1}\\\nonumber
    &-\log(1-e^{2\pi\sqrt{-1}(t+s)})+\log(1-e^{2\pi\sqrt{-1}(t-s)})=0. 
\end{align}

\begin{proposition}  \label{prop-critical}
 The critical point equations (\ref{equation-critical1}), (\ref{equation-critical2}) have a unique solution $(t_0,s_0)$ in the region $D_{0\mathbb{C}}=\{(t,s)\in \mathbb{C}^2| (\text{Re}(t),\text{Re}(s))\in D_0\}$.
\end{proposition}
\begin{proof}
See Proposition \ref{prop-critcalequation} for the proof.   
\end{proof}
Now we set 
$\zeta(p)$ to be the critical value of the potential function $V(p,t,s)$, i.e. 
\begin{align}
    \zeta(p)=V(p,t_0,s_0),
\end{align}
and set
\begin{align} \label{formula-zetaR(p)}
    \zeta_\mathbb{R}(p)=\text{Re} \zeta(p)=\text{Re} V(p,t_0,s_0). 
\end{align}

\begin{lemma} \label{lemma-volumeestimate}
When $p\geq 2$, we have the following estimation for $\zeta_{\mathbb{R}}(p)$
\begin{align}
    2\pi \zeta_{\mathbb{R}}(p)\geq v_8-\frac{49\pi^2}{64}\frac{1}{p^2},
\end{align}
where $v_8$ denotes the volume of the ideal regular octahedron, i.e. $v_8\approx 3.66386$.
\end{lemma}
\begin{proof}
    See Appendix \ref{appendix-2} for a proof. 
\end{proof}

We compute the Hessian matrix of the potential function $V(p,t,s)$ as follows. By formulas (\ref{equation-critical1}) and (\ref{equation-critical2}), we obtain 
\begin{align}
    \frac{\partial^2 V}{\partial t^2}=6\pi \sqrt{-1}\frac{-e^{2\pi\sqrt{-1}t}}{1-e^{2\pi\sqrt{-1}t}}+2\pi \sqrt{-1}\frac{e^{2\pi\sqrt{-1}(t+s)}}{1-e^{2\pi\sqrt{-1}(t+s)}}+2\pi \sqrt{-1}\frac{e^{2\pi\sqrt{-1}(t-s)}}{1-e^{2\pi\sqrt{-1}(t-s)}},
\end{align}
\begin{align}
    \frac{\partial^2 V}{\partial s^2}=(4p+2)\pi \sqrt{-1}+2\pi \sqrt{-1}\frac{e^{2\pi\sqrt{-1}(t+s)}}{1-e^{2\pi\sqrt{-1}(t+s)}}+2\pi \sqrt{-1}\frac{e^{2\pi\sqrt{-1}(t-s)}}{1-e^{2\pi\sqrt{-1}(t-s)}},
\end{align}
\begin{align}
    \frac{\partial^2 V}{\partial t\partial s}=2\pi \sqrt{-1}\frac{e^{2\pi\sqrt{-1}(t+s)}}{1-e^{2\pi\sqrt{-1}(t+s)}}-2\pi \sqrt{-1}\frac{e^{2\pi\sqrt{-1}(t-s)}}{1-e^{2\pi\sqrt{-1}(t-s)}}.
\end{align}

Let $x=e^{2\pi\sqrt{-1}t}$ and $y=e^{2\pi \sqrt{-1}s}$, then we have
\begin{align} \label{formula-hessV}
    &\det(Hess(V))\\\nonumber
    &=\frac{\partial^2 V}{\partial t^2}\frac{\partial^2 V}{\partial s^2}-\left(\frac{\partial^2 V}{\partial t\partial s}\right)^2\\\nonumber
    &=(2\pi\sqrt{-1})^2\left(\frac{-3(2p+1)}{\frac{1}{x}-1}+\frac{2p+1}{\frac{1}{xy}-1}+\frac{2p+1}{\frac{y}{x}-1}-\frac{3}{(\frac{1}{x}-1)(\frac{1}{xy}-1)}\right.\\\nonumber
    &\left.-\frac{3}{(\frac{1}{x}-1)(\frac{y}{x}-1)}+\frac{4}{(\frac{1}{xy}-1)(\frac{y}{x}-1)}\right).
\end{align}

For convenience, we let 
\begin{align}
    H(p,x,y)&=\frac{-3(2p+1)}{\frac{1}{x}-1}+\frac{2p+1}{\frac{1}{xy}-1}+\frac{2p+1}{\frac{y}{x}-1}-\frac{3}{(\frac{1}{x}-1)(\frac{1}{xy}-1)}\\\nonumber
    &-\frac{3}{(\frac{1}{x}-1)(\frac{y}{x}-1)}+\frac{4}{(\frac{1}{xy}-1)(\frac{y}{x}-1)}.
\end{align}

\subsection{Fourier coefficients $\hat{h}_N(m,n)$ with $m\neq 0$} \label{subsection-mneq0}
Motivated by Lemma \ref{lemma-Li2}, we introduce the following  function for $(t,s)\in D_0$.

\begin{equation} \label{eq:2}
F(X,Y;m,n)=\left\{ \begin{aligned}
         &0  &  \ (\text{if} \ X+Y\geq 0) \\
         &\left((t+s)-\frac{3}{2}\right)(X+Y) & \ (\text{if} \ X+Y<0)
                          \end{aligned} \right.
                          \end{equation}
\begin{equation*}
 +\left\{ \begin{aligned}
         &0  &  \ (\text{if} \ X-Y\geq 0) \\
         &\left((t-s)-\frac{1}{2}\right)(X-Y) & \ (\text{if} \ X-Y<0)
                          \end{aligned} \right.
                          \end{equation*}
\begin{equation*}
+\left\{ \begin{aligned}
         &0  &  \ (\text{if} \ X\geq 0) \\
         &\left(\frac{3}{2}-3t\right)X & \ (\text{if} \ X<0)
                          \end{aligned} \right.  +\left(p+\frac{3}{2}+n-(2p+1)s\right)Y+(m+1)X
                          \end{equation*}
where we have used $t+s-\frac{3}{2}$ instead of $t+s-\frac{1}{2}$ in the first summation since in our situation $1<t+s<2$.

  Note that $F(X,Y;m,n)$ is a piecewise linear function, we discuss the asymptotic property of this function.
  
Case 1: $Y=0$ and $X\geq 0$, then
\begin{align}
F(X,0;m,n)&=(m+1)X.
\end{align}
If $m\leq -2$, then $ F(X,0;m,n)\rightarrow -\infty$ as $X\rightarrow +\infty$. 
If $m=-1$,  then $ F(X,0;m,n)=0$ for any $X\geq 0$.

Case 2: $Y=0$ and $X\leq 0$, then 
\begin{align}
    F(X,0;m,n)=\left(\frac{1}{2}-t+m\right)X.
\end{align}
If $m\geq 1$, then $\frac{1}{2}-t+m>\frac{1}{2}$ since $\frac{1}{2}<t<1$, we obtain
$ F(X,0;m,n)\rightarrow -\infty$ when $X\rightarrow -\infty$.

We obtain
\begin{lemma} \label{lemma-m-2m1}
For any $(t,s)\in D_0$, 

(i) when $m\leq -2$,  $f(X,0;m,n)$ is a decreasing  function with respect to $X$, and
\begin{align}
    f(X,0;m,n)\rightarrow -\infty \ \text{as } \ X\rightarrow +\infty.
\end{align}

(ii) when $m\geq 1$,  $f(X,0;m,n)$ is an increasing  function with respect to $X$, and
\begin{align}
    f(X,0;m,n)\rightarrow -\infty \ \text{as} \ X\rightarrow -\infty. 
\end{align}
\end{lemma}
\begin{proof}
By formula (\ref{formula-partialfX}), we have 
\begin{align}
\frac{\partial f}{\partial X}=    -3\text{arg}(1-x)+\text{arg}(1-xy)+\text{arg}(1-x/y)+(2m+2)\pi,
\end{align}
where $x=e^{2\pi\sqrt{-1}(t+\sqrt{-1}X)}$, $y=e^{2\pi\sqrt{-1}s}$. 
For any $(t,s)\in D_0$, by direct computation, we obtain $\frac{\partial f}{\partial X}<0    $ for $m\leq -2$ and $
\frac{\partial f}{\partial X}>0    
$ for $m\geq 1$. 

Moreover, using the above case (1), we obtain that, for $m\leq -2$, then $F(X,Y;m,n)\rightarrow -\infty$ as $X\rightarrow +\infty$. By Lemma \ref{lemma-Li2}, we have 
       \begin{align}
           f(X,0,m,n)\rightarrow -\infty.
       \end{align}
      as $X\rightarrow +\infty$.  Hence, we obtain (i).  
    Similarly, for $m\geq 1$,  using the above (2), we obtain (ii). 
\end{proof}

\begin{proposition} \label{prop-mgeq1-leg-2}
    When $m\leq -2$ or $m\geq 1$, then for any $n\in \mathbb{Z}$, there exists $\epsilon>0$ such that 
    \begin{align}
        \hat{h}_{N}(m,n)&=(-1)^{p+m+n}e^{\frac{\pi\sqrt{-1}}{4}}\frac{\left(N+\frac{1}{2}\right)^{\frac{3}{2}}}{\sin\frac{\pi}{2N+1}}\int_{D_0}\psi(t,s)\sin(2\pi s)e^{(N+\frac{1}{2})V_N(p,t,s;m,n)}dtds\\\nonumber
        &=O\left(e^{(N+\frac{1}{2})(\zeta_{\mathbb{R}}(p)-\epsilon)}\right).
    \end{align}
\end{proposition}
\begin{proof}
    We note that $V_N(p,t,s;m,n)$ uniformly converges to $V(p,t,s;m,n)$.
we show the existence of a homotopy $S_{h}$ ($0\leq h\leq h_0$) with $S_0=D_0$ such that
\begin{align}
        S_{h_0}\subset \{(t,s)\in \mathbb{C}^2| \text{Re} V(p,t,s;m,n)<\zeta_{\mathbb{R}}(p)-\epsilon\},  \label{formula-Ddelta}\\ 
        \partial S_{h}\subset \{(t,s)\in \mathbb{C}^2| \text{Re} V(p,t,s;m,n)<\zeta_{\mathbb{R}}(p)-\epsilon\}, \label{formula-partialD}
    \end{align}

    For each fixed $(t,s)\in D_0$,  we move $(X,Y)$ from $(0,0)$ along the flow $(-\frac{\partial f}{\partial X},0)$ to construct such a homotopy. When $m\leq -2$, by Lemma \ref{lemma-m-2m1}, the value of $\text{Re} V(p,t+X\sqrt{-1},s;m,n)$ monotonically decreases and it goes to $-\infty$. Therefore, there exists $X_0<0$ such that $\text{Re} V(p,t+X\sqrt{-1},s;m,n)<\zeta_{\mathbb{R}}(p)-\epsilon$ for some $\epsilon>0$. 
    So we can define the end of homotopy by $S_{h_0}=\{(t+X_0\sqrt{-1},s)|(t,s)\in D_0\}$, 
   and hence (\ref{formula-Ddelta}) is fulfilled by this construction. 
    As for (\ref{formula-partialD}), since $\partial D_0\subset \{(t,s)\in \mathbb{C}^2| \text{Re} V(p,t,s)<\zeta_{\mathbb{R}}(p)-\epsilon\}$ and the value of $\text{Re} V$ monotonically decreases. 
\end{proof}

\begin{proposition} \label{prop-m-1}
    When  $m=-1$, then for any $n\in \mathbb{Z}$, there exists $\epsilon>0$, such that  
    \begin{align}
        \hat{h}_{N}(-1,n)&=(-1)^{p-1+n}e^{\frac{\pi\sqrt{-1}}{4}}\frac{\left(N+\frac{1}{2}\right)^{\frac{3}{2}}}{\sin\frac{\pi}{2N+1}}\int_{D_0}\psi(t,s)\sin(2\pi s)e^{(N+\frac{1}{2})V_N(p,t,s;-1,n)}dtds\\\nonumber
        &=O\left(e^{(N+\frac{1}{2})(\zeta_{\mathbb{R}}(p)-\epsilon)}\right).
    \end{align}
\end{proposition}
\begin{proof}

For any $(t,s)\in D_0$, when $m=-1$, 
we have 
    \begin{align}
         f(X,0;-1,n)\rightarrow 0\ \text{as}  \ X\rightarrow  +\infty,
    \end{align}
by Case 1 and Lemma \ref{lemma-Li2}. Moreover, by formula (\ref{formula-partialfXX}), we have
\begin{align}
\frac{\partial^2 f}{\partial X^2}=2\pi \text{Im}\left(-\frac{3}{1-x}+\frac{1}{1-xy}+\frac{1}{1-x/y}\right),
\end{align}
where $x=e^{2\pi\sqrt{-1}t}$ and $y=e^{2\pi\sqrt{-1}s}$. By direct computation, we obtain $\frac{\partial^2 f}{\partial X^2}>0$. Then, similar to the proof of Proposition \ref{prop-mgeq1-leg-2}, we can construct a homotopy and finish the proof of Proposition \ref{prop-m-1}.
\end{proof}

Now, let us consider the Fourier coefficients with $m=0$.  After the discussion of the asymptotic behavior of the function $F(X,Y;0,n)$ shown in Section \ref{appendix-LemmaUn},  we introduce the following region for $n\in \mathbb{Z}$,
\begin{small}
    \begin{align}
    U'_n=\left\{(t,s)\in \mathbb{R}^2|\frac{p+n+1-t}{2p-1}<s<\frac{p+n+t}{2p-1}, \frac{p+n+t}{2p}<s<\frac{p+2+n-t}{2p}\right\}.
\end{align}
\end{small}
Let $U_n=U'_n\cap D_0$.

\begin{figure}[!htb] \label{figure1} 
\begin{align*} 
\raisebox{-15pt}{
\includegraphics[width=270 pt]{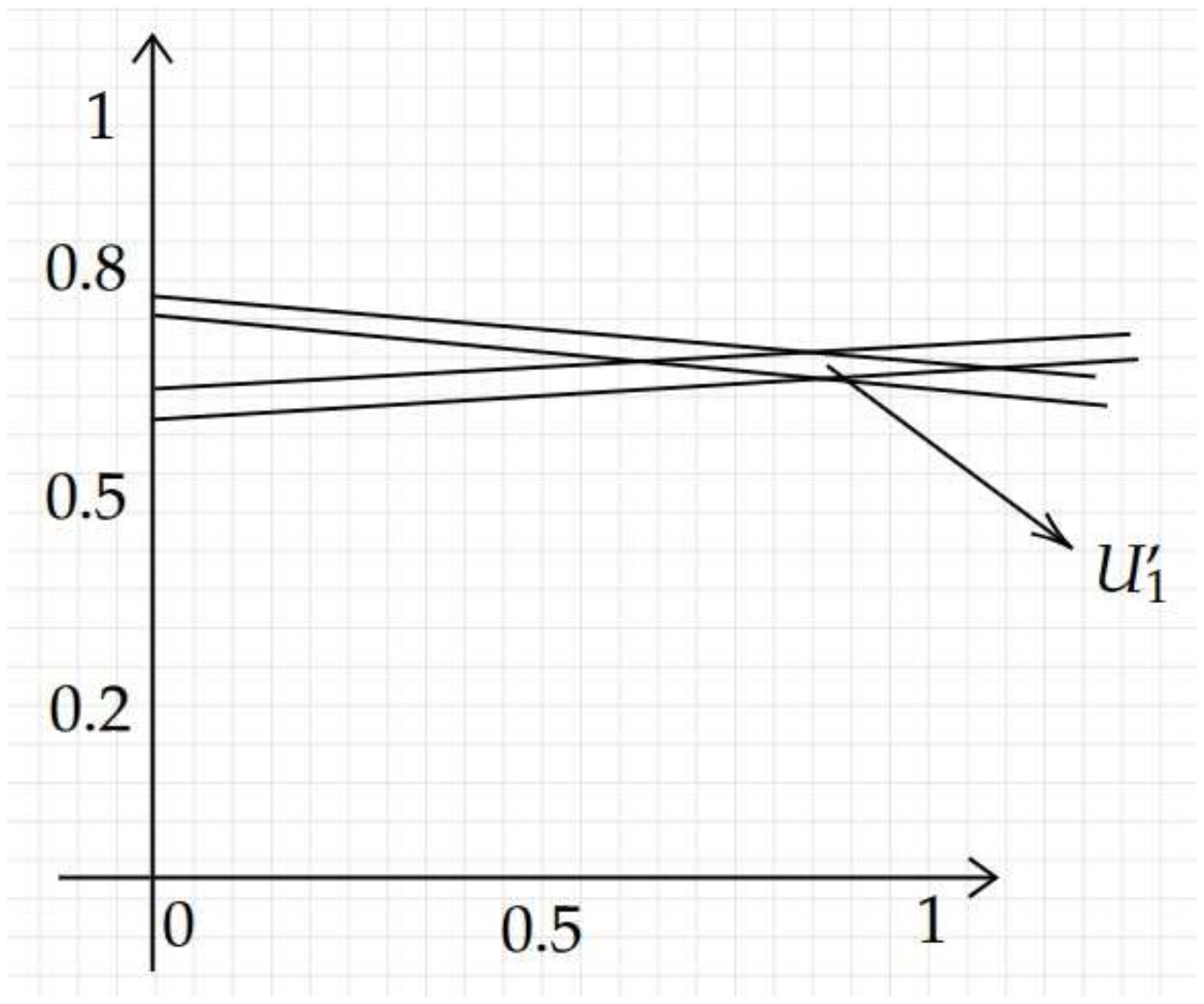}}.
\end{align*}
\caption{The region $U'_1$}
\end{figure}

We have
\begin{lemma} \label{lemma-UnnotUn}
For any $(t,s)\in D_0$, 

(i) if $(t,s)\in U'_n$, then we have 
\begin{align}
    f(X,Y;n)\rightarrow \infty \ \text{as } \ X^2+Y^2\rightarrow +\infty.
\end{align}

(ii) if $(t,s)\notin \overline{U'_n}$, then we have 
\begin{align}
    f(X,Y;n)\rightarrow -\infty \ \text{in some directions of} \  X^2+Y^2\rightarrow +\infty.
\end{align}

(iii) if $(t,s)\in \partial{U'_n}$, then we have 
\begin{align}
    f(X,Y;n)\rightarrow 0 \  \text{in some directions of} \  X^2+Y^2\rightarrow +\infty.
\end{align}
\end{lemma}
\begin{proof}
    See Appendix \ref{appendix-LemmaUn} for a proof. 
\end{proof}

\begin{lemma} \label{lemma-topbottom}
(i) The top point of $U'_n$ is given by $(\frac{3p-2-n}{4p-1},\frac{1}{2}+\frac{5+4n}{2(4p-1)})$, the bottom point of $U'_n$ is given by $(\frac{3p+n}{4p-1},\frac{1}{2}+\frac{3+4n}{2(4p-1)})$,

(ii) For $p\geq 6$, $U_0=U'_0\cap D_0\subset D_{H}$.   
\end{lemma}
\begin{proof}
    Solving the linear equations
\begin{equation} 
\left\{ \begin{aligned}
         s&=\frac{p+n+t}{2p-1} \\
                s&=\frac{p+2+n-t}{2p}
                          \end{aligned} \right.
                          \end{equation}
and 
\begin{equation} \label{eq:2}
\left\{ \begin{aligned}
         s&=\frac{p+n+1-t}{2p-1} \\
                s&=\frac{p+n+t}{2p}
                          \end{aligned} \right.
                          \end{equation}
                          respectively, we obtain (i).

The line $t=0.909$ intersects two lines $s+t=\frac{3}{2}$ and $s=\frac{p+2-t}{2p}$ at the points $(0.909,0.591)$ and $(0.909,\frac{p+1.091}{2p})$ respectively. We have
\begin{align}
    0.591>\frac{p+1.091}{2p}
\end{align}
since $p\geq 6$. Hence $U_0\subset D_{H}$. 
\end{proof}

\begin{figure}[!htb] \label{figure1} 
\begin{align*} 
\raisebox{-15pt}{
\includegraphics[width=270 pt]{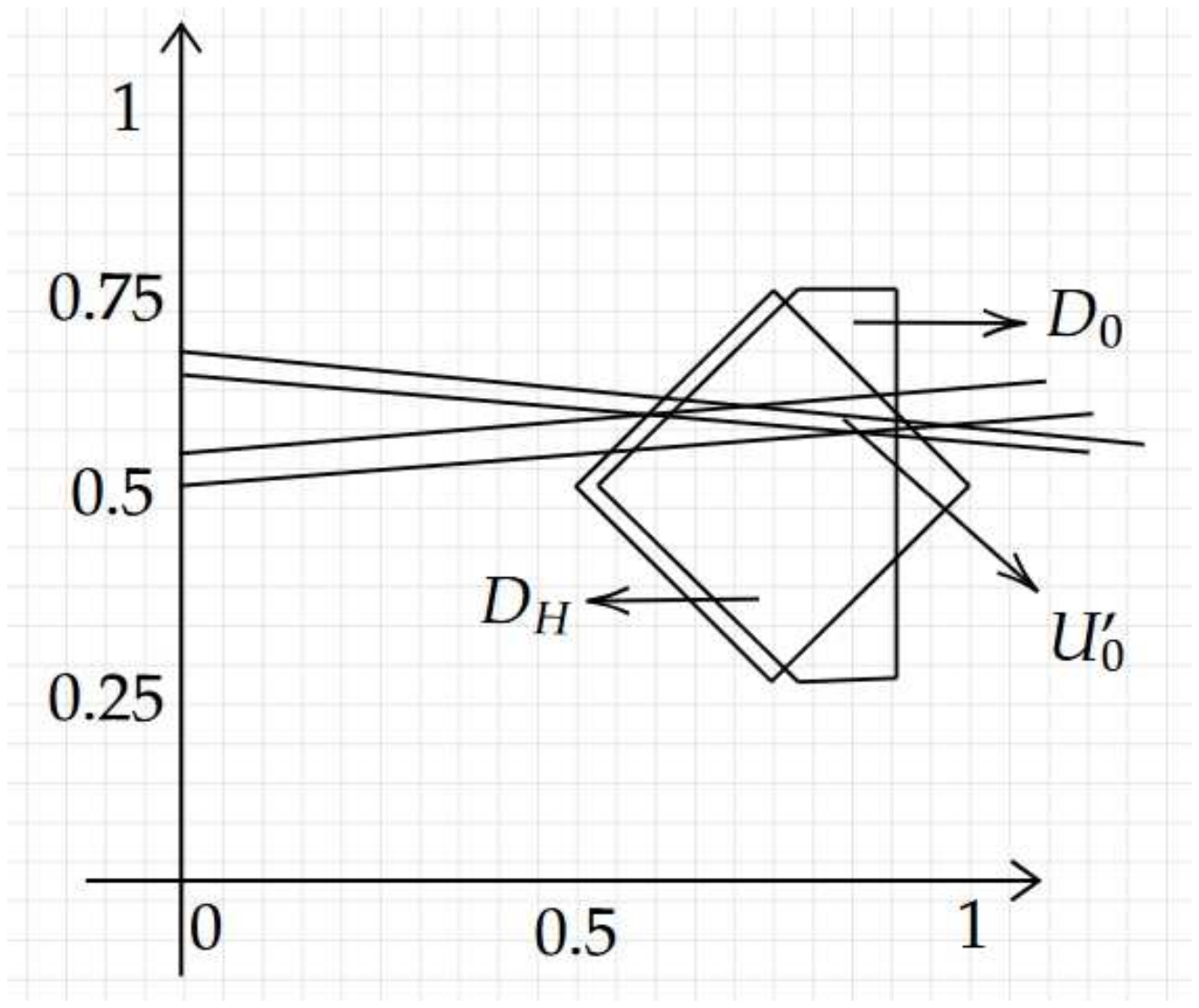}}.
\end{align*}
\caption{The region $U_0=D_0\cap U_0'$}
\end{figure}

\begin{proposition} \label{prop-m-0}
    When  $m=0$, then for any $n\geq p-1$ or $n\leq -(p+1)$, we have 
    \begin{align}
        \hat{h}_{N}(m,n)&=(-1)^{p+m+n}e^{\frac{\pi\sqrt{-1}}{4}}\frac{\left(N+\frac{1}{2}\right)^{\frac{3}{2}}}{\sin\frac{\pi}{2N+1}}\int_{D_0}\psi(t,s)\sin(2\pi s)e^{(N+\frac{1}{2})V_N(p,t,s;m,n)}dtds\\\nonumber
        &=O\left(e^{(N+\frac{1}{2})(\zeta_{\mathbb{R}}(p)-\epsilon)}\right).
    \end{align}
\end{proposition}
\begin{proof}
    Since $(t,s)\in D_0$, we have $0<s<1$. For $n\geq p-1$,  we have 
    \begin{align}
         \frac{1}{2}+\frac{3+4n}{2(4p-1)}\geq 1>s,
    \end{align}
and for $n\leq -(p+1)$, we have 
 \begin{align}
         \frac{1}{2}+\frac{5+4n}{2(4p-1)}\leq 0<s.
    \end{align}
By Lemma \ref{lemma-topbottom}, we know that $\overline{U'_{n}}\cap D_0=\emptyset$ for $n\geq p-1$ and $n\leq -(p+1)$.  Therefore, if  
$(t,s)\in D_0$, then $(t,s)\notin \overline{U'_{n}}$, and  by Lemma \ref{lemma-UnnotUn}, we have
\begin{align}
    f(X,Y;n)\rightarrow -\infty \ \text{in some directions of } \ X^2+Y^2\rightarrow +\infty. 
\end{align}
Now we can finish the proof by constructing a homotopy similar to the proof of Proposition \ref{prop-mgeq1-leg-2}.
\end{proof}

\begin{remark}
    Proposition \ref{prop-mgeq1-leg-2}, Proposition \ref{prop-m-1} and Proposition \ref{prop-m-0} show that the Fourier coefficients $\hat{h}_{N}(m,n)$  with $m,n$ satisfying the conditions in those Propositions can be neglected when we study the asymptotic expansion.  We can also prove Proposition \ref{prop-mgeq1-leg-2} and Proposition \ref{prop-m-1}  by using the method as showed in Ohtsuki's original paper \cite{Oht16,Oht18} (cf. the section of verifying the assumption of the Poisson summation formula for $V$).  
Hence, by formula (\ref{formula-fouriercoeff}), we obtain 
\begin{align} \label{formula-Poission-after}
      J_N(\mathcal{K}_p;\xi_N)=\sum_{n=-p}^{p-2}\hat{h}_{N}(0,n)+O\left(e^{(N+\frac{1}{2})\left(\zeta_{\mathbb{R}}(p)-\epsilon\right)}\right).
\end{align}
So in the following, we only need to investigate the Fourier coefficients $\hat{h}_N(0,n)$ with $-p\leq n\leq p-2$. Note that in Ohtsuki's work \cite{Oht16,Oht17,OhtYok18,Oht18}, after verifying the assumption of the Poisson summation formula,  only one Fourier coefficient (or two Fourier coefficients in \cite{Oht18}) remains to be considered. 
\end{remark}

\subsection{Fourier coefficients $\hat{h}_N(0,n)$ with $0\leq n\leq p-2$} \label{subsection-m=0np}
First, we introduce the quantity
\begin{align} \label{formula-cupperp}
    c_{upper}(p)=\frac{1}{\pi}(v_8-2\pi\zeta_{\mathbb{R}}(p))^{\frac{1}{2}}+\frac{1}{2},
\end{align}
and set
\begin{align}
    c_{0}(p)=\frac{7}{8p}+\frac{1}{2}.
\end{align}
By Lemma \ref{lemma-volumeestimate},  we have $2\pi\zeta_{\mathbb{R}}(p)>v_{8}-\frac{49\pi^2}{64}\frac{1}{p^2}$. Then, by formula (\ref{formula-cupperp}), we obtain 
\begin{align}  \label{formula-cupper<c0}
    c_{upper}(p)<c_0(p).
\end{align} 
For a fix constant $c\in \mathbb{R}$, we define the subset 
\begin{align}
    D_0(c)=\{(t,s)\in D_0| s=c\}.
\end{align}
We have
\begin{proposition} \label{prop-saddleonedim}
  For $n\in \mathbb{Z}$ and for $c_{upper}(p)\leq c\leq 0.75$ or $0.25\leq c\leq 1-c_{upper}(p)$, there exists a constant $C$ independent of $c$, such that
    \begin{align}
        |\int_{D_0(c)}\psi(t,c) \sin(2\pi c)e^{(N+\frac{1}{2})V_N(p,t,c;0,n)}dt|<Ce^{(N+\frac{1}{2})\left(\zeta_{\mathbb{R}}(p)-\epsilon\right)}. 
    \end{align}
\end{proposition}
\begin{proof}
    See Appendix \ref{appendix-onesaddle} for a proof. 
\end{proof}

We introduce the region
\begin{align} \label{formula-D''0}
    D'_0=\{(t,s)\in D_0|1-c_0(p)\leq s\leq c_0(p)\}.
\end{align}

\begin{figure}[!htb] \label{figure1} 
\begin{align*} 
\raisebox{-15pt}{
\includegraphics[width=270 pt]{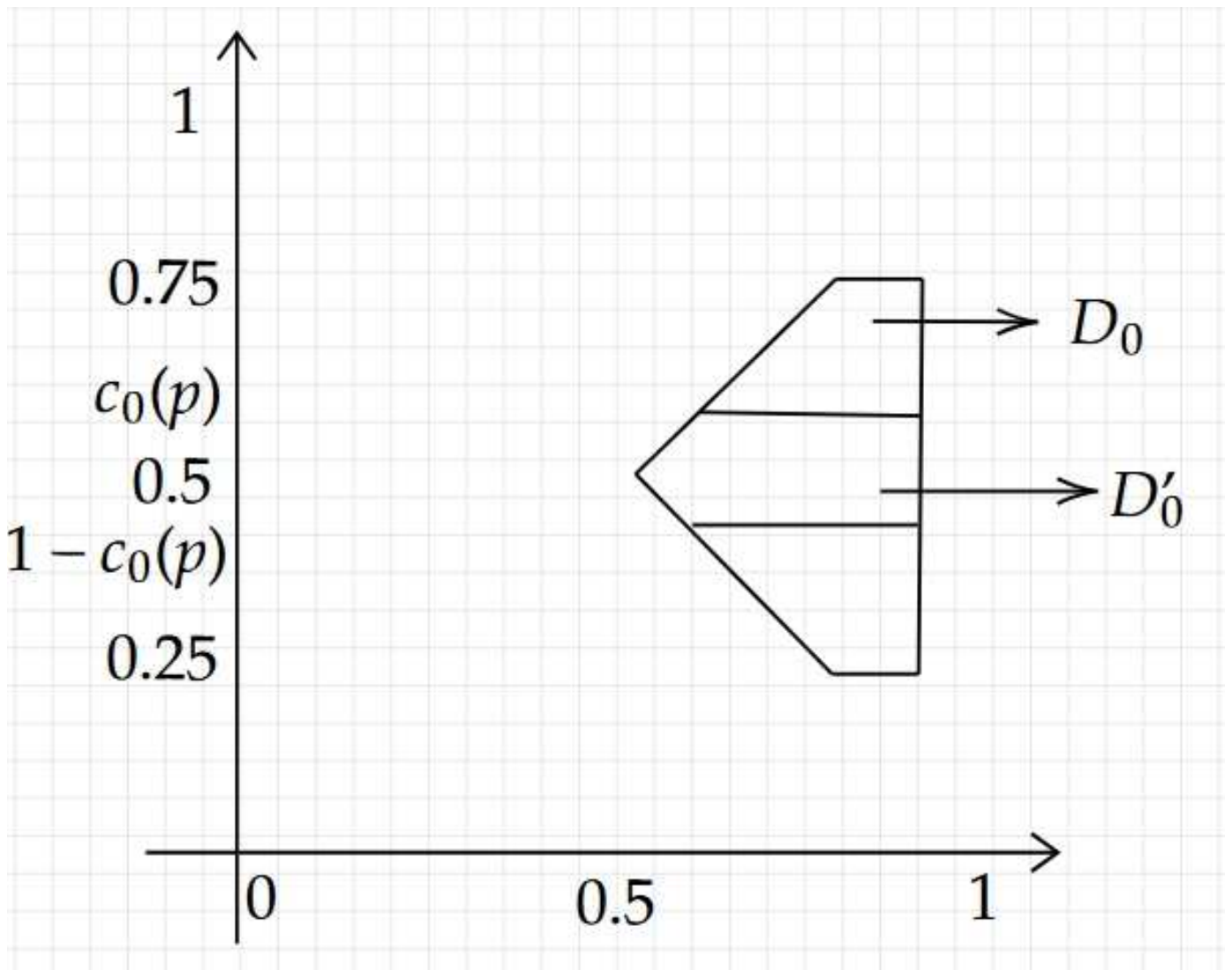}}.
\end{align*}
\caption{The region $D'_0$}
\end{figure}

\subsubsection{Fourier coefficients $\hat{h}_N(0,n)$ with $1\leq n\leq p-2$}
\begin{lemma}  \label{lemma-UnD''}
    For $n\geq 1$, we have
\begin{align}
    \overline{U'_n}\cap D'_0=\emptyset.
\end{align}
\end{lemma}
\begin{proof}
    By Lemma \ref{lemma-topbottom}, the bottom point of $U_n'$ is given by $(\frac{3p+n}{4p-1},\frac{1}{2}+\frac{3+4n}{2(4p-1)})$, clearly 
    \begin{align}
        \frac{1}{2}+\frac{3+4n}{2(4p-1)}>c_0(p)=\frac{1}{2}+\frac{7}{8p},
    \end{align}
    for $n\geq 1$. Hence $\overline{U'_n}\cap D'_0=\emptyset$ for $n\geq 1$.
\end{proof}

\begin{proposition}
 \label{prop-m0np}
    For $1\leq n\leq p-2$, we have 
    \begin{align}
        |\int_{D_0}\psi(t,s)\sin(2\pi s)e^{(N+\frac{1}{2})V_N(p,t,s;0,n)}dtds|=O\left(e^{(N+\frac{1}{2})(\zeta_{\mathbb{R}}(p)-\epsilon)}\right),
    \end{align}
    for some sufficiently small $\epsilon>0$.
\end{proposition}

\begin{proof}
By formula (\ref{formula-cupper<c0}) and  Proposition \ref{prop-saddleonedim}, we have
\begin{align}
&|\int_{D_0}\psi(t,s)\sin(2\pi s)e^{(N+\frac{1}{2})V_N(p,t,s;0,n)}dtds|\\\nonumber
&=|\int_{D'_0}\psi(t,s)\sin(2\pi s)e^{(N+\frac{1}{2})V_N(p,t,s;0,n)}dtds|+O(e^{(N+\frac{1}{2})(\zeta_{\mathbb{R}}(p)-\epsilon)}).     
\end{align}

So we only need to estimate the following integral 
\begin{align}
 |\int_{D'_0}\psi(t,s)\sin(2\pi s)e^{(N+\frac{1}{2})V_N(p,t,s;0,n)}dtds|.
\end{align}
Note that $V_N(p,t,s;0,n)$ uniformly converges to $V(p,t,s;0,n)$ on $D'_{0}$ by Lemma \ref{lemma-varphixi3}.
We show that there is a homotopy $S'_{h}$ $(0\leq h\leq h_0)$ with $S'_0=D'_0$  such that 
\begin{align}
    &S'_{h_0}\subset \{(t,s)\in \mathbb{C}^2|Re V(p;t,s;0,n)<\zeta_\mathbb{R}(p)-\epsilon\}, \label{formula-Delta'} \\
    &|\int_{\partial S'_{h}}\psi(t,s)\sin(2\pi s)e^{(N+\frac{1}{2})V(p,t,s;0,n)}dtds|=O(e^{(N+\frac{1}{2})(\zeta_\mathbb{R}(p)-\epsilon)}), \label{formula-pDelta'}
\end{align}
for some sufficiently small $\epsilon>0$.

In the fiber of the projection $\mathbb{C}^2\rightarrow \mathbb{R}^2$ at $(t,s)\in D'_0$, we consider the flow from $(X,Y)=(0,0)$ determined by the vector field $\left(-\frac{\partial f}{\partial X},-\frac{\partial f}{\partial Y}\right)$, by Lemma \ref{lemma-UnD''}, we obtain that, for $n\geq 1$, $(t,s)\notin \overline{U'_n} $ since $(t,s)\in D'_0$. Then by Lemma \ref{lemma-UnnotUn}, there exists a direction of $X^2+Y^2\rightarrow +\infty$ such that the value of $\text{Re} V(p;t,s;0,n)$ monotonically decreases and goes to $-\infty$.  So we can construct $S'_{h_0}$ as in the proof of Proposition \ref{prop-mgeq1-leg-2} such that formula (\ref{formula-Delta'}) holds.

As for (\ref{formula-pDelta'}), note that the boundary $\partial D'_0$ consists of $D_0(c_{0}(p))$, $D_0(1-c_{0}(p))$ and the partial boundaries of $D_0$, denoted by $D_{0b}$.  Hence $\partial S'_{h}$ consists of three parts denoted by 
\begin{align}
    \partial S'_{h}=A_1\cup A_2\cup B,
\end{align}
where $A_1$ and $A_2$ comes from the flows start at $(t,s)\in D_0(c_{0}(p))$ and  $(t,s)\in D_0(1-c_{0}(p))$ respectively, while $B$ comes from the flows start at $(t,s)\in D_{0b}$.

By its definition, $D_{0b}\subset \partial D_{0}\subset \{(t,s)\in \mathbb{C}^2|\text{Re} V(p,t,s;0,n)<\zeta_\mathbb{R}(p)-\epsilon\}$, and the function $\text{Re} V(p,t,s;0,n)$ decreases under the flow, so we have 
\begin{align} \label{formula-integralB}
    |\int_{B}\psi(t,s)\sin(2\pi s)e^{(N+\frac{1}{2})V(p,t,s;0,n)}dtds|=O(e^{(N+\frac{1}{2})(\zeta_\mathbb{R}(p)-\epsilon)}).
\end{align}

By Proposition \ref{prop-saddleonedim}, the integral on $D_0(c_{0}(p))$ and $D_0(1-c_{0}(p))$ 
is also of order $O(e^{(N+\frac{1}{2})(\zeta_\mathbb{R}(p)-\epsilon}))$. By applying one-dimensional saddle point method to the slices of the region $A_1\cup A_2$ as shown in Appendix \ref{appendix-onesaddle}, we can show that
\begin{align} \label{formula-integralA1A2}
    |\int_{A_1\cup A_2}\psi(t,s)\sin(2\pi s)e^{(N+\frac{1}{2})V(p,t,s;0,n)}dtds|=O(e^{(N+\frac{1}{2})(\zeta_\mathbb{R}(p)-\epsilon)}).
\end{align}
Combining formulas (\ref{formula-integralB}) and (\ref{formula-integralA1A2}) together, we prove (\ref{formula-pDelta'}).
Hence, the required homotopy exists.  
\end{proof}

\subsubsection{Fourier coefficient $\hat{h}_N(0,0)$}

We set $c_0=0.60084$ and let

\begin{align}
\Delta_1&=\{(t,s)\in D_0| t+s\geq \frac{3}{2},s\leq c_0 \} \\   
\Delta_2&=\{(t,s)\in D_0| t-s\geq \frac{1}{2},s\geq 1-c_0 \} 
\end{align}

We define 
\begin{align} \label{formula-D''0}
    D''_0=\{(t,s)\in D_0| 1-c_0\leq s\leq c_0 \} \setminus (\Delta_1\cup \Delta_2) \subset D_H.
    \end{align}

\begin{figure}[!htb] \label{figure1} 
\begin{align*} 
\raisebox{-15pt}{
\includegraphics[width=270 pt]{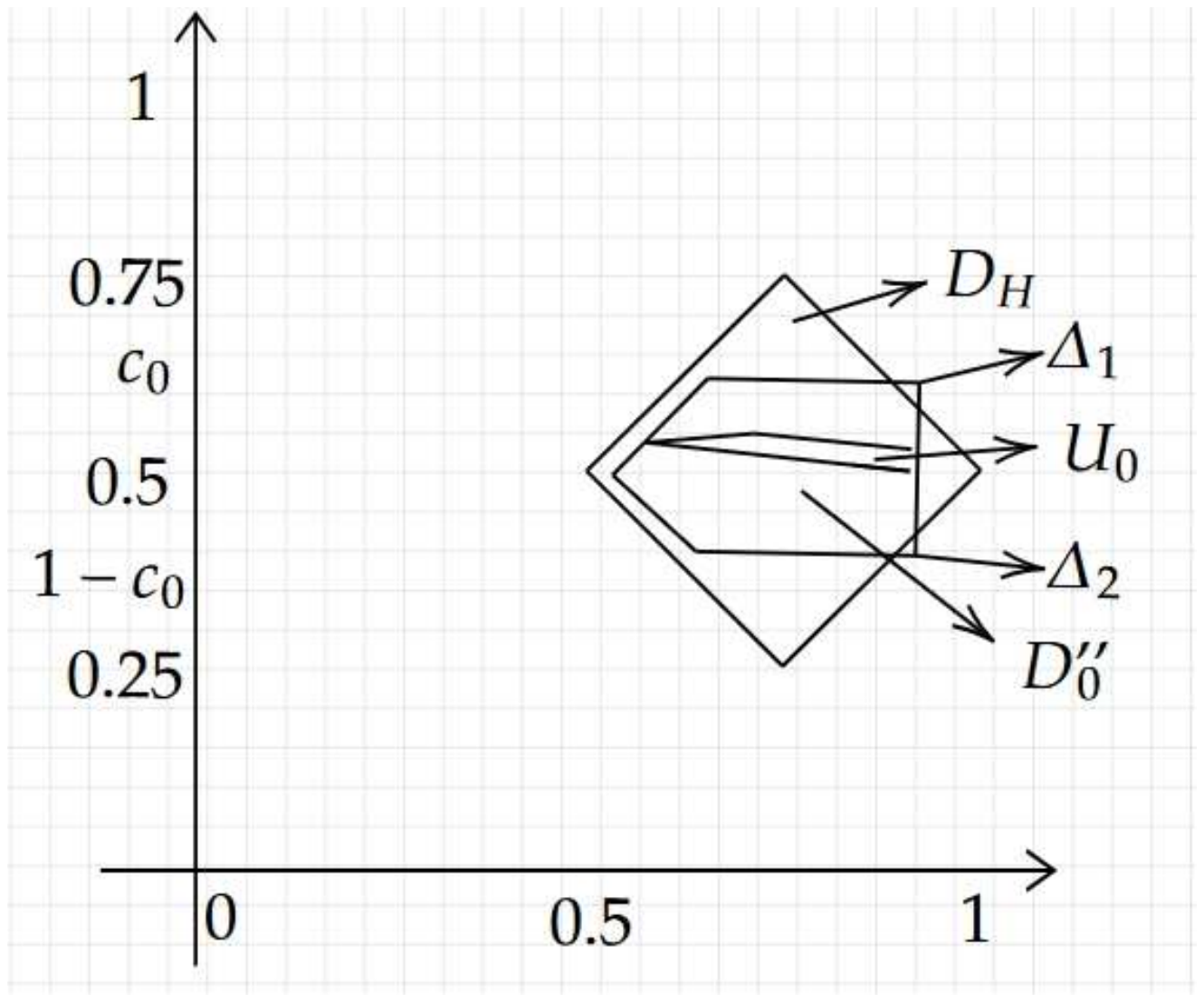}}.
\end{align*}
\caption{The regions $U_0, D''_0, D_H, \Delta_1$ and $\Delta_2$}
\end{figure}

\begin{lemma} \label{lemma-triangleintegral}
\begin{align} \label{formula-triangleintegral}
|\int_{\Delta_i}\psi(t,s)\sin(2\pi s)e^{(N+\frac{1}{2})V_N(p,t,s)}dtds|=O\left(e^{(N+\frac{1}{2})(\zeta_{\mathbb{R}}(p)-\epsilon)}\right)     
\end{align}
for some sufficiently small $\epsilon>0$, where $i=1,2$. 
\end{lemma}
\begin{proof}
We consider the function $\text{Re} V(p,t,s)=v(t,s)$, where $v(t,s)$ is given by the formula (\ref{formula-vts}), it is easy to show that $v(t,s)$ takes the maximal value at the point $(\frac{3}{2}-c_0,c_0)$ on the triangle region $\Delta_{1}$. Hence, on $\Delta_1$,  we have
\begin{align}
    \text{Re}V(p,t,s)\leq \text{Re}V\left(p,\frac{3}{2}-c_0,c_0\right)=\frac{3.56870}{2\pi}<\zeta_{\mathbb{R}}(p)-\epsilon
\end{align}
for $p\geq 6$ and sufficiently small $\epsilon>0$.
So we prove formula (\ref{formula-triangleintegral}) for $\Delta_1$, the proof is similar for $\Delta_2$.  
\end{proof}

\begin{proposition} \label{prop-m0n0} 
For $p\geq 6$,  we have
\begin{align}
    \hat{h}_N(0,0)&=\frac{(-1)^pe^{\frac{1}{4}\pi\sqrt{-1}}(N+\frac{1}{2})^{\frac{3}{2}}}{\sin \frac{\pi}{2N+1}}\int_{D_0}\psi(t,s)\sin(2\pi s)e^{(N+\frac{1}{2})V_N(p,t,s)}dtds\\\nonumber
    &=\frac{(-1)^{p+1}2\pi e^{\frac{1}{4}\pi\sqrt{-1}} (N+\frac{1}{2})^{\frac{1}{2}}}{\sin \frac{\pi}{2N+1}}\omega(p)e^{(N+\frac{1}{2})\zeta(p)}\\\nonumber
&\cdot\left(1+\sum_{i=1}^d\kappa_i(p)\left(\frac{2\pi\sqrt{-1}}{N+\frac{1}{2}}\right)^i+O\left(\frac{1}{(N+\frac{1}{2})^{d+1}}\right)\right),
    \end{align}
    for $d\geq 1$, where $\omega(p)$ and $\kappa_i(p)$ are constants determined by $\mathcal{K}_p$.
\end{proposition}
\begin{proof}
Since $c_{upper}(p)\leq c_{upper}(6)=0.5877<c_0$, by Proposition \ref{prop-saddleonedim}, we have
\begin{align}
    \hat{h}_N(0,0)&=\frac{(-1)^pe^{\frac{1}{4}\pi\sqrt{-1}}(N+\frac{1}{2})^{\frac{3}{2}}}{\sin \frac{\pi}{2N+1}}\int_{D_0}\psi(t,s)\sin(2\pi s)e^{(N+\frac{1}{2})V_N(p,t,s)}dtds\\\nonumber
    &=\frac{(-1)^pe^{\frac{1}{4}\pi\sqrt{-1}}(N+\frac{1}{2})^{\frac{3}{2}}}{\sin \frac{\pi}{2N+1}}\int_{\{(t,s)\in D_0| 1-c_0\leq s\leq c_0 \}}\psi(t,s)\sin(2\pi s)e^{(N+\frac{1}{2})V_N(p,t,s)}dtds \\\nonumber
    &+O\left(e^{(N+\frac{1}{2})\left(\zeta_{\mathbb{R}}(p)-\epsilon\right)}\right),
\end{align}
Furthermore, by Lemma \ref{lemma-triangleintegral}, we have
\begin{align}\label{formula-integralD''}
    \hat{h}_N(0,0)
    &=\frac{(-1)^pe^{\frac{1}{4}\pi\sqrt{-1}}(N+\frac{1}{2})^{\frac{3}{2}}}{\sin \frac{\pi}{2N+1}}\int_{D''_0}\sin(2\pi s)e^{(N+\frac{1}{2})V_N(p,t,s)}dtds \\\nonumber
    &+O\left(e^{(N+\frac{1}{2})\left(\zeta_{\mathbb{R}}(p)-\epsilon\right)}\right),
\end{align}

We will verify the conditions of Proposition \ref{proposition-saddlemethod} for saddle point method in Proposition \ref{propostion-checksaddle}. By Lemma \ref{lemma-Vr} and  Remark \ref{remark-saddle},   
we can apply the Proposition \ref{proposition-saddlemethod} to the above integral (\ref{formula-integralD''}). Let $(t_0,s_0)$ be the critical point of $V(p,t,s)$, we obtain that

\begin{align}
  &\int_{D''_0}\sin(2\pi s)\exp\left((N+\frac{1}{2})V_N(p,t,s)\right)dtds\\\nonumber
        &=\frac{2\pi}{2N+1}\frac{2\alpha(t_0,s_0)}{\sqrt{\det Hess(V)(t_0,s_0)}}e^{(N+\frac{1}{2})\zeta(p)}\\\nonumber
        &\left(1+\sum_{i=1}^d\kappa_i(p)\left(\frac{2\pi\sqrt{-1}}{N+\frac{1}{2}}\right)^i+O\left(\frac{1}{(N+\frac{1}{2})^{d+1}}\right)\right)   
\end{align}
 where the function  
\begin{align}
    \alpha(t,s)=&\psi(t,s)\sin (2\pi s)\\\nonumber
    &\cdot e^{-\frac{1}{2}\left(\log(1-e^{2\pi\sqrt{-1}(t+s)})+\log(1-e^{2\pi\sqrt{-1}(t-s)})-4\pi\sqrt{-1}t\right)}.
\end{align}
and the determinant of the Hessian matrix at $(t_0,s_0)$ is given by the formula (\ref{formula-hessV})
\begin{align}
    &\det Hess(V)(t_0,s_0)=(2\pi\sqrt{-1})^2H(p,x_0,y_0)
\end{align}
where 
\begin{align}
    H(p,x_0,y_0)&=\left(\frac{-3(2p+1)}{\frac{1}{x_0}-1}+\frac{2p+1}{\frac{1}{x_0y_0}-1}+\frac{2p+1}{\frac{1}{x_0/y_0}-1}-\frac{3}{(\frac{1}{x_0}-1)(\frac{1}{x_0y_0}-1)}\right.\\\nonumber
    &\left.-\frac{3}{(\frac{1}{x_0}-1)(\frac{1}{x_0/y_0}-1)}+\frac{4}{(\frac{1}{x_0y_0}-1)(\frac{1}{x_0/y_0}-1)}\right),
\end{align}
with $x_0=e^{2\pi\sqrt{-1}t_0}$ and $y_0=e^{2\pi\sqrt{-1}s_0}$.

Since $(t_0,s_0)$ is the critical point of $V(p,t,s)$, then it satisfies the identity
\begin{align}
    &-\frac{1}{2}\left(\log(1-e^{2\pi\sqrt{-1}(t_0+s_0)})+\log(1-e^{2\pi\sqrt{-1}(t_0-s_0)})-4\pi\sqrt{-1}t_0\right)\\\nonumber
    &=\pi\sqrt{-1}-\frac{3}{2}\log(1-e^{2\pi\sqrt{-1}t_0})+2\pi\sqrt{-1}t_0,
\end{align}
we obtain
\begin{align}
    \alpha(t_0,s_0)=\frac{\sin (2\pi s_0)e^{2\pi\sqrt{-1}t_0}}{(1-e^{2\pi\sqrt{-1}t_0})^{\frac{3}{2}}}.
\end{align}
Therefore, we have
\begin{align}
        \hat{h}_{N}(0,0)&=\frac{(-1)^{p+1}2\pi e^{\frac{1}{4}\pi\sqrt{-1}} (N+\frac{1}{2})^{\frac{1}{2}}}{\sin \frac{\pi}{2N+1}}\omega(p) e^{(N+\frac{1}{2})V(p,t_0,s_0)}\\\nonumber
        &\left(1+\sum_{i=1}^d\kappa_i(p)\left(\frac{2\pi\sqrt{-1}}{N+\frac{1}{2}}\right)^i+O\left(\frac{1}{(N+\frac{1}{2})^{d+1}}\right)\right),
    \end{align}
    where  
\begin{align}
   \omega(p)&=\frac{\sin (2\pi s_0)e^{2\pi\sqrt{-1}t_0}}{(1-e^{2\pi\sqrt{-1}t_0})^{\frac{3}{2}}\sqrt{\det Hess(V)(t_0,s_0)}}\\\nonumber
   &=\frac{(y_0-y_0^{-1})x_0}{-4\pi (1-x_0)^\frac{3}{2}\sqrt{H(p,x_0,y_0)}}.
\end{align}
\end{proof}
\begin{proposition} \label{propostion-checksaddle}
    When we apply Proposition \ref{proposition-saddlemethod} (saddle point method)  to the integral (\ref{formula-integralD''}), the assumptions of Proposition \ref{proposition-saddlemethod} holds. 
\end{proposition}
\begin{proof}
We note that, by Lemma \ref{lemma-varphixi3}, $V_N(p,t,s)$ uniformly converges to the $V(p,t,s)$ on $D''_0$ as $N\rightarrow \infty$. Hence,  we only need to verify the assumptions of the saddle point method for $V(p,t,s)$.  
 We show that there exists a homotopy $S''_h$ ($0\leq h\leq h_0 $) with $S''_0=D''_0$  such that 
\begin{align}
 &(t_0,s_0)\in S''_{h_0},  \label{saddle-1} \\ 
    &S''_{h_0}-\{(t_0,s_0)\}\subset \{(t,s)\in \mathbb{C}^2|\text{Re} V(p,t,s)<\zeta_\mathbb{R}(p)\}, \label{saddle-2} \\
    & \int_{\partial S''_{h}} \sin(2\pi s)e^{(N+\frac{1}{2})V(p,t,s)}dtds=O\left(e^{(N+\frac{1}{2})(\zeta_{\mathbb{R}}(p)-\epsilon)}\right). \label{saddle-3}
\end{align}

   In the fiber of the projection $\mathbb{C}^2\rightarrow \mathbb{R}^2$ at $(t,s)\in D_0''$, we consider the flow from $(X,Y)=(0,0)$ determined by the vector field $(-\frac{\partial f}{\partial X},-\frac{\partial f}{\partial Y})$. By the construction of the region $D''_0$ in (\ref{formula-D''0}), together with Lemma \ref{lemma-UnnotUn} and Lemma \ref{lemma-HessXY}, the region $U_0$ of $(t_0,s_0)$ satisfies the following holds. 
 \begin{itemize}
        \item[(1)] If $(t,s)\in U_0$, then $f$ has a unique minimal point, and the flow goes there.
        \item[(2)] If $(t,s)\in D''_0\setminus U_0$, then the flow goes to infinity. 
    \end{itemize}
We put $\mathbf{g}(t,s)=(g_1(t,s),g_2(t,s))$ to be the minimal point of $(1)$. In particular, $|\mathbf{g}(t,s)|\rightarrow \infty$ as $(t,s)$ goes to $\partial U_0$. Further, for a sufficiently large $R>0$, we stop the flow when $|\mathbf{g}(t,s)|=R$. We construct the revised flow $\hat{\mathbf{g}}(t,s)$, by putting $\hat{\mathbf{g}}(t,s)=\mathbf{g}(t,s)$ for $(t,s)\in U_0$ with $|\mathbf{g}(t,s)|<R$, otherwise, by putting $|\hat{\mathbf{g}}(t,s)|=(R,R)$. 

 We define the ending of the homotopy by
\begin{align}
    S''_{h_0}=\{(t,s)+\hat{\mathbf{g}}(t,s)\sqrt{-1}|(t,s)\in D''_0\}. 
\end{align}
Further, we define the internal part of the homotopy by setting it along the flow from $(t,s)$ determined by the vector field $\left(-\frac{\partial f}{\partial X},-\frac{\partial f}{\partial Y}\right)$. 

We show (\ref{saddle-1}) and (\ref{saddle-2}) as follows. We consider the function
\begin{align}
    h(t,s)=\text{Re} V(t,s,\hat{\mathbf{g}}(t,s)). 
\end{align}
If $(t,s)\notin U_0$, by (2), $-h(t,s)$ is sufficiently large (because we let $R$ be sufficiently large), hence (\ref{saddle-2}) holds in this case. Otherwise, $(t,s)\in U_0$, in this case, $\hat{\mathbf{g}}(t,s)=\mathbf{g}(t,s)$. It is shown from the definition of $\hat{\mathbf{g}}(t,s)$ that 
\begin{align}
    \frac{\partial \text{Re} V}{\partial X}=\frac{\partial \text{Re} V}{\partial Y}=0 \ \text{at} \ (X,Y)=\mathbf{g}(t,s),
\end{align}
which implies 
\begin{align}
    \text{Im}\frac{\partial V}{\partial t}=\text{Im} \frac{\partial V}{\partial s}=0 \ \text{at} \ (t,s)+\mathbf{g}(t,s)\sqrt{-1}.
\end{align}
On the other hand, 
    \begin{align}
        \frac{\partial h}{\partial t}=\text{Re}\frac{\partial V}{\partial t}, \ \frac{\partial h}{\partial s}=\text{Re} \frac{\partial V}{\partial s} \ \text{at} \ (t,s)+\mathbf{g}(t,s)\sqrt{-1}.
    \end{align}
Therefore, when  $(t,s)+\mathbf{g}(t,s)\sqrt{-1}$ is a critical point of $V$, $(t,s)$ is a critical point of $h(t,s)$. Hence by Proposition \ref{prop-critical}, $h(t,s)$ has a unique maximal point at $(t_{0R},s_{0R})$ which is equal to $\zeta_{\mathbb{R}}(p)$. Therefore, (\ref{saddle-1}) and (\ref{saddle-2}) holds.

We show (\ref{saddle-3}) as follows. Note that the boundary of $\partial D''_0$ consists of $D'_0(c_{0})$, $D'_0(1-c_{0})$ and the partial boundaries of $D'_0$, $\Delta_1$ and $\Delta_2$, denoted by $D'_{0b}$.  Hence $\partial S''_{h}$ consists of three parts denoted by 
\begin{align}
    \partial S''_{h}=A_1\cup A_2\cup B,
\end{align}
where $A_1$ and $A_2$ comes from the flows start at $(t,s)\in D'_0(c_{0})$ and  $(t,s)\in D'_0(1-c_{0})$ respectively, while $B$ comes from the flows start at $(t,s)\in D'_{0b}$.

By its definition, $D'_{0b}\subset \{(t,s)\in \mathbb{C}^2|\text{Re} V(p,t,s;0,n)<\zeta_\mathbb{R}(p)-\epsilon\}$, and the function $\text{Re} V(p,t,s)$ decreases under the flow, so we have 
\begin{align} \label{formula-integralB2}
    \int_{B}\sin(2\pi s)e^{(N+\frac{1}{2})V(p,t,s)}dtds=O(e^{(N+\frac{1}{2})(\zeta_\mathbb{R}(p)-\epsilon)}).
\end{align}

By Proposition \ref{prop-saddleonedim}, the integral on $D'_0(c_{0}(p))$ and $D'_0(1-c_{0}(p))$ 
is also of order $O(e^{(N+\frac{1}{2})(\zeta_\mathbb{R}(p)-\epsilon}))$. By applying one-dimensional saddle point method to the slices of the region $A_1\cup A_2$ as shown in Appendix \ref{appendix-onesaddle}, we can show that
\begin{align} \label{formula-integralA1A22}
    \int_{A_1\cup A_2}\sin(2\pi s)e^{(N+\frac{1}{2})V(p,t,s)}dtds=O(e^{(N+\frac{1}{2})(\zeta_\mathbb{R}(p)-\epsilon)}).
\end{align}
Combining formulas (\ref{formula-integralB2}) and (\ref{formula-integralA1A22}) together, we prove (\ref{formula-pDelta'}).

By (\ref{saddle-1}) (\ref{saddle-2}) and (\ref{saddle-3}), the required homotopy exists. Hence the assumptions of Proposition \ref{proposition-saddlemethod} holds when we apply the saddle point method to the integral (\ref{formula-integralD''}).  
\end{proof}

\subsection{Final proof} \label{subsection-final}

Now we can finish the proof of Theorem \ref{theorem-main} as follows.
\begin{proof}
Using formula (\ref{formula-Poission-after}),  Proposition \ref{prop-m0np} and Proposition \ref{prop-m0n0} together, we obtain
\begin{align}
        J_{N}(\mathcal{K}_p;\xi_N)&=2\hat{h}_{N}(0,0)+O(e^{(N+\frac{1}{2})(\zeta_{\mathbb{R}}(p)-\epsilon)})
        \\\nonumber
        &=(-1)^{p+1}\frac{4\pi e^{\frac{1}{4}\pi\sqrt{-1}}(N+\frac{1}{2})^{\frac{1}{2}}}{\sin\frac{\frac{\pi}{2}}{N+\frac{1}{2}}}\omega(p)e^{(N+\frac{1}{2})\zeta(p)}\\\nonumber
&\cdot\left(1+\sum_{i=1}^d\kappa_i(p)\left(\frac{2\pi\sqrt{-1}}{N+\frac{1}{2}}\right)^i+O\left(\frac{1}{(N+\frac{1}{2})^{d+1}}\right)\right),
    \end{align}
    for $d\geq 1$, where $\omega(p)$ and $\kappa_i(p)$ are constants determined by $\mathcal{K}_p$.
\end{proof}

\section{Geometry of the critical point} \label{Section-geometry}

\subsection{Critical point equations vs geometric equation}
\subsubsection{Critical point equations}
Recall that the critical point equations of the potential function 
\begin{align}      
&V(p,t,s)=\pi \sqrt{-1}\left((2 p+1)s^2-(2 p+3)s-2 t\right)\\\nonumber      
&+\frac{1}{2\pi\sqrt{-1}}\left(\text{Li}_2(e^{2\pi\sqrt{-1}(t+s)})+\text{Li}_2(e^{2\pi\sqrt{-1}(t-s)})-3\text{Li}_2(e^{2\pi\sqrt{-1}t})+\frac{\pi^2}{6}\right)
\end{align}
are given by 
the equations (\ref{equation-critical1}) and (\ref{equation-critical2}).

Set $x=e^{2\pi\sqrt{-1}t},\ y=e^{2\pi\sqrt{-1}s}$, then the equations (\ref{equation-critical1}) and (\ref{equation-critical2}) are equivalent to the equations
\begin{align}
\left\{\begin{aligned} \label{equation-critical}   
   & p(4\pi \sqrt{-1}s-2\pi\sqrt{-1})=\log\left(1-xy\right)-\log\left({1-xy^{-1}}\right)-2\pi\sqrt{-1}s+3\pi\sqrt{-1} , \\
   &3\log(1-x)-\log(1-xy)-\log(1-xy^{-1})=2\pi\sqrt{-1}.
    \end{aligned}\right.
\end{align}
We exponentiate the equations (\ref{equation-critical}) and obtain
\begin{align} \label{equation-criticalexp}
 \left\{ \begin{aligned}
     &x=\frac{y^{2p+1}+1}{y^{2p}+y}, \\
      &(x-1)^2=y+y^{-1}-2.
                          \end{aligned} \right.   
\end{align}
 The  equations (\ref{equation-criticalexp}) yields a single equation of $y$,  
\begin{align} \label{equation-singley}
 y^{4p}-y^{4p-1}-4y^{2p}-y+1=0
\end{align}
since $y\neq 1$. Let $y=\tilde{y}^{-2}$, then we obtain 
\begin{align}
\tilde{y}^{8p}-\tilde{y}^{8p-2}-4\tilde{y}^{4p}-\tilde{y}^{2}+1=0    
\end{align}
which can be factorized as 
\begin{align}
 (\tilde{y}^{4p}+\tilde{y}^{4p-1}+\tilde{y}-1)(\tilde{y}^{4p}-\tilde{y}^{4p-1}-\tilde{y}-1)=0.     
\end{align}

Note that 
\begin{align}
(-\tilde{y})^{4p}+(-\tilde{y})^{4p-1}+(-\tilde{y})-1=\tilde{y}^{4p}-\tilde{y}^{4p-1}-\tilde{y}-1,   
\end{align}
Hence, if we want to solve the equation (\ref{equation-singley}), we only need to consider the equation 
\begin{align} \label{equation-singley0}
\tilde{y}^{4p}-\tilde{y}^{4p-1}-\tilde{y}-1=0,
\end{align}
or
\begin{align} \label{equation-tildey}
   \tilde{y}^{4p}=\frac{\tilde{y}(\tilde{y}+1)}{(\tilde{y}-1)}. 
\end{align}

\begin{example} \label{example1}
When $p=3$, equation (\ref{equation-singley0}) has $12$ different solutions: two real solutions $r_0=1.229573607$ and $-r_0^{-1}$, two pure imaginary solutions $\pm \sqrt{-1}$ and $8$ complex solutions: $y_1=1.022851097+0.4998786069\sqrt{-1}$, $y_2=0.5656001655+0.8914919570\sqrt{-1}$, $-y_1^{-1}$,$-y_2^{-1}$, $\overline{y}_1,\overline{y}_2,-\overline{y}_1^{-1},-\overline{y}_2^{-1}$.     
See Figure \ref{figure-RootDistribution} for the distribution of these 12 solutions in the complex plane. Note that all the solutions lie inside the two bounded non-convex lunes determined by the circles are $|z|=1$ and $|z-1|=\sqrt{2}$.  
    \begin{figure}[!htb] 
\begin{align*} 
\raisebox{-15pt}{
\includegraphics[width=250 pt]{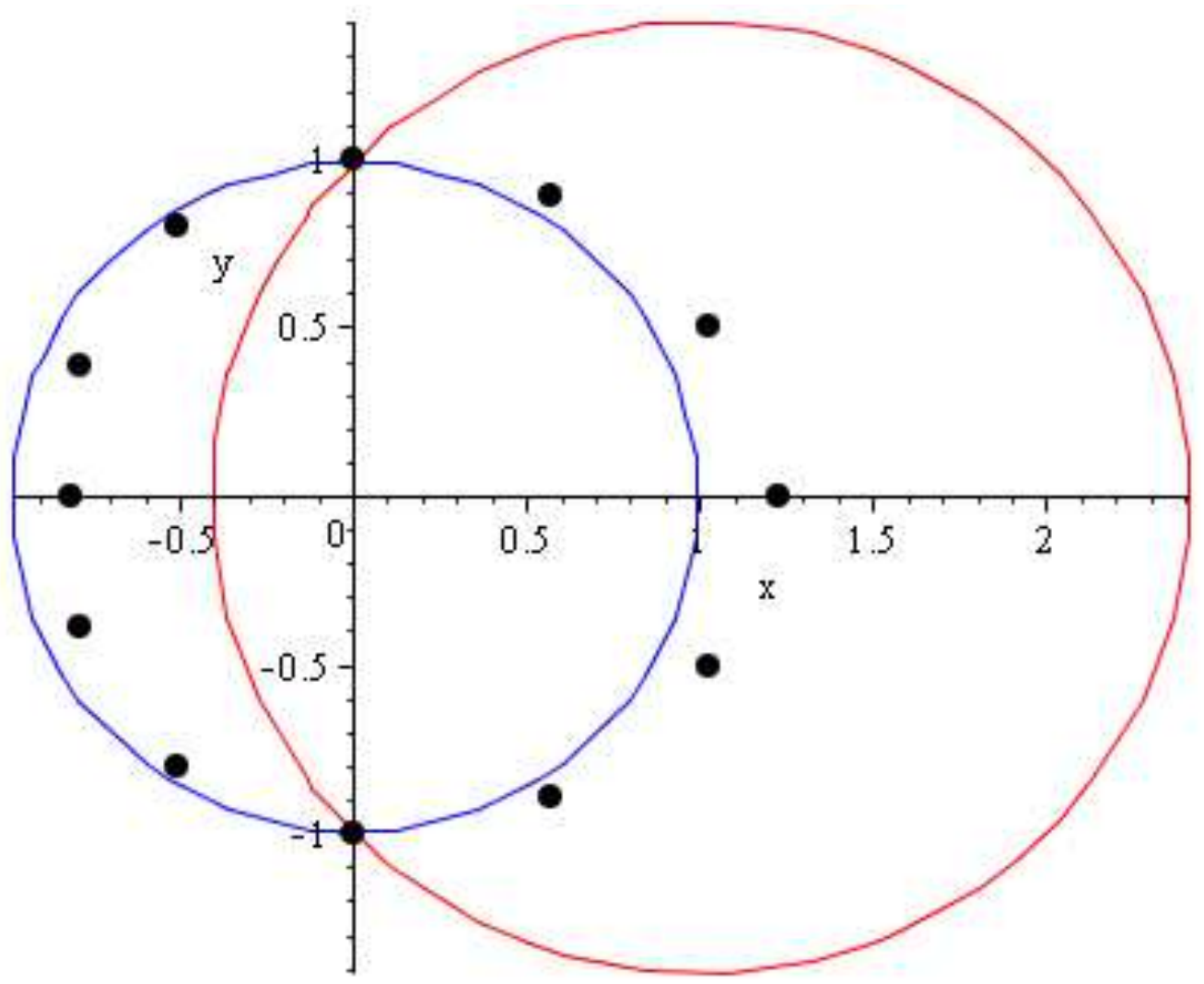}} 
\end{align*}
\caption{Distribution of the solutions of equation (\ref{equation-singley0}) for $p=3$ }\label{figure-RootDistribution} 
\end{figure}

It is easy to check numerically that 
$y_0=(-\bar{y}_2^{-1})^{-2}=-0.4748543622-1.008455997\sqrt{-1}$ and $x_0=1.058181376-1.691279149\sqrt{-1}$ with 

\begin{equation} \label{solution-(t,s)}
\left\{ \begin{aligned}
        t_0&=\frac{\log(x_0)}{2\pi\sqrt{-1}}+1=0.8389804151-0.1099223733\sqrt{-1},\\
    s_0&=\frac{\log(y_0)}{2\pi\sqrt{-1}}+1=0.6799598221-0.01727638566\sqrt{-1}.
                          \end{aligned} \right.
                          \end{equation}
is the unique solution of the critical point equations (\ref{equation-critical}) 
 in the region $D_{0\mathbb{C}}$.
\end{example}

Furthermore, note that
\begin{align}
   \tilde{y}^{4p}-\tilde{y}^{4p-1}-\tilde{y}-1=(\tilde{y}^2+1)w_p(\tilde{y})
\end{align}
where 
\begin{align} 
    w_p(\tilde{y})=\tilde{y}^{4p-2}-\sum_{j=0}^{2p-2}(-1)^j(\tilde{y}^{2j+1}+y^{2j}). 
\end{align}
 So we only need to consider the roots of the polynomial $w_p(\tilde{y})$.  Actually, Hoste and Shanahan \cite{HS01} obtained the following lemma by using the argument principle. 
\begin{lemma}[cf. Proposition 1 in \cite{HS01}] \label{lemma-HS}
For $p>0$, then 

(1) the roots of the polynomial $w_p(\tilde{y})$ lie inside the two bounded non-convex lunes determined by the circles $|x|=1$ and $|x-1|=\sqrt{2}$.  

(2) $w(\tilde{y})$ has $4p-2$ roots given by two real solutions $r_0,-r_0^{-1}$ and $p-1$ distinct solutions in each quadrant. Moreover, if $y_1,...,y_{p-1}$ are the roots in the first quadrant, then 
\begin{align}
    \frac{2k-1}{4p}\pi<\text{arg}(y_{k})<\frac{2k\pi}{4p}\pi. 
\end{align}
In particular, we have
\begin{align}
     -\frac{2p+3}{4p}\pi<\text{arg}(-\bar{y}_{p-1}^{-1})<-\frac{2p+2}{4p}\pi,
\end{align}
\begin{align} \label{formula-ybarp}
   -\frac{2p-2}{4p}\pi<\text{arg}(\bar{y}_{p-1})<-\frac{2p-3}{4p}\pi. 
\end{align}

(3) The roots of $w_p(\tilde{y})$ lie on the curve whose equation, in polar coordinates, is 
\begin{align}
\cos \theta=\left(\frac{r+r^{-1}}{2}\right)\left(\frac{r^{4p-1}-r^{-(4p-1)}}{r^{4p-1}+r^{-(4p-1)}}\right).    
\end{align}
$r$ is an increasing function of $\theta$ on $[-\pi,0]$. If $\tilde{y}_0$ is a root of $w_p(\tilde{y}_0)$, then $|\tilde{y}_0|$ determines $\text{arg}(\tilde{y}_0)$ and vice versa. Finally, as $p\rightarrow \infty$, the roots of $w_p(\tilde{y})$ converge uniformly to the unit circle. 
\end{lemma}

\begin{proposition} \label{prop-critcalequation}
     The critical point equations (\ref{equation-critical}) for twist knot $\mathcal{K}_p$ has a unique solution $(t_0,s_0)\in D_{0\mathbb{C}}$, which is given by
\begin{equation} \label{equation-solution}
\left\{ \begin{aligned}
        y_0&=\bar{y}_{p-1}^2,\\ s_0&=\frac{\log(y_0)}{2\pi\sqrt{-1}}+1,\\ x_0&=1+e^{-\pi\sqrt{-1}s_0}-e^{\pi\sqrt{-1}s_0}, \\ 
        t_0&=\frac{\log(x_0)}{2\pi\sqrt{-1}}+1. 
                          \end{aligned} \right.
                          \end{equation}
\end{proposition}
\begin{proof}
First, we prove that (\ref{equation-solution}) is a solution of  (\ref{equation-critical}). 
By lemma \ref{lemma-HS}, it is easy to see that $(x_0,y_0)$ given by (\ref{equation-solution}) satisfies (\ref{equation-criticalexp}). Taking the logarithm, we obtain  
\begin{align}
\left\{\begin{aligned} \label{equation-criticalargument}   
   & 2p(2\text{Re}(s_0)-1)+2\text{Re}(s_0)-3-\left(\text{arg}(1-x_0y_0)-\text{arg}(1-x_0y_0^{-1})\right)\in 2k_1\pi, \\
   &3\text{arg}(1-x_0)-\text{arg}(1-x_0y_0)-\text{arg}(1-x_0y_0^{-1})-2\pi\in 2k_2\pi.
    \end{aligned}\right.
\end{align}
for some $k_1,k_2\in \mathbb{Z}$. Hence, to prove that (\ref{equation-solution}) is the solution of (\ref{equation-critical}), we only need to show that $k_1=k_2=0$.

By (\ref{formula-ybarp}), we have
\begin{align}
\text{Re}(s_0)\in \left(\frac{1}{2}+\frac{2}{4p},\frac{1}{2}+\frac{3}{4p}\right)    
\end{align}
and by statement (1) in Lemma \ref{lemma-HS}, we have $|\bar{y}_{p-1}|=|e^{\pi\sqrt{-1}s_0}|>1$. 
Through a straightforward computation, we have the following estimates for arguments:
\begin{align}
\text{arg}(1-e^{\pi\sqrt{-1}s_0})\in \left(\left(\frac{1}{2p}-\frac{1}{2}\right)\pi,\left(\frac{3}{8p}-\frac{1}{4}\right)\pi\right),    
\end{align}
\begin{align}
 \text{arg}(1-e^{2\pi\sqrt{-1}s_0})\in \left(\frac{1}{2p}\pi,\frac{3}{2p}\pi\right),    
\end{align}
\begin{align}
\text{arg}(1+e^{-\pi\sqrt{-1}s_0})\in \left(\left(-\frac{1}{4}-\frac{3}{8p}\right)\pi,0\right),    
\end{align}
\begin{align}
    \text{arg}(1-e^{-2\pi\sqrt{-1}s_0})\in \left(-\frac{3}{4p}\pi,0\right).
\end{align}
Since
\begin{align}
1-x_0y_0&=(1-e^{\pi\sqrt{-1}s_0})(1-e^{2\pi\sqrt{-1}s_0}) \\\nonumber
1-x_0y_0^{-1}&=(1+e^{-\pi\sqrt{-1}s_0})(1-e^{-2\pi\sqrt{-1}s_0}),
\end{align}
we have 
\begin{align} \label{formula-x0y0}
\text{arg}(1-x_0y_0)&=\text{arg}(1-e^{\pi\sqrt{-1}s_0})+\text{arg}(1-e^{2\pi\sqrt{-1}s_0})\\\nonumber
&\in \left(\left(\frac{1}{p}-\frac{1}{2}\right)\pi,\left(\frac{15}{8p}-\frac{1}{4}\right)\pi\right),     
\end{align}
and 
\begin{align} \label{formula-x0y0-1}
\text{arg}(1-x_0y_0^{-1})&=\text{arg}(1+e^{-\pi\sqrt{-1}s_0})+\text{arg}(1-e^{-2\pi\sqrt{-1}s_0})\\\nonumber
&\in \left(\left(-\frac{9}{8p}-\frac{1}{4}\right)\pi,0\right).     
\end{align}
Moreover, 
\begin{align}
 \text{arg}(1-x_0)\in \left(\left(\frac{1}{4}+\frac{3}{8p}\right)\pi,\left(\frac{1}{2}+\frac{3}{4p}\right)\pi\right).   
\end{align}

Therefore, 
\begin{align}
 &2p(2\text{Re}(s_0)-1)+2\text{Re}(s_0)-3-\left(\text{arg}(1-x_0y_0)-\text{arg}(1-x_0y_0^{-1})\right)\\\nonumber
 &\in \left(-\frac{3}{4p}\pi,\left(\frac{3}{2}+\frac{1}{2p}\right)\pi\right),   
\end{align}
and 
\begin{align}
    &3\text{arg}(1-x_0)-\text{arg}(1-x_0y_0)-\text{arg}(1-x_0y_0^{-1})-2\pi\\\nonumber
    &\in \left(\left(-1-\frac{3}{4p}\right)\pi,\left(\frac{1}{4}+\frac{19}{8p}\right)\pi\right).
\end{align}
Compared to (\ref{equation-criticalargument}), we obtain $k_1=k_2=0$. Hence we prove that 
(\ref{equation-solution}) gives a solution of (\ref{equation-critical}). 

Next, we should show that any other solutions of the equation (\ref{equation-criticalexp}) given by Lemma \ref{lemma-HS} can not be the solution of (\ref{equation-critical}) by using the analogue argument estimates. Without loss of generality, we consider the solution $y'_0=\bar{y}_{p-2}^{2}$ of (\ref{equation-criticalexp}). By Lemma \ref{lemma-HS}, we have 
\begin{align}
\text{arg}(y'_0)\in \left(-\frac{2p-4}{2p}\pi,-\frac{2p-5}{2p}\pi\right).   
\end{align}
If $y'_0=e^{2\pi\sqrt{-1}s'_0}$ can be a solution of (\ref{equation-critical}) with $s'_0$ lying in $D_{0\mathbb{C}}$, then
\begin{align}
s'_0=\frac{\log(y'_0)}{2\pi\sqrt{-1}}+1   
\end{align}
since $\text{Re}(s'_0)\in \left(\left(\frac{1}{2}+\frac{4}{4p}\right)\pi,\left(\frac{1}{2}+\frac{5}{4p}\right)\pi\right)$. But in this case, we have
\begin{align}
   2p(2\text{Re}(s'_0)-1)+2\text{Re}(s'_0)-3\in \left(\left(2+\frac{4}{2p}\right)\pi,\left(3+\frac{5}{2p}\right)\pi\right), 
\end{align}
and for any $x'_0$
\begin{align}
 -\left(\text{arg}(1-x'_0y'_0)-\text{arg}(1-x'_0(y'_0)^{-1})\right)\in (-2\pi,2\pi).   
\end{align}
Hence
    \begin{align}
 2p(2\text{Re}(s'_0)-1)+2\text{Re}(s'_0)-3-\left(\text{arg}(1-x'_0y'_0)-\text{arg}(1-x'_0(y'_0)^{-1})\right)>\frac{4}{2p}.   
\end{align}
It implies that $y'_0, s'_0$ with any $x'_0$ can not be a solution of (\ref{equation-critical}). 

\end{proof}

\subsubsection{Geometric equation}
Let $W$ be the complement of Whitehead link shown in Figure \ref{figure-WL}.

\begin{figure}[!htb] 
\begin{align*} 
\raisebox{-15pt}{
\includegraphics[width=140 pt]{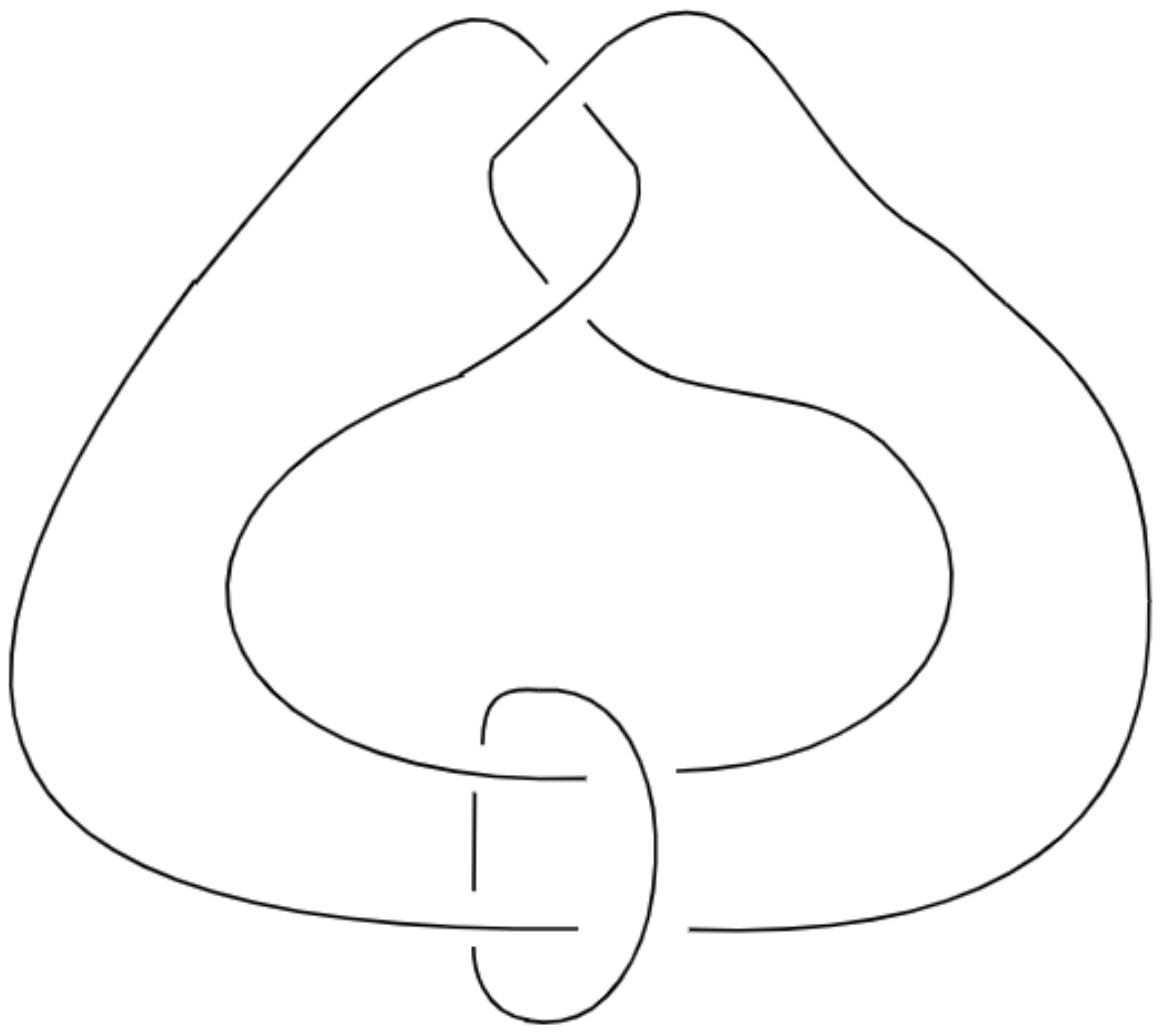}} 
\end{align*}
\caption{Whitehead link}\label{figure-WL} 
\end{figure}

In \cite{Th77}, Thurston shows that the Whitehead link complement $W$ can be obtained by identifying pairs of faces of an octahedron with its vertices deleted. The identification matches face $A$ with $A'$, $B$ with $B'$, etc., in Figure \ref{figure-idealOct}, so as to respect the labeling of the edges.   

\begin{figure}[!htb] 
\includegraphics[width=200 pt]{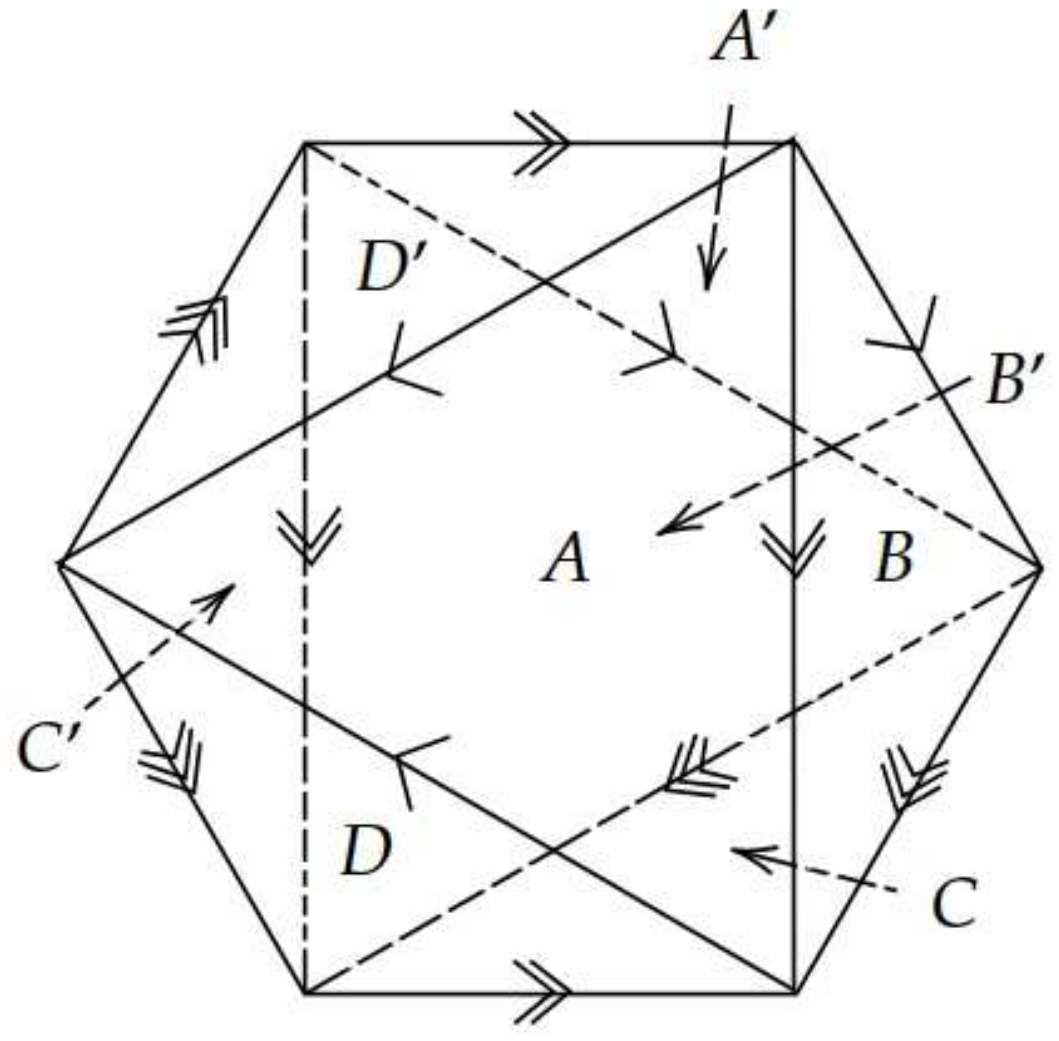}
\caption{Octahedron} \label{figure-idealOct}
\end{figure}

Placing the octahedron in the hyperbolic 3-space $\mathbb{H}^3$ as an ideal octahedron induces a hyperbolic structure on the octahedron. If this is done correctly, then the complement of $W$ will inherit this hyperbolic structure. If the ideal octahedron is regular, then we obtain the complete finite volume hyperbolic structure on $W$, and by deforming the regular octahedron, we obtain incomplete hyperbolic structures on $W$, whose metric completions are hyperbolic Dehn fillings and generalized Dehn fillings on $W$. 

As described in \cite{Th77} and \cite{NZ85}, see also \cite{Mar16} and \cite{Pur20}, a generalized Dehn filling on the Whitehead link is parameterized by a ``generalized Dehn filling invariant" $(p_i,q_i)\in \mathbb{R}^2\cup \{\infty\}$ for each cusp of $W$, denoted by $W(p_1,q_1;p_2,q_2)$, which is defined if the $(p_i,q_i)$ are sufficiently close to $\infty$ by Thurston's hyperbolic Dehn filling theorem. In particular, if $(p_i,q_i)=\infty$ for $i=1, 2$ then the corresponding cusp of $W$ is still a cusp in $W(p_1,q_1;p_2,q_2)$ and if $(p_i,q_i)\in \mathbb{Z}^2$ then the cusp has been filled in by a geodesic along which $W(p_1,q_1;p_2,q_2)$ has the structure of a hyperbolic orbifold, or manifold if $p,q$ are relatively prime.    

It is convenient to use an ideal triangulation of $W$ to study its Dehn fillings, since the shape of an ideal hyperbolic tetrahedron is determined by a single complex parameter.  By subdividing the octahedron of Figure \ref{figure-idealOct} as in Figure \ref{figure-ideal4}, we obtain an ideal triangulation of $W$ with four ideal tetrahedra labeled with complex parameters. We may solve a nonlinear system of equations in these complex parameters to find a hyperbolic structure on a Dehn filling of $W$.   
\begin{figure}[!htb] 
\includegraphics[width=300 pt]{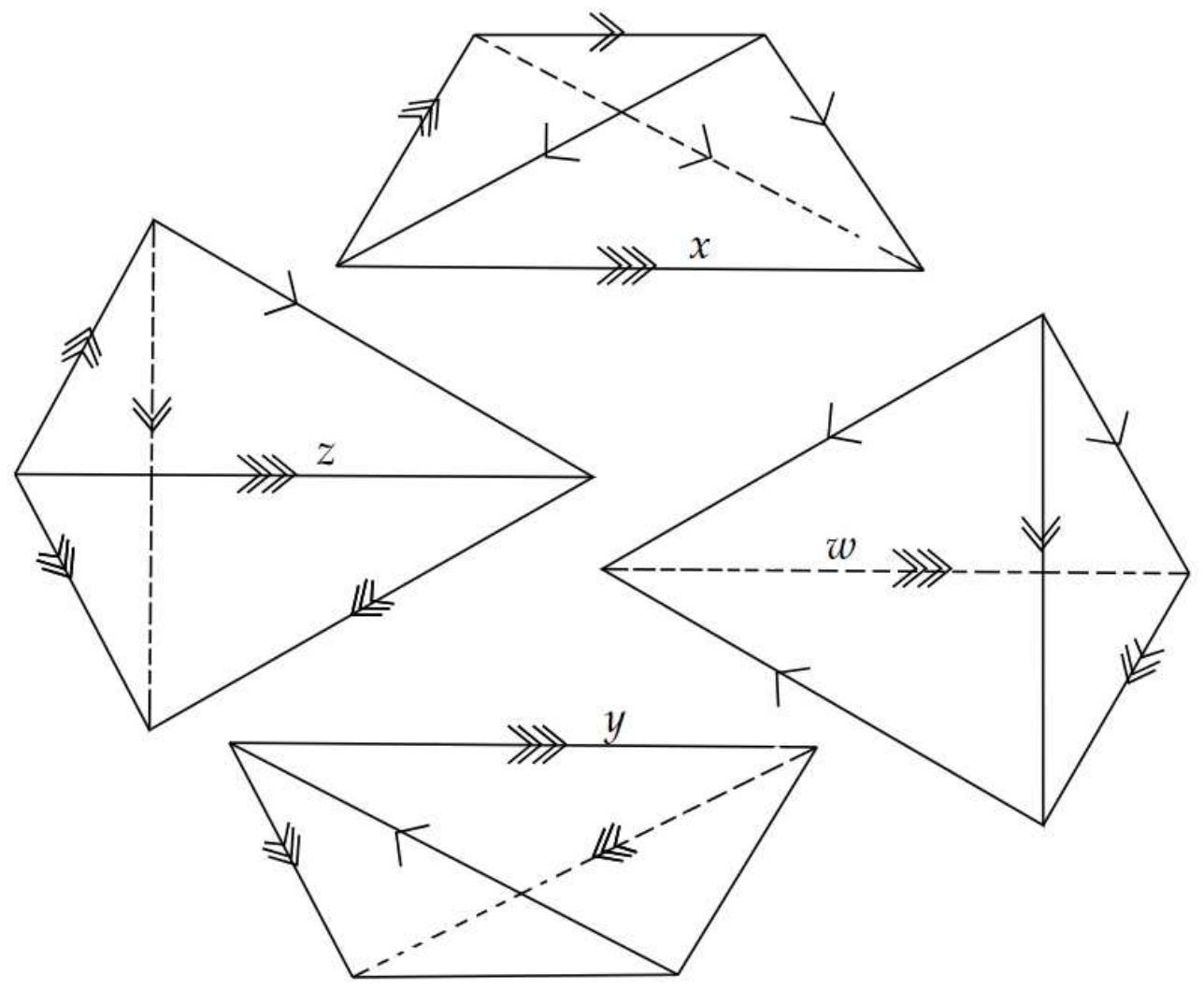}
\caption{An ideal triangulation of $W$} \label{figure-ideal4}
\end{figure}

Following the work \cite{NR90}, the edge gluing equations and Dehn filling equations for $W(p_1,q_1;p_2,q_2)$ (for our purpose, here we use the conjugated version, see Remark \ref{remark-conjugate}) are given by
\begin{equation} \label{equation-gluing} 
\left\{ \begin{aligned}
         &\log w+\log x+\log y+\log z = -2\pi\sqrt{-1} \\
          &        \log(1-w)+\log(1-x)-\log(1-y)-\log(1-z)=0
                          \end{aligned} \right.
                          \end{equation}
and 
\begin{equation}   \label{equation-dehnfilling}
\left\{ \begin{aligned}
         & p_1u_1+q_1v_1 = -2\pi\sqrt{-1} \\
          &  p_2u_2+q_2v_2 = -2\pi\sqrt{-1}      
                          \end{aligned} \right.
                          \end{equation}
respectively, where 
\begin{align}
    u_1&=\log (w-1)+\log x+\log y-\log(y-1)+\pi\sqrt{-1},\\\nonumber
    u_2&=\log (w-1)+\log x+\log z-\log(z-1)+\pi\sqrt{-1},\\\nonumber
    v_1&=2\log x +2\log y+2\pi\sqrt{-1},\\\nonumber
    v_2&=2\log x +2\log z+2\pi\sqrt{-1}.
\end{align}
\begin{remark} \label{remark-conjugate}
    Note that the above equations can be viewed as the conjugated version of the equations (6.1a), (6.1b), (6.3), (6.2a), (6.2b),(6.2c), (6.2d) given in \cite{NR90}. In other words, if $(w_0,x_0,y_0,z_0)$ is a solution to the above equations (\ref{equation-gluing}) and (\ref{equation-dehnfilling}), then $(\bar{w}_0,\bar{x}_0,\bar{y}_0,\bar{z}_0)$ is a solution of those corresponding equations in \cite{NR90}.   
\end{remark}

We are only interested in Dehn filling in one cusp of $W$, so we always have $(p_1,q_1)=\infty$. This requires that the first cusp is complete $u_1=v_1=0$. It is well-known that twist knot complement space $S^3\setminus \mathcal{K}_p$ is homeomorphic to the cusp manifold $W(\infty;1,-p)$.  In this case, the edge gluing equations and Dehn filling equations can be written as follows
\begin{equation} \label{equation-geom1} 
\left\{ \begin{aligned}
         &\log w+\log x+\log y+\log z =- 2\pi\sqrt{-1}, \\
          &        \log(1-w)+\log(1-x)-\log(1-y)-\log(1-z)=0,\\
          &2\log x +2\log y+2\pi\sqrt{-1}=0, \\
          &-u+pv=-2\pi\sqrt{-1},
                          \end{aligned} \right.
                          \end{equation}
where 
\begin{align}
    u&=\log (w-1)+\log x+\log z-\log(z-1)+\pi\sqrt{-1},\\\nonumber
    v&=2\log x +2\log z+2\pi\sqrt{-1}.
\end{align}

From the first three equations of (\ref{equation-geom1}), we obtain 
\begin{align}
\left\{\begin{aligned}
    &y=-\frac{1}{x}, \\
    &w=-\frac{1}{z}, \\
    &z=x.
    \end{aligned}\right.
\end{align}
It turns out that 
\begin{align} \label{equation-uv}
    u&=\log z+\log(z+1)-\log(z-1),\\\nonumber
v&=4\log \left(z\right)+2\pi \sqrt{-1}.
\end{align}
Hence the equations (\ref{equation-geom1}) are reduced to one single equation
\begin{align} \label{equation:geometry-one}
p(4\log z+2\pi\sqrt{-1})=\log z+\log (z+1)-\log(z-1)-2\pi\sqrt{-1},
\end{align}
which is referred as the ``geometric equation" for the twist knot $\mathcal{K}_p$ in the following.

Obviously, taking the exponential of (\ref{equation:geometry-one}), we obtain 
\begin{align} \label{equation-z4p}
   z^{4p}=\frac{z(z+1)}{(z-1)}.
\end{align}
Note that it is equal to equation (\ref{equation-tildey}) which is derived from the critical point equation (\ref{equation-critical}).

As a consequence of Lemma \ref{lemma-HS}, it is easy to show
\begin{proposition} \label{prop-geometricequation}
The geometric equation (\ref{equation:geometry-one}) has a unique solution $z_0=-\bar{y}_{p-1}^{-1}$ given in Lemma \ref{lemma-HS}. Hence 
\begin{align}
     -\frac{2p+3}{4p}\pi<\text{arg}(z_0)<-\frac{2p+2}{4p}\pi.
\end{align}
\end{proposition}
 
\begin{example}
For $p=3$, the geometric equation (\ref{equation:geometry-one}) has a unique solution 
\begin{align}
z_0=-0.5074187881-0.7997871925\sqrt{-1}.
\end{align}
\end{example}
Comparing with computations in Example \ref{example1}, we have 
\begin{align}
z_0=e^{-\pi\sqrt{-1}s_0}=-\bar{y}_2^{-1}, \  y_0=z_0^{-2}.   
\end{align}

\subsubsection{Equivalence of the critical equations and geometric equation}

\begin{proposition} \label{propostion-equiv}
(i) If $(t_0,s_0)\in D_{0\mathbb{C}}$ is the solution of critical point equations (\ref{equation-critical}), then 
\begin{align}
    z_0=e^{-\pi\sqrt{-1} s_0}
\end{align}
gives the solution of geometric equation (\ref{equation:geometry-one}). 

(ii) If $z_0$ is the solution of geometric equation (\ref{equation:geometry-one}), then 
\begin{align}
    (t_0=\frac{\log(1+(z_0-z_0^{-1}))}{2\pi\sqrt{-1}}+1,s_0=-\frac{\log z_0}{\pi \sqrt{-1}})
\end{align}
is the solution of the critical point equations (\ref{equation-critical}). 

\end{proposition}
\begin{proof}
(i) We assume $(t_0,s_0)\in D_{0\mathbb{C}}$ is the solution of (\ref{equation-critical}), then by Proposition \ref{prop-critcalequation}, 
$\text{Re}(s_0)\in \left(\frac{1}{2}+\frac{2}{4p},\frac{1}{2}+\frac{3}{4p}\right)$. 
Using the variable transformation $z_0=e^{-\pi \sqrt{-1}s_0}$, then we have
\begin{align} \label{equation-p}
    p\left(4\log \left(z_0\right)+2\pi \sqrt{-1}\right)=p(-4\pi\sqrt{-1}s_0+2\pi\sqrt{-1}).
\end{align}
So we only need to prove the following identity 
\begin{align} \label{formula-needtoprove}
    &\log\left(1-x_0y_0\right)-\log\left({1-x_0y_0^{-1}}\right)-2\pi\sqrt{-1}s_0+\pi\sqrt{-1}\\\nonumber
    &=-\log z_0-\log (z_0+1)+\log(z_0-1).
\end{align}

By Proposition \ref{prop-critcalequation}, we also have the relations $x_0=1+z_0-z_0^{-1}$ and
$y_0=e^{2\pi \sqrt{-1}s_0}=z_0^{-2}$. 
It follows that
\begin{align} \label{formula-xy=z}
    \frac{1-x_0y_0}{x_0-y_0}=\frac{z_0-1}{z_0(z_0+1)}.
\end{align}
Taking the logarithm of formula (\ref{formula-xy=z}), we conclude that there is an integer $k\in \mathbb{Z}$ such that 
\begin{align} \label{equation-k}
   & \text{arg}\left(1-x_0y_0\right)-\text{arg}(1-x_0y_0^{-1})-2\pi \text{Re}(s_0)+\pi\\\nonumber
   &=-\text{arg} (z_0)-\text{arg}(z_0+1)+\text{arg}(z_0-1)+2k\pi.
\end{align}

By the proof of Proposition \ref{prop-critcalequation} (cf. (\ref{formula-x0y0}) and (\ref{formula-x0y0-1})), it is easy to obtain that 
\begin{align}
     \text{arg}\left(1-x_0y_0\right)-\text{arg}(1-x_0y_0^{-1})-2\pi \text{Re}(s_0)+\pi\in (-\pi,\pi).
\end{align}

On the other hand, since $\text{arg} (z_0)=-\pi \text{Re} (s_0)$, $\text{arg}(z_0+1)\in (-\pi \text{Re}(s_0),0)$ and $\text{arg}(z_0-1)\in (-\pi,-\pi \text{Re}(s_0))$, we obtain
\begin{align}
    -\text{arg}(z_0)-\text{arg}(z_0+1)+\text{arg}(z_0-1)\in (-\pi,\pi).  
\end{align}
Therefore, $k$ is equal to $0$ in formula (\ref{equation-k}) which implies the identity (\ref{formula-needtoprove}). 
Hence, we have shown that if $(t_0,s_0)\in D_{0\mathbb{C}}$ is the solution to the critical point equations (\ref{equation-critical}) given in Proposition \ref{prop-critcalequation}, then $z_0=e^{-\pi\sqrt{-1}s_0}$ is the solution to the geometric equation (\ref{equation:geometry-one}). 

(ii) If $z_0$ is the solution of the geometric equation (\ref{equation:geometry-one}).
By Proposition \ref{prop-geometricequation}, we have
$ -\frac{2p+3}{4p}\pi<\text{arg}(z_0)<-\frac{2p+2}{4p}\pi.$ Let $s_0=-\frac{\log z_0}{\pi \sqrt{-1}}$, $y_0=z_0^{-2}$, $x_0=1+(z_0-z_0^{-1})$ and $t_0=\frac{\log(x_0)}{2\pi\sqrt{-1}}+1$, then $(s_0,x_0,y_0,t_0)$ gives the solution to the critical point equations (\ref{equation-critical})
by the proof of Proposition \ref{prop-critcalequation}. 
\end{proof}

\subsection{Complex volume}
For the solution $(t_0,s_0)\in D_{0\mathbb{C}}$ of the critical point equations (\ref{equation-critical}), let $\zeta(p)=V(p,t_0,s_0)$,
we will prove that 
\begin{theorem} \label{theorem-zeta=volume}
For $p\geq 2$, we have the following formula
    \begin{align}
        2\pi \zeta(p)\equiv vol(S^3\setminus \mathcal{K}_p)+\sqrt{-1}cs(\mathcal{S}^3\setminus \mathcal{K}_p) \mod \pi^2\sqrt{-1}.
    \end{align}
\end{theorem}
We first reformulate a result due to Neumann-Zagier \cite{NZ85} and Yoshida \cite{Yoshida85}. 
\subsubsection{Neumann-Zagier's potential function}
We follow the notation and the statements in \cite{WongYang20-1,WongYang20-2} in our setting.  
Let $M$ be a hyperbolic 3-manifold obtained by doing a hyperbolic $\frac{p}{q}$-Dehn filling from a cusp component $T$ of a hyperbolic link $\mathcal{L}$ in $S^3$. Suppose $m$ and $l$ are respectively the meridian and longitude of the boundary of a tubular neighborhood of this component $T$, and $u,v$  are respectively the holonomy of $m,l$, then a solution to 
\begin{align}
    pu+qv=2\pi\sqrt{-1}
\end{align}
near the complete structure gives a hyperbolic structure on the result manifold $M$. 

Let $\Phi(u)$ be the Neumann-Zagier potential function defined on the deformation space of hyperbolic structures on $S^3\setminus \mathcal{L}$ parametrized by the holonomy of the meridians $u$, which is characterized by the following differential equation
\begin{equation} \label{equation:NZpotentialfunction}
\left\{ \begin{aligned}
         \frac{\partial \Phi(u)}{\partial u} &= \frac{v}{2} \\
                  \Phi(0)&=\sqrt{-1}(vol(S^3\setminus \mathcal{L})+\sqrt{-1}cs(S^3\setminus \mathcal{L})) \ \mod \pi^2\mathbb{Z},
                          \end{aligned} \right.
                          \end{equation}
where $S^3\setminus \mathcal{L}$ is with the complete hyperbolic metric.

We choose a curve $\Gamma$ on the boundary of a tubular neighborhood of $T$, that
is isotopic to the core curve of the filled solid torus. We choose the orientation of $\Gamma$ such that the intersection number $(pm+ql)\cdot \Gamma=1$, and let $\gamma$ be the holonomy of $\Gamma$. Then we have
\begin{align}  \label{equation-vol=NZpotential}
    vol(M)+\sqrt{-1}cs(M)=\frac{\Phi(u)}{\sqrt{-1}}-\frac{uv}{4\sqrt{-1}}+\frac{\pi \gamma}{2}  \mod \sqrt{-1}\pi^2\mathbb{Z}.
\end{align}

\subsubsection{Proof of the Theorem \ref{theorem-zeta=volume}}

\begin{proof}
    Using the transformation $z=e^{-\pi\sqrt{-1}s}$, and by equations (\ref{equation-uv}), 
    we have 
\begin{align} \label{formula-u}
    u&=\log z+\log\left(z+1\right)-\log\left(z-1\right)
\end{align}
and
\begin{align}
    v&=4\log\left(z\right)+2\pi\sqrt{-1}=4\pi\sqrt{-1}\left(-s+\frac{1}{2}\right).
\end{align}

In the following, we construct the Neumann-Zagier potential function $\Phi(u)$ defined on the deformation space of $W$ near the complete structure given by $u=v=0$, i.e. $s=\frac{1}{2}$.     

We set
\begin{align}
    U(p,t,s)&=\frac{1}{2\pi\sqrt{-1}}\left(\text{Li}_2(e^{2\pi\sqrt{-1}(t+s)})+\text{Li}_2(e^{2\pi\sqrt{-1}(t-s)})-3\text{Li}_2(e^{2\pi\sqrt{-1}t})+\frac{\pi^2}{6}\right)\\\nonumber
    &+\frac{1}{4\pi\sqrt{-1}}uv+\pi\sqrt{-1}s^2-\pi\sqrt{-1}s-2\pi\sqrt{-1}t+2\pi\sqrt{-1}
\end{align}
and introduce the function
\begin{align}
    \Phi(u)=2\pi\sqrt{-1}U(p,t(s),s)
\end{align}
where $t(s)=\frac{\log(1+z-z^{-1})}{2\pi\sqrt{-1}}+1$. Note that $u$ is a function of $s$.

We claim that the following two identities holds, i.e.
\begin{align} \label{formula-Vt=0}
    \frac{\partial V(p,t,s)}{\partial t}|_{(t(s),s)}&=-2\pi\sqrt{-1}+3\log(1-e^{2\pi\sqrt{-1}t(s)})\\\nonumber
    &-\log(1-e^{2\pi\sqrt{-1}(t(s)+s)})-\log(1-e^{2\pi\sqrt{-1}(t(s)-s)})=0,
\end{align}
and
\begin{align} \label{formula-Vt=0next}
    -\log(1-e^{2\pi\sqrt{-1}(t(s)+s)})+\log(1-e^{2\pi\sqrt{-1}(t(s)-s)})-u+2\pi\sqrt{-1}s-\pi\sqrt{-1}=0.
\end{align}

First, by exponentiating (\ref{formula-Vt=0}) and (\ref{formula-Vt=0next}), we have
\begin{align}
 (1-e^{2\pi\sqrt{-1}t(s)})^2= e^{2\pi\sqrt{-1}s}+e^{-2\pi\sqrt{-1}s}-2=(z-z^{-1})^2   
\end{align}
and
\begin{align}
 \frac{1-e^{2\pi\sqrt{-1}t(s)}e^{2\pi\sqrt{-1}s}}{e^{2\pi\sqrt{-1}t(s)}-e^{2\pi\sqrt{-1}s}}=\frac{z-1}{z(z+1)}.   
\end{align}
Obviously, $t(s)=\frac{\log(1+z-z^{-1})}{2\pi\sqrt{-1}}+1$ satisfies the above two equations. 

Next, we need to show that both
\begin{align}
3\text{arg}(1-e^{2\pi\sqrt{-1}t(s)})-\text{arg}(1-e^{2\pi\sqrt{-1}(t(s)+s)})-\text{arg}(1-e^{2\pi\sqrt{-1}(t(s)-s)})-2\pi
\end{align}
and 
\begin{align}
-\text{arg}(1-e^{2\pi\sqrt{-1}(t(s)+s)})+\text{arg}(1-e^{2\pi\sqrt{-1}(t(s)-s)})-\text{arg}(u)+2\pi \text{Re}(s)-\pi
\end{align}
lie in $(-\pi,\pi)$. In fact, similar to the proof of Proposition \ref{prop-critcalequation}, one can estimate the ranges of the above arguments using the condition that $s$ is sufficiently close to $\frac{1}{2}$, and finish the proof of (\ref{formula-Vt=0}) and (\ref{formula-Vt=0next}).   

From (\ref{formula-Vt=0}), we have 
\begin{align}
 \frac{\partial U(p,t,s)}{\partial t}|_{(t(s),s)}= \frac{\partial V(p,t,s)}{\partial t}|_{(t(s),s)}=0,  
\end{align}
and together with (\ref{formula-Vt=0next}), we obtain
\begin{align}
     &\frac{d U(p,t,s)}{d s}|_{(t=t(s),s)}\\\nonumber
     &=\frac{\partial U}{\partial t}|_{(t(s),s)}\frac{\partial t}{\partial s}+\frac{\partial U}{\partial s}|_{(t(s),s)}\\\nonumber
     &=-\log(1-e^{2\pi\sqrt{-1}(t(s)+s)})+\log(1-e^{2\pi\sqrt{-1}(t(s)-s)})\\\nonumber
     &+\frac{1}{4\pi\sqrt{-1}}\left(v\frac{d u}{d s}+u\frac{d v}{d s}\right)+2\pi\sqrt{-1}s-\pi\sqrt{-1}\\\nonumber
     &=\frac{1}{4\pi\sqrt{-1}}v\frac{du}{ds}.
\end{align}
Therefore, we get
\begin{align}
    \frac{d \Phi(u)}{d u}&= \frac{d \Phi(u)}{d s}\frac{d s}{d u}=2\pi\sqrt{-1}\frac{dU}{ds}\frac{ds}{du}=\frac{v}{2}.
\end{align}

When $s=\frac{1}{2}$,  we have $z=-\sqrt{-1}$, then
\begin{align}
    t\left(\frac{1}{2}\right)=\frac{\log(1-2\sqrt{-1})}{2\pi\sqrt{-1}}+1.
\end{align}
By Lemma \ref{lemma-U=Vol}, we have 
\begin{align}   \Phi(0)&=2\pi\sqrt{-1}U(p,t\left(\frac{1}{2}\right),\frac{1}{2})\\\nonumber
    &= \sqrt{-1}(vol(W)+\sqrt{-1}cs(W)). 
\end{align}
So we prove that $\Phi(u)$ satisfies the differential equation (\ref{equation:NZpotentialfunction}). Hence $\Phi(u)$ is the Neumann-Zagier potential function.

Moreover, when $(u,v)$ satisfies the geometric equation $u-pv=2\pi\sqrt{-1}$, i.e.(\ref{equation:geometry-one}), then $\gamma=v$ by the definition of holonomy $\gamma$. Since $v=4\pi\sqrt{-1}(-s+\frac{1}{2})$, we compute 
\begin{align}  \label{formula-Psi(u)=V}
     &\frac{\Phi(u)}{\sqrt{-1}}-\frac{uv}{4\sqrt{-1}}+\frac{\pi \gamma}{2}\\\nonumber
     &=\frac{1}{\sqrt{-1}}\left(\text{Li}_2(e^{2\pi\sqrt{-1}(t+s)})+\text{Li}_2(e^{2\pi\sqrt{-1}(t-s)})-3\text{Li}_2(e^{2\pi\sqrt{-1}t})+\frac{\pi^2}{6}\right)\\\nonumber
    &+\frac{1}{2\sqrt{-1}}uv+2\pi(\pi\sqrt{-1}s^2-\pi\sqrt{-1}s-2\pi\sqrt{-1}t+2\pi\sqrt{-1})\\\nonumber
    &-\frac{uv}{4\sqrt{-1}}+\frac{\pi v}{2}\\\nonumber
    &=\frac{1}{\sqrt{-1}}\left(\text{Li}_2(e^{2\pi\sqrt{-1}(t+s)})+\text{Li}_2(e^{2\pi\sqrt{-1}(t-s)})-3\text{Li}_2(e^{2\pi\sqrt{-1}t})+\frac{\pi^2}{6}\right)\\\nonumber
    &+2\pi\left(\pi\sqrt{-1}((2p+1)s^2-(2p+3)s-2t)\right)+(p+6)\pi^2\sqrt{-1}\\\nonumber
    &=2\pi V(p,t,s)+(p+6)\pi^2\sqrt{-1}.
\end{align}
On the other hand, by (\ref{equation-vol=NZpotential}),  we have
\begin{align} \label{formula-Psi(u)=vol}
    \frac{\Phi(u)}{\sqrt{-1}}-\frac{uv}{4\sqrt{-1}}+\frac{\pi \gamma}{2}=vol(S^3\setminus\mathcal{K}_p)+\sqrt{-1}cs(S^3\setminus\mathcal{K}_p) \mod (\pi^2 \sqrt{-1}).
\end{align}

Finally, by Proposition \ref{propostion-equiv}, suppose $(t_0,s_0)$ is the solution of equation (\ref{equation-critical}), then under the variable transformation $z=e^{-\pi\sqrt{-1}s_0}$,  the corresponding $u=\log z+\log(z+1)-\log(z-1),v=4\log z+2\pi\sqrt{-1}$ satisfies the geometric equation (\ref{equation:geometry-one}). Therefore, combining  (\ref{formula-Psi(u)=V}) and (\ref{formula-Psi(u)=vol}) together, we obtain
\begin{align}
    2\pi V(p,t_0,s_0)\equiv vol(S^3\setminus\mathcal{K}_p)+\sqrt{-1}cs(S^3\setminus\mathcal{K}_p) \mod (\pi^2 \sqrt{-1}).
\end{align}
\end{proof}

\subsubsection{A dilogarithm identity}
\begin{lemma} \label{lemma-dilogidentity}
For $x_0=1-2\sqrt{-1}$, we have

\begin{align}
  3\text{Li}_2(x_0)-2\text{Li}_2(-x_0)+4\text{Li}_2(\sqrt{-1})+2\pi\sqrt{-1}\log(x_0)-\frac{5}{6}\pi^2=0.  
\end{align}
\end{lemma}

\begin{proof}
Let $x=2$, $y=1-\sqrt{-1}$ in Hill's formula 
\begin{align}
\text{Li}_2(xy)=\text{Li}_2(x)+\text{Li}_2(y)+\text{Li}_2(\frac{xy-x}{1-x})+\text{Li}_2(\frac{xy-y}{1-y})+\frac{1}{2}\log^2\left(\frac{1-x}{1-y}\right),   
\end{align}
we obtain
\begin{align} \label{formula-Li22}
&\text{Li}_2(2(1-\sqrt{-1}))\\\nonumber
&=\text{Li}_2(2)+\text{Li}_2(1-\sqrt{-1})+\text{Li}_2(2\sqrt{-1})+\text{Li}_2(-1-\sqrt{-1})+\frac{1}{2}\log^2(\sqrt{-1})\\\nonumber
&=\text{Li}_2(2\sqrt{-1})+\text{Li}_2(-1-\sqrt{-1})-\text{Li}_2(\sqrt{-1})+\frac{\pi^2}{6}-\frac{5\log 2}{4}\pi\sqrt{-1},
\end{align}
where in the second ``=" we have used 
\begin{align}
\text{Li}_2(1-\sqrt{-1})=-\text{Li}_2(\sqrt{-1})+\frac{\pi^2}{24}-\frac{\pi\log 2}{4}\sqrt{-1}    
\end{align}
and 
\begin{align}
  \text{Li}_2(2)=\frac{\pi^2}{4}-\sqrt{-1}\pi\log2. 
\end{align}
By using the formula  
\begin{align}
 \text{Li}_2(z)+\text{Li}_2(1-z)=\frac{\pi^2}{6}-\log z\log(1-z),   
\end{align}
we have
\begin{align}
\text{Li}_2(1-2\sqrt{-1})&=\frac{\pi^2}{6}-\text{Li}_2(2\sqrt{-1})-\log(2\sqrt{-1})\log(1-2\sqrt{-1})\\\nonumber
&=\frac{\pi^2}{6}-\text{Li}_2(2\sqrt{-1})-(\log2+\frac{\pi}{2}\sqrt{-1})\log(1-2\sqrt{-1})
\end{align}
and
\begin{align}
&\text{Li}_2(-1+2\sqrt{-1})\\\nonumber
&=\text{Li}_2(1-2(1-\sqrt{-1}))\\\nonumber
&=-\text{Li}_2(2(1-\sqrt{-1}))+\frac{\pi^2}{6}-\log(2(1-\sqrt{-1}))\log(-1+2\sqrt{-1})\\\nonumber
&=-\text{Li}_2(2\sqrt{-1})-\text{Li}_2(-1-\sqrt{-1})+\text{Li}_2(\sqrt{-1})-\frac{\pi^2}{4}-\frac{\log 2}{4}\pi\sqrt{-1}\\\nonumber
&+\left(-\frac{3\log2}{2}+\frac{\pi}{4}\sqrt{-1}\right)\log(1-2\sqrt{-1}),
\end{align}
where in the third ``=" we have used formula (\ref{formula-Li22}). 

Hence
\begin{align}
&3\text{Li}_2(1-2\sqrt{-1})-2\text{Li}_2(-1+2\sqrt{-1})\\\nonumber
&=-\text{Li}_2(2\sqrt{-1})+2\text{Li}_2(-1-\sqrt{-1})-2\text{Li}_2(\sqrt{-1})\\\nonumber
&+\pi^2+\frac{\log 2}{2}\pi\sqrt{-1}-2\pi\sqrt{-1}\log(1-2\sqrt{-1})    
\end{align}
Furthermore, since $2\sqrt{-1}=(\pm 1\pm \sqrt{-1})^2$, we have 
\begin{align}
\text{Li}_2(2\sqrt{-1})=2\text{Li}_2(1+\sqrt{-1})+2\text{Li}_2(-1-\sqrt{-1}).    
\end{align}
Moreover, 
\begin{align}
    \text{Li}_2(1+\sqrt{-1})=\text{Li}_2(\sqrt{-1})+\frac{\pi^2}{12}+\frac{\log 2}{4}\pi\sqrt{-1}
\end{align}

Therefore, we obtain 
\begin{align}
&3\text{Li}_2(1-2\sqrt{-1})-2\text{Li}_2(-1+2\sqrt{-1})\\\nonumber
&=-2\text{Li}_2(1+\sqrt{-1})-2\text{Li}_2(\sqrt{-1})\\\nonumber
&+\pi^2+\frac{\log 2}{2}\pi\sqrt{-1}-2\pi\sqrt{-1}\log(1-2\sqrt{-1})\\\nonumber
&=-4\text{Li}_2(\sqrt{-1})+\frac{5}{6}\pi^2-2\pi\sqrt{-1}\log(1-2\sqrt{-1}).
\end{align}
\end{proof}
\begin{lemma} \label{lemma-U=Vol}
\begin{align}
2\pi\sqrt{-1}U(p,t\left(\frac{1}{2}\right),\frac{1}{2})= \sqrt{-1}(vol(W)+\sqrt{-1}cs(W))
\end{align}    
\end{lemma}
\begin{proof}
Since $x_0=e^{2\pi\sqrt{-1}t\left(\frac{1}{2}\right)}$, and $2\pi\sqrt{-1}t\left(\frac{1}{2}\right)=\log(x_0)-2\pi \sqrt{-1}$, we have 
    \begin{align}
        &2\pi\sqrt{-1}U(p,t\left(\frac{1}{2}\right),\frac{1}{2})\\\nonumber
        &=\left(2\text{Li}_2(-x_0)-3\text{Li}_2(x_0)+\frac{\pi^2}{6}\right)+2\pi\sqrt{-1}\left(-\frac{1}{4}\pi\sqrt{-1}-\log(x_0)\right)\\\nonumber
        &=2\text{Li}_2(-x_0)-3\text{Li}_2(x_0)-2\pi\sqrt{-1}\log(x_0)+\frac{2}{3}\pi^2\ (\text{by Lemma \ref{lemma-dilogidentity}})  \\\nonumber
        &=4\text{Li}_2(\sqrt{-1})-\frac{5}{6}\pi^2+\frac{2}{3}\pi^2  \\\nonumber
        &=\sqrt{-1}(4G+\sqrt{-1}\frac{\pi^2}{4})\\\nonumber
        &=\sqrt{-1}(vol(W)+\sqrt{-1}cs(W)).
    \end{align}
\end{proof}

\begin{remark}
In our first version on arXiv, we use the explicit formula of complex volume for $W(\infty, -1/p)$ obtained by Meyerhoff-Neumann \cite{MeyNeu92} to prove Theorem \ref{theorem-zeta=volume}, by directly comparing with the critical value of the potential function.
\end{remark}

\subsection{Adjoint twisted Reidemeister Torsion}

Recall that $\omega(p)$ is given by formula (\ref{formula-omegap}), let $\mathbb{T}_{\mathcal{K}_p,\mu}$ be the adjoint twisted Reidemeister torsion of $\mathcal{K}_p$ with respect to the meridan $\mu$, this section is devoted to proving the following result.
\begin{theorem} \label{theorem-omega=Rtorsion}
We have the following identity   \begin{align}
   \omega(p)=
\frac{1}{4\pi\sqrt{2}\sqrt{\mathbb{T}_{\mathcal{K}_p,\mu}^\rho}}.
\end{align} 
\end{theorem}

First, we review the computation of the adjoint twisted Reidemeister torsion of the twist knots  $\mathcal{K}_p$ illustrated in \cite{DHY09, Tr14, Tr15, Tr16}, we refer to \cite{Porti97} for the detailed definition of the adjoint twisted Reidemeister torsion.

The reader should  note that the twist knot  $\mathcal{K}_p$ is denoted by $J(2,2p)$ in literature \cite{HS05, DHY09,Tr14, Tr15, Tr16}. 
Hence the knot group of the twist knot $\mathcal{K}_p$ has a presentation 
\begin{align}
    \pi_1(S^3\setminus \mathcal{K}_p)=\langle a,b| w^pa=bw^p  \rangle
\end{align}
where $a,b$ are two generators and $w=ba^{-1}b^{-1}a$. Furthermore, a meridian $\mu$ and the preferred longitude $\lambda$ are given by 
\begin{align}
\mu=a,  \ \text{and} \ \lambda=(w^*)^pw^p    
\end{align}
 where $w^*$ is the word obtained by writing $w$ in the reversed order.
 
Suppose $\rho: \pi_1(S^3\setminus \mathcal{K}_p)\rightarrow SL_2(\mathbb{C})$ is a nonabelian representation. Up to conjugation, we may assume that 
\begin{align}
    \rho(a)=\begin{pmatrix}
    s & 1 \\
    0  &  s^{-1}
    \end{pmatrix}
\end{align}
and 
\begin{align}
    \rho(b)=\begin{pmatrix}
    s & 0 \\
    2-r  &  s^{-1}
    \end{pmatrix}
\end{align}
where $r=tr(\rho(ab^{-1}))$. Then $(s,r)\in \mathbb{C}^2$ satisfies the matrix equation $\rho(w^pa)=\rho(bw^p)$
This matrix equation gives the following Riley equation \cite{Tr15,Tr16},  
\begin{align} \label{equation-Riley}
    (1-(r-s^2-s^{-2}))S_{p-1}(z)-S_{p-2}(z)=0.
\end{align}
where 
\begin{align}  \label{formula-z}
    z=tr \rho(w)=2+(r-s^2-s^{-2})(r-2),
\end{align}
and $S_{k}(z)$'s  are the Chebychev polynomials of the second kind defined by $S_0(z)=1$, $S_1(z)=z$ and   $S_k(z)=zS_{k-1}(z)-S_{k-2}(z)$ for all $k\in \mathbb{Z}$.

Moreover, we have 
\begin{align}
    \rho(\lambda)=\begin{pmatrix}
        l & * \\ 
        0 & l^{-1}
    \end{pmatrix}
\end{align}
where 
\begin{align}  \label{formula-l}
    l=\frac{s^2(r-1)-1}{s^2-(r-1)}.
\end{align}

We consider the following variable transformation  
\begin{align} \label{formula-variabletrans}
z=y+y^{-1}, \ \ r=x+1.
\end{align}
When $s=1$, the equation (\ref{equation-Riley}) and formula (\ref{formula-z})  give 
\begin{align} \label{formula-xz}
    \frac{x+1-z}{x-1}=1-\frac{z-2}{x-1}=3-r=\frac{S_{p-2}(z)}{S_{p-1}(z)}
\end{align}
and
\begin{align}
     (x-1)^2&=y+y^{-1}-2
     \end{align}
respectively.

From formula (\ref{formula-xz}), we obtain
\begin{align}
    x&=\frac{(z-1)S_{p-1}(z)-S_{p-2}(z)}{S_{p-1}(z)-S_{p-2}(z)}\\\nonumber
    &=\frac{S_{p}(z)-S_{p-1}(z)}{S_{p-1}(z)-S_{p-2}(z)}\\\nonumber
    &=\frac{y^{2p+1}+1}{y^{2p}+y}.
\end{align}

Therefore, when $s=1$, the Riley equation (\ref{equation-Riley}) is equivalent to the following equations
\begin{align} \label{equation-Riley1}
 \left\{ \begin{aligned}
     &x=\frac{y^{2p+1}+1}{y^{2p}+y}, \\
      &(x-1)^2=y+y^{-1}-2.
                          \end{aligned} \right.   
\end{align}
under the variable transformations (\ref{formula-variabletrans}). 
Note that the above equations are just the exponential form for the critical equations for $V(p,t,s)$ given by (\ref{equation-criticalexp}).

The formula for the adjoint twisted Reidemeister torsion of $\mathcal{K}_p$ with respect to the longitude $\lambda$ has been computed  in \cite{Tr14} 
\begin{align}
\mathbb{T}_{\mathcal{K}_p,\lambda}^{\rho}&=\frac{-1}{(r+2-t_{a}^2)(4-t_{\alpha}^2+(r-2)(r+2-t_a^2))}\\\nonumber
&\cdot\left(\frac{(2p-1)r^2+rt_{a}^2-2pt_{a}^2(t_{a}^2-2)}{r^2-rt_{a}^2+2t_{a}^2}+2p\right),
\end{align}
where  $t_{a}=tr\rho(a)$.

By \cite{Porti97}, the adjoint twisted Reidemeister torsion with respect to the meridian $\mu=a$ is given by  
\begin{align}
    \mathbb{T}_{\mathcal{K}_p,\mu}^{\rho}=\frac{\mathbb{T}_{\mathcal{K}_p,\lambda}^{\rho}}{\frac{\partial l}{\partial s}}\frac{l}{s}
\end{align}
Using (\ref{formula-l}), we obtain 
\begin{align}
    \left(\frac{\partial l}{\partial s}\right)^{-1}\frac{l}{s}
    &=\frac{(s^2-(r-1))(s^2(r-1)-1)}{2s^2r(2-r)}.
\end{align}

When $s=1$, then $t_{a}^2=(s+s^{-1})^2=4$, so we have
\begin{align}
\mathbb{T}_{\mathcal{K}_p,\mu}^{\rho}&=\frac{-1}{(r-2)^3}\left(\frac{(2p-1)r^2+4r-16p}{(r^2-4r+8)}+2p\right)\frac{(2-r)(r-2)}{2r(2-r)}\\\nonumber
    &=\frac{-1}{2(r-2)^2}\frac{(4p-1)r+(4-8p)}{(r-2)^2+4}.
\end{align}

Under the above variable transformation (\ref{formula-variabletrans}), we obtain
\begin{align}
    (r-2)^2=(x-1)^2=y+y^{-1}-2, \ (r-2)^2+4=y+y^{-1}+2.
\end{align}
Hence, we get
\begin{align}
    \mathbb{T}_{\mathcal{K}_p,\mu}^{\rho}=\frac{-((4p-1)(x+1)+(4-8p))}{2(y-y^{-1})^2}.
\end{align}

\begin{theorem} \label{theorem-Rtorsion-identity1}
If $(x,y)$ satisfies (\ref{equation-Riley1}), then we have the following identity
\begin{align}
    \frac{2x^2(y-y^{-1})^2}{(1-xy)(1-x/y)}\mathbb{T}_{\mathcal{K}_p,\mu}^\rho=H(p,x,y).
\end{align}
\end{theorem}
\begin{proof}
\begin{align}
    \frac{2x^2(y-y^{-1})^2}{(1-xy)(1-x/y)}\mathbb{T}_{\mathcal{K}_p,\mu}^\rho&=-\frac{x^2y((4p-1)(x+1)+(4-8p))}{(1-xy)(y-x)} \\\nonumber
    &=-\frac{x^2y((4p-1)(x-1)+2)}{(1-xy)(y-x)}\cdot\frac{(1-x)}{(1-x)}\\\nonumber
    &=\frac{x^2y((4p-1)(1-x)^2-2(1-x))}{(1-xy)(y-x)(1-x)}
\end{align}
On the other hand, we have 
\begin{align}
    &H(p,x,y)\\\nonumber&=\begin{vmatrix}
    \frac{xy}{1-xy}+\frac{x/y}{1-x/y}-\frac{3x}{1-x} & \frac{xy}{1-xy}-\frac{x/y}{1-x/y} \\
    \frac{xy}{1-xy}-\frac{x/y}{1-x/y}  & 2p+1+ \frac{xy}{1-xy}+\frac{x/y}{1-x/y}
    \end{vmatrix}\\\nonumber
    &=\frac{x((1-2p)yx^2+(-y^2+2y-1+4p(y^2-y+1))x+(2p+1)y^2-(6p+3)y+(2p+1))}{(y-x)(1-xy)(1-x)}
\end{align}

So, we only need to show the following identity holds
\begin{align}
    &(1-2p)yx^2+(-y^2+2y-1+4p(y^2-y+1))x+(2p+1)y^2-(6p+3)y+(2p+1)\\\nonumber 
    &=xy((4p-1)(1-x)^2-2(1-x))\\\nonumber
    &=xy((4p-1)(y+y^{-1}-2)-2(1-x)).
\end{align}

By direct computations, we have 
\begin{align}
    &(1-2p)yx^2+(-y^2+2y-1+4p(y^2-y+1))x+(2p+1)y^2-(6p+3)y+(2p+1)\\\nonumber
    &-xy((4p-1)(y+y^{-1}-2)-2(1-x))\\\nonumber
    &=(2p+1)(-yx^2+2yx+y^2-3y+1)\\\nonumber
    &=-y(2p+1)(x^2-2x+3-y-y^{-1})\\\nonumber
    &=-y(2p+1)((x-1)^2+2-y-y^{-1})=0.
\end{align}
\end{proof}

Now, for the solution $(t_0,s_0)$ of the critical point equations (\ref{equation-critical}) given by Proposition \ref{prop-critcalequation}, then $(x_0,y_0)$ satisfies (\ref{equation-Riley1}).  Therefore, by Theorem \ref{theorem-Rtorsion-identity1} and the formula  $(\ref{formula-omegap})$ for $\omega(p)$, we have the following identity
\begin{align}
\frac{1}{\sqrt{\mathbb{T}_{\mathcal{K}_p,\mu}^\rho}}=\frac{\sqrt{2}x_0(y_0-y_0^{-1})}{\sqrt{(1-x_0y_0)(1-x_0/y_0)}\sqrt{H(p,x_0,y_0)}}=4\pi\sqrt{2}\omega(p)
\end{align}
which implies Theorem \ref{theorem-omega=Rtorsion}.

\begin{remark}
In \cite{MuTr}, J. Murakami and A. Tran proposed a general conjecture that the determinant of the Hessian matrix for the potential function of the colored Jones polynomial will give the adjoint twisted Reidemeister torsion. They proved their conjecture for the double twist knots and other special links. So, the proof of Theorem \ref{theorem-Rtorsion-identity1} presented here can be regarded as a special case of their proof for the double twist knots. The authors thank Professor J. Murakami for sending us their preprint \cite{MuTr} to us.  
\end{remark}

\section{Appendices} \label{Section-App}
\subsection{Proof of Lemma \ref{lemma-regionD'0}}   \label{appendix-0}
We define the function
\begin{align}
    v(t,s)=\Lambda(t+s)+\Lambda(t-s)-3\Lambda\left(t\right).
\end{align}
We set
    \begin{align}
    D&=\{(t,s)\in \mathbb{R}^2| 1< t+s< 2, 0< t-s<1, \frac{1}{2}< t<1\}, \\\nonumber
    D_0&=\{0.02 \leq t-s, 1.02 \leq t+s, 0.25 \leq s\leq 0.75,0.5\leq t\leq 0.909\}.
\end{align}
Then we have
\begin{lemma} 
The following domain
    \begin{align} 
        \left\{(t,s)\in D| v(t,s)>\frac{3.56337}{2\pi}\right\}
    \end{align}
    is included in the region $D_0$.
\end{lemma}
\begin{proof}
Recall the definition of the function $\Lambda(t)$:
\begin{align}
  \Lambda(t)=Re\left(\frac{1}{2\pi\sqrt{-1}}\text{Li}_2(e^{2\pi\sqrt{-1}t})\right). 
\end{align}
We have $\Lambda'(t)=-\log 2\sin(\pi t)$ and $\Lambda''(t)=-\pi \cot(\pi t)$ for $t\in [0,1]$. In the following, we only describe how to obtain the upper bound 0.909 of $t$ in the definition of $D'_0$, the other bounds can be derived similarly.   

For a fixed $t\in (\frac{1}{2},1)$, we regard $v(t,s)$ as function of $s$, it follows that 
\begin{align}
    v_s(t,s)=\Lambda'(t+s)-\Lambda'(t-s)=\log\frac{\sin(\pi(t-s))}{\sin(\pi(t+s-1))}. 
\end{align}
The critical point is given by $s=\frac{1}{2}$. Furthermore, $v_{ss}(t,\frac{1}{2})<0$, i.e. for any fixed $t\in (\frac{1}{2},1)$, 
as a function of $s$, $v(t,s)$ takes the maximal value at $s=\frac{1}{2}$.  

As a function of $t$, 
\begin{align}
    v(t,\frac{1}{2})=\Lambda(t+\frac{1}{2})+\Lambda(t-\frac{1}{2})-3\Lambda(t)=2\Lambda(t+\frac{1}{2})-3\Lambda(t).
\end{align}
On can compute that there is a point $t_0\approx 0.8$ that is a maximal point of $v(t,\frac{1}{2})$, and such that $v(t,\frac{1}{2})$ increases (resp. decreases) on the interval $(\frac{1}{2},t_0)$ (resp. the interval $(t_0,1)$). 
We compute
$2\pi v(0.909,\frac{1}{2})=3.4589<3.56337$. By the above analysis, if $(t,s)\in D$ is such that $v(t,s)>\frac{3.56337}{2\pi }$, then $t<0.909$. 
\end{proof}

\subsection{Proof of Lemma \ref{lemma-UnnotUn}} \label{appendix-LemmaUn}

Based on Lemma \ref{lemma-Li2}, in order to study the asymptotic behaviour of the function 
\begin{align}
    f(X,Y;n)=ReV(p,t+X\sqrt{-1},s+Y\sqrt{-1}), 
\end{align}
we introduce the following  function
\begin{equation} \label{eq:2}
F(X,Y;n)=\left\{ \begin{aligned}
         &0  &  \ (\text{if} \ X+Y\geq 0) \\
         &\left((t+s)-\frac{3}{2}\right)(X+Y) & \ (\text{if} \ X+Y<0)
                          \end{aligned} \right.
                          \end{equation}
\begin{equation*}
 +\left\{ \begin{aligned}
         &0  &  \ (\text{if} \ X-Y\geq 0) \\
         &\left((t-s)-\frac{1}{2}\right)(X-Y) & \ (\text{if} \ X-Y<0)
                          \end{aligned} \right.
                          \end{equation*}
\begin{equation*}
+\left\{ \begin{aligned}
         &0  &  \ (\text{if} \ X\geq 0) \\
         &\left(\frac{3}{2}-3t\right)X & \ (\text{if} \ X<0)
                          \end{aligned} \right.  +\left(p+\frac{3}{2}+n-(2p+1)s\right)Y+X
                          \end{equation*}
where we use $t+s-\frac{3}{2}$ instead of $t+s-\frac{1}{2}$ in the first summation since in our situation $1< t+s<2$.

  Note that $F(X,Y;n)$ is a piecewise linear function,  we subdivide the plane $\{(X,Y)\in \mathbb{R}^2\}$ into eight regions to discuss this function. We study the conditions such that $F(X,Y,n)$ has the following property: 
\begin{align} \label{formula-F(X,Y)2}
    F(X,Y;n)\rightarrow \infty \ \text{as} \ X^2+Y^2\rightarrow \infty.
\end{align}

(I) $X\geq Y\geq 0$, then
\begin{align}
    F(X,Y;n)&=\left(\left(p+\frac{3}{2}+n\right)-(2p+1)s\right)Y+X\\\nonumber
  &\geq \left(\left(p+\frac{5}{2}+n\right)-(2p+1)s\right)Y.  
\end{align}
When $s<\frac{p+\frac{5}{2}+n}{2p+1}$, then the property (\ref{formula-F(X,Y)2}) holds.

(II) $Y\geq X\geq 0$, then
\begin{align}
     F(X,Y;n)=\left(t-s+\frac{1}{2}\right)X+(p+2+n-2ps-t)Y.
\end{align}
When $t+2ps<p+2+n$, we have
\begin{align}
     F(X,Y;n)&=\left(t-s+\frac{1}{2}\right)X+(p+2+n-2ps-t)Y\\\nonumber
    &\geq \left(p+\frac{5}{2}+n-(2p+1)s\right)X.
\end{align}
Hence, the property (\ref{formula-F(X,Y)2}) holds in this case if  $\left(p+\frac{5}{2}+n\right)-(2p+1)s>0$.

(III) $X+Y\geq 0$ and $X\leq 0$, then
\begin{align}
     F(X,Y;n)=\left(2-2t-s\right)X+(p+2+n-2ps-t)Y.
\end{align}
When $t+2ps<p+2+n$, then 
\begin{align}
     F(X,Y;n)&\geq \left(2-2t-s\right)X+(p+2+n-2ps-t)(-X) \\\nonumber
    &=(p+n+t-(2p-1)s)(-X).
\end{align}
If $p+n+t-(2p-1)s>0$, 
we obtain that the property (\ref{formula-F(X,Y)2}) holds in this case. 

(IV) $X+Y\leq 0$ and $Y\geq 0$, then 
\begin{align}
     F(X,Y;n)&=\left(\frac{1}{2}-t\right)X+\left(p+\frac{1}{2}+n-(2p-1)s\right)Y\\\nonumber
     &=\left(-\frac{1}{2}+t\right)(-X)+\left(p+\frac{1}{2}+n-(2p-1)s\right)Y\\\nonumber
    &\geq (p+n+t-(2p-1)s)Y.
\end{align}
When $p+n+t-(2p-1)s>0$, we can see that the property (\ref{formula-F(X,Y)2}) holds in this case.

(V) $X-Y\leq 0$ and $Y\leq 0$, then 
\begin{align}
     F(X,Y;n)=\left(\frac{1}{2}-t\right)X+\left(p+\frac{1}{2}+n-(2p-1)s\right)Y.
\end{align}
Since $\frac{1}{2}<t<1$, 
\begin{align}
     F(X,Y;n)&=\left(-\frac{1}{2}+t\right)(-X)+\left(p+\frac{1}{2}+n-(2p-1)s\right)Y\\\nonumber
    &\geq \left(-\frac{1}{2}+t\right)(-Y)+\left(p+\frac{1}{2}+n-(2p-1)s\right)Y\\\nonumber
    &=(-p-1-n+t+(2p-1)s)(-Y),
\end{align}
if $t+(2p-1)s-p-1-n>0$, it follows that the property (\ref{formula-F(X,Y)2}) holds in this case. 

(VI) $X-Y\geq 0$ and $X\leq 0$, then 
\begin{align}
     F(X,Y;n)=\left(1+s-2t\right)X+\left(p+n+t-2ps\right)Y.
\end{align}
When $p+n+t-2ps<0$, and since $-Y\geq -X\geq 0$,  
 \begin{align}
     F(X,Y;n)&\geq \left(1+s-2t\right)X-\left(p+n-2ps+t\right)(-X)\\\nonumber
    &=(-p-1-n+t+(2p-1)s)(-X),
\end{align}
if $t+(2p-1)s-p-1-n>0$, it follows that the property (\ref{formula-F(X,Y)2}) holds in this case.

(VII) $X+Y\leq 0$ and $X\geq 0$, then 
   \begin{align}
    F(X,Y;n)=\left(t+s-\frac{1}{2}\right)X+\left(p+n-2ps+t\right)Y,
\end{align}
since $t+s-\frac{1}{2}>0$, if $p+n-2ps+t< 0$, by $Y\leq 0$,   
it follows that the property (\ref{formula-F(X,Y)2}) holds in this case.

(VIII) $X+Y\geq 0$ and $Y\leq 0$, then 
   \begin{align}
     F(X,Y;n)=X+\left(p+\frac{3}{2}+n-(2p+1)s\right)Y.
\end{align}
By $X\geq -Y$,
   \begin{align}
     F(X,Y;n)&\geq (-Y)+\left(p+\frac{3}{2}+n-(2p+1)s\right)Y\\\nonumber
    &=\left(-p-\frac{1}{2}-n+(2p+1)s\right)(-Y),
\end{align}
if $-p-\frac{1}{2}-n+(2p+1)s>0$, we obtain that the property (\ref{formula-F(X,Y)2}) holds in this case.

Given $n\in \mathbb{Z}$, we introduce the region $U_n$ as follows.
\begin{small}
    \begin{align}
    U_n=\left\{(t,s)\in D_0\bigg|\frac{p+n+1-t}{2p-1}<s<\frac{p+n+t}{2p-1}, \frac{p+n+t}{2p}<s<\frac{p+2+n-t}{2p}\right\}.
\end{align}
\end{small}
\begin{remark} \label{remark-Un}
   For $n\geq p-1$ or $n\leq -(p+1)$, $U_n=\emptyset$. In fact, for $n\geq p-1$, we have 
       $\frac{p+n+1-t}{2p-1}>\frac{p+2+n-t}{2p}$ on $D_0$, and for $n\leq -(p+1)$, we have 
        $\frac{p+n+t}{2p}>\frac{p+n+t}{2p-1}$ on $D_0$. Hence $U_n=\emptyset$.  
\end{remark}

From the above discussions, together with Lemma \ref{lemma-Li2}, we have
\begin{lemma} 
For any $(t,s)\in D_0$, 

(i) if $(t,s)\in U'_n$, then we have 
\begin{align}
    f(X,Y;n)\rightarrow \infty \ \text{as } \ X^2+Y^2\rightarrow +\infty.
\end{align}

(ii) if $(t,s)\notin \overline{U'_n}$, then we have 
\begin{align}
    f(X,Y;n)\rightarrow -\infty \ \text{in some directions of} \  X^2+Y^2\rightarrow +\infty.
\end{align}

(iii) if $(t,s)\in \partial{U'_n}$, then we have 
\begin{align}
    f(X,Y;n)\rightarrow 0 \  \text{in some directions of} \  X^2+Y^2\rightarrow +\infty.
\end{align}
\end{lemma}
\begin{proof}
When $-p\leq n\leq p-2$, $U_n\neq \emptyset$.  For any $(t,s)\in U_n$, it satisfies $\frac{p+\frac{1}{2}+n}{2p+1}<s<\frac{p+\frac{5}{2}+n}{2p+1}$. It follows that if $(t,s)\in U_n$, then all the inequalities in the eight region of the above analysis for the function $F(X,Y;n)$ 
are satisfied. Hence, for $(t,s)\in U_n$, we have $F(X,Y;n)\rightarrow \infty \ \text{as } \ X^2+Y^2\rightarrow +\infty.$ So Lemma \ref{lemma-Li2}, we obtain (i).  

As to (ii) and (iii), if $(t,s)\notin U_n$, then $(t,s)$ does not satisfy at least one of the inequality in the definition of $U_n$.  For example, suppose $\frac{p+1+n-t}{2p-1}\geq s$, by the case (V) in the previous discuss for the function $F(X,Y;n)$, when $X-Y\leq 0$ and $Y\leq 0$, we have 
\begin{align}
    F(X,Y;n)=\left(\frac{1}{2}-t\right)X+\left(p+\frac{1}{2}+n-(2p-1)s\right)Y.
\end{align}
Let $X=Y\rightarrow -\infty$, then $F(Y,Y;n)=(p+1+n-t-(2p-1)s)Y\rightarrow -\infty$ if $p+1+n-t-(2p-1)s>0$ and 
$F(Y,Y;n)=(p+1+n-t-(2p-1)s)Y\rightarrow 0$ if $p+1+n-t-(2p-1)s=0$. By Lemma \ref{lemma-Li2}, we obtain $f(X,Y;n)\rightarrow -\infty$ or $0$ in the direction $X=Y\rightarrow -\infty$ when $p+1+n-t-(2p-1)s\geq 0$. 
\end{proof}

\subsection{Proof of Lemma \ref{lemma-volumeestimate}}  \label{appendix-2}
In this section, we prove the Lemma \ref{lemma-volumeestimate} which gives the estimation of the critical value 
\begin{align}
        \zeta_{\mathbb{R}}(p)=Re V(p,t_0,s_0), 
\end{align}
where $(t_0,s_0)$ is the unique solution of the equations
\begin{align}  
    \frac{\partial V(p,t,s)}{\partial t}&=-2\pi\sqrt{-1}+3\log(1-e^{2\pi\sqrt{-1}t})\\\nonumber
    &-\log(1-e^{2\pi\sqrt{-1}(t+s)})-\log(1-e^{2\pi\sqrt{-1}(t-s)})=0,
\end{align}
\begin{align} 
    \frac{\partial V(p,t,s)}{\partial s}&=(4p+2)\pi\sqrt{-1}s-(2p+3)\pi\sqrt{-1}\\\nonumber
    &-\log(1-e^{2\pi\sqrt{-1}(t+s)})+\log(1-e^{2\pi\sqrt{-1}(t-s)})=0. 
\end{align}

By expanding the above equations, we obtain the expansions of $t(\gamma)$ and $s(\gamma)$ as follows.
\begin{align}
    t(\gamma)&=\frac{\log(1-2\sqrt{-1})}{2\pi\sqrt{-1}}+1+\frac{(1+2\sqrt{-1})\pi}{40}\gamma^2+\frac{(3+\sqrt{-1})\pi}{80}\gamma^3\\\nonumber
    &+(\frac{180\pi+19\pi^3}{9600}-\frac{45\pi-4\pi^3}{4800}\sqrt{-1})\gamma^4+O(\gamma^5),\\\nonumber
    s(\gamma)&=\frac{1}{2}+\frac{1}{2}\gamma+\frac{1-\sqrt{-1}}{8}\gamma^2-\frac{\sqrt{-1}}{16}\gamma^3\\\nonumber
    &-(\frac{1}{61}+\frac{3+\pi^2}{192}\sqrt{-1})\gamma^4+O(\gamma^5). 
\end{align}

By using Taylor expansion, we obtain 
\begin{align}
 2\pi V(p,t(\gamma),s(\gamma))&=v_8-(p+\frac{23}{4})\pi^2\sqrt{-1}-\pi^2 \sqrt{-1}\gamma-\pi^2\frac{1+\sqrt{-1}}{4}\gamma^2\\\nonumber
 &-\pi^2\frac{1}{8}\gamma^3-\pi^2\frac{6+\pi^2-6\sqrt{-1}}{192}\gamma^4+O(\gamma^5).   
\end{align}
Therefore
\begin{align}
    2\pi \zeta_{\mathbb{R}}(p)=2\pi Re V(p,t(\gamma),s(\gamma))=v_8-\pi^2\frac{1}{4}\gamma^2-\pi^2\frac{1}{8}\gamma^3-\pi^2 \frac{6+\pi^2}{192}\gamma^4+O(\gamma^5).
\end{align}

Actually, we will prove the following stronger estimate which may have independent interests.
\begin{align} \label{estimates}
    2\pi \zeta_{\mathbb{R}}(p)\geq v_8-\frac{25\pi^2}{64}\frac{1}{p^2},
\end{align}
where $v_8$ denotes the volume of the ideal regular octahedron, i.e. $v_8\approx 3.66386$.

For the case $p\leq 100$, one can check the above formula by direct computation.  So in the following, we always assume $p\geq 100$, putting $\gamma=\frac{1}{p}$, we regard $(t_0,s_0)$ as a function of $\gamma$, and the denote it by $(t(\gamma),s(\gamma))$,  then 
\begin{align}
    \zeta_{\mathbb{R}}(p)=Re V(p,t(\gamma),s(\gamma)). 
\end{align}
For convenience, we write $s(\gamma)$ as
\begin{align}
    s(\gamma)&=\frac{1}{2}+\widehat{s}(\gamma)\gamma\\\nonumber
    \widehat{s}(\gamma)&=\widehat{s}_R(\gamma)-\sqrt{-1}\widehat{s}_I(\gamma).
\end{align}

Since $(t(\gamma),s(\gamma))$ satisfies the critical point equations, we have
\begin{align}
\frac{dV(p,t(\gamma),s(\gamma))}{d\gamma}&=V_\gamma(p,t(\gamma),s(\gamma))=2\pi\sqrt{-1}\frac{1}{\gamma^2}s(\gamma)(1-s(\gamma))\\\nonumber
&=2\pi\sqrt{-1}\left(\frac{1}{4\gamma^2}-\widehat{s}(\gamma)\right).  
\end{align}


Then we have
\begin{align}
\frac{d\text{Re}V(p,t(\gamma),s(\gamma))}{d\gamma}=-8\pi^2 \widehat{s}_I(\gamma) \widehat{s}_R(\gamma)>-\frac{25}{32}\pi^2\gamma,   
\end{align}
if $ \widehat{s}_{R}(\gamma)\widehat{s}_{I}(\gamma)<\frac{25}{256}\gamma $. It follows that
\begin{align}
2\pi \text{Re}V(p,t(\gamma),s(\gamma))>v_8-\frac{25\pi^2}{64}\gamma^2
\end{align}
which is just the estimate (\ref{estimates}).

Therefore, we only need to prove
\begin{align} \label{equation-range}
   0.12\gamma
<\widehat{s}_{I}(\gamma)<0.13\gamma,\frac{1}{2}+0.12\gamma
<\widehat{s}_{R}(\gamma)<\frac{1}{2}+0.13\gamma, 
\end{align}
since it follows that
\begin{align}
\widehat{s}_{I}(\gamma)\widehat{s}_{R}(\gamma)<(\frac
{1}{2}+0.13\gamma)0.13\gamma<\frac{25}{256}\gamma. 
\end{align}
 
In the following, we prove Proposition \ref{prop-range} which implies (\ref{equation-range}).

Let $x=e^{2\pi\sqrt{-1}t}$ and $y=e^{2\pi\sqrt{-1}s}$, from the critical point equation, we have 
\begin{align}
  (1-y)(1-y^{2p})^{3}=y^{2p}(1+y)^{2}(1+y^{2p-1}),   
\end{align}
which gives
\begin{align} \label{equationy}
    &\left(  y^{3p+\frac{1}{2}}+y^{-(3p+\frac{1}{2})}\right)  -\left(
y^{3p-\frac{1}{2}}+y^{-(3p-\frac{1}{2})}\right)  -4\left(  y^{p+\frac{1}{2}
}+y^{-(p+\frac{1}{2})}\right) \\\nonumber
&+\left(  y^{p-\frac{1}{2}}+y^{-(p-\frac{1}{2}
)}\right)  -\left(  y^{p-\frac{3}{2}}+y^{-(p-\frac{3}{2})}\right)=0.
\end{align}

The real and imaginary parts of (\ref{equationy}) give
\begin{align} \label{equation-realpart}
&\sinh((6+\gamma)\pi\widehat{\widehat{s}}_{I}(\gamma))\cos
((6+\gamma)\pi\widehat{\widehat{s}}_{R}(\gamma))+\sinh((6-\gamma)\pi
\widehat{\widehat{s}}_{I}(\gamma))\cos((6-\gamma)\pi\widehat{\widehat{s}}%
_{R}(\gamma))\\\nonumber
&-4\sinh((2+\gamma)\pi\widehat{\widehat{u}}_{I}(\gamma))\cos((2+\gamma
)\pi\widehat{\widehat{u}}_{R}(\gamma))-\sinh((2-\gamma)\pi\widehat{\widehat
{u}}_{I}(\gamma))\cos((2-\gamma)\pi\widehat{\widehat{u}}_{R}(\gamma))\\\nonumber
&-\sinh((2-3\gamma)\pi\widehat{\widehat{u}}_{I}(\gamma))\cos((2-3\gamma
)\pi\widehat{\widehat{u}}_{R}(\gamma))=0
\end{align}
and 
\begin{align} \label{equation-imaginary}
&\cosh((6+\gamma)\pi\widehat{\widehat{s}}_{I}(\gamma
))\sin((6+\gamma)\pi\widehat{\widehat{s}}_{R}(\gamma))+\cosh((6-\gamma
)\pi\widehat{\widehat{s}}_{I}(\gamma))\sin((6-\gamma)\pi\widehat{\widehat{s}%
}_{R}(\gamma))\\\nonumber 
&-4\cosh((2+\gamma)\pi\widehat{\widehat{s}}_{I}(\gamma))\sin((2+\gamma
)\pi\widehat{\widehat{s}}_{R}(\gamma))-\cosh((2-\gamma)\pi\widehat{\widehat
{s}}_{I}(\gamma))\sin((2-\gamma)\pi\widehat{\widehat{s}}_{R}(\gamma))\\\nonumber
&-\cosh((2-3\gamma)\pi\widehat{\widehat{s}}_{I}(\gamma))\sin((2-3\gamma
)\pi\widehat{\widehat{s}}_{R}(\gamma))=0.
\end{align}

First, we consider the real part. Let
\begin{align}
    F(x,y,\gamma)&=\sinh((6+\gamma)\pi x)\cos((6+\gamma)\pi y)+\sinh((6-\gamma)\pi
x)\cos((6-\gamma)\pi y)\\\nonumber
&-4\sinh((2+\gamma)\pi x)\cos((2+\gamma)\pi y)-\sinh((2-\gamma)\pi
x)\cos((2-\gamma)\pi y)\\\nonumber
&-\sinh((2-3\gamma)\pi x)\cos((2-3\gamma)\pi y).
\end{align}

First, we establish the following estimations. 
\begin{lemma} \label{Lemmafive}
For $\gamma\leq\gamma_{0}=\frac{1}{100}$ and $c\in [0.12,0.13]$, 
we have the following estimates:
\begin{align} \label{identity1}
 &\pi c\gamma(6+\gamma+(36c^{2}-4.4814449672-0.002684392992\pi^{2}c^{2})\pi
^{2}\gamma^{2})\\\nonumber
&<\sinh(\pi(6c\gamma+c\gamma^{2}))\cos(\pi(1.22\gamma
+0.12\gamma^{2}));   
\end{align}
\begin{align}
  &\pi c\gamma(6-\gamma+(\frac{5.99^{3}}{6}c^{2}-0.1452)\pi^{2}\gamma^{2}%
)\\\nonumber
&<\sinh(\pi(6c\gamma-c\gamma^{2}))\cos(\pi(0.22\gamma-0.12\gamma^{2}));  
\end{align}
\begin{align} \label{identity3}
&\sinh(\pi(2c\gamma+c\gamma^{2}))\cos(\pi(0.74\gamma+0.12\gamma^{2}))\\\nonumber
&<\pi
c\gamma(2+\gamma+(1.35411021675c^{2}-0.5476)\pi^{2}\gamma^{2});
\end{align}
\begin{align}
&\sinh(\pi(2c\gamma-c\gamma^{2}))\cos(\pi(0.26\gamma+0.12\gamma^{2}))\\\nonumber
&<\pi
c\gamma(2-\gamma+(1.334c^{2}-0.0676)\pi^{2}\gamma^{2});
\end{align}
\begin{align}
& \sinh(\pi(2c\gamma-3c\gamma^{2}))\cos(\pi(1.26\gamma+0.36\gamma^{2}))\\\nonumber
&<\pi
c\gamma(2-3\gamma+(1.334c^{2}-1.57272192\\\nonumber
&+0.0003050686030133\pi
-0.00010589292\pi^{2}c^{2})\pi^{2}\gamma^{2}).  
\end{align}
\end{lemma}
\begin{proof}
We only provide the proof of (\ref{identity1}) and (\ref{identity3}), the others can be proved similarly.  Note that 
$\sinh x\geq x+\frac{1}{6}x^{3}$, hence
\begin{align}
\pi c\gamma\left(  6+\gamma+36\pi^{2}c^{2}\gamma^{2}\right)  &=\pi(6c\gamma+c\gamma^{2})+\frac{\pi^{3}(6c\gamma)^{3}}{6}\\\nonumber
&\leq\pi(6c\gamma+c\gamma^{2})+\frac{\pi^{3}(6c\gamma+c\gamma^{2})^{3}}{6}\\\nonumber
&\leq\sinh(\pi(6c\gamma+c\gamma^{2})).
\end{align}

Note that $\cos x\geq 1-\frac{x^{2}}{2}$, hence

\begin{align}
&1-\frac{148.84}{200}\pi^{2}\gamma^{2}-\frac{2929.44}{20000}\pi^{2}\gamma^{3}\\\nonumber
&=1-\frac{1}{200}\pi^{2}\gamma^{2}(148.84+29.28\gamma+\frac{1.44}{100}%
\gamma)\\\nonumber
&\leq1-\frac{1}{200}\pi^{2}\gamma^{2}(148.84+29.28\gamma+1.44\gamma
^{2})\\\nonumber
&=1-\frac{1}{2}\pi^{2}(1.22\gamma+0.12\gamma^{2})^{2}\\\nonumber
&<\cos(\pi(1.22\gamma+0.12\gamma^{2})).
\end{align}
Then we obtain

\begin{align}
&\pi c\gamma(6+\gamma+(36c^{2}-4.4814449672-0.002684392992\pi^{2}c^{2})\pi
^{2}\gamma^{2})\\\nonumber
&=\pi c\gamma(6+\gamma+36\pi^{2}c^{2}\gamma^{2}-\frac{446.52}{100}\pi
^{2}\gamma^{2}-\frac{16230.32}{1000000}\pi^{2}\gamma^{2}\\\nonumber
&-\frac{1339.56}%
{500000}\pi^{4}c^{2}\gamma^{2}-\frac{2929.44}{200000000}\pi^{2}\gamma
^{2}-\frac{26364.96}{5000000000}\pi^{4}c^{2}\gamma^{2})\\\nonumber
&\leq\pi c\gamma(6+\gamma+36\pi^{2}c^{2}\gamma^{2}-\frac{446.52}{100}\pi
^{2}\gamma^{2}-\frac{16230.32}{10000}\pi^{2}\gamma^{3}\\\nonumber
&-\frac{1339.56}{50}%
\pi^{4}c^{2}\gamma^{4}-\frac{2929.44}{20000}\pi^{2}\gamma^{4}-\frac
{26364.96}{5000}\pi^{4}c^{2}\gamma^{5})\\\nonumber
&=\pi c\gamma\left(  6+\gamma+36\pi^{2}c^{2}\gamma^{2}\right)  (1-\frac
{148.84}{200}\pi^{2}\gamma^{2}-\frac{2929.44}{20000}\pi^{2}\gamma^{3})\\\nonumber
&<\sinh(\pi(6c\gamma+c\gamma^{2}))\cos(\pi(1.22\gamma+0.12\gamma^{2})).
\end{align}
We finish the proof of (\ref{identity1}).

For $0.12\leq c\leq0.13$ and $0\leq\gamma\leq\frac{1}{100}$, we have
\begin{align}
&\pi(2c\gamma+c\gamma^{2})<\pi(2\times0.13\frac{1}{100}+0.13\frac{1}%
{10000})\\\nonumber
&=\pi(\frac{2.6}{1000}+\frac{2.6}{200000})=0.00820898<0.031621459.
\end{align}
and by 
$\sinh x=x+\frac{\cosh\xi}{6}x^{3}<x+\frac{\cosh x}{6}x^{3}<x+\frac{667}%
{4000}x^{3}$ for $0<\xi<x$ and $\cosh x<\frac{2001}{2000}$ if $0<x<0.031621459$.
Then
\begin{align}
&\pi(2c\gamma+c\gamma^{2})\leq\sinh(\pi(2c\gamma+c\gamma^{2}))<\pi
(2c\gamma+c\gamma^{2})+\frac{667}{4000}\pi^{3}(2c\gamma+c\gamma^{2})^{3}\\\nonumber
&\leq\pi(2c\gamma+c\gamma^{2})+\frac{667}{4000}\pi^{3}(2c\gamma+\frac{1}%
{100}c\gamma)^{3}=\pi c\gamma(2+\gamma+\frac{667\times2.01^{3}}{4\times10^{3}%
}\pi^{2}c^{2}\gamma^{2})
\end{align}

For $0.12\leq c\leq0.13$ and $0\leq\gamma\leq\frac{1}{100}$, we have

$\pi(0.74\gamma+0.12\gamma^{2})\leq\pi(0.74\frac{1}{100}+0.12\frac{1}%
{10000})=\pi(\frac{7.4}{1000}+\frac{1.2}{100000})=0.0232855<0.03$,
and by 
$\cos x=1-\frac{x^{2}}{2}+\frac{\sin\xi}{6}x^{3}<1-\frac{x^{2}}{2}+\frac{\sin
x}{6}x^{3}<1-\frac{x^{2}}{2}+\frac{1}{200}x^{3}$ for $0<\xi<x$ and $\sin x<0.03$ if $0<x<0.03$.

Then we have
$\cos(\pi(0.74\gamma+0.12\gamma^{2}))$

$<1-\frac{1}{2}\pi^{2}(0.74\gamma+0.12\gamma^{2})^{2}+\frac{1}{200}\pi
^{3}(0.74\gamma+0.12\gamma^{2})^{3}=1-\frac{1}{200}\pi^{2}\gamma
^{2}(54.76+17.76\gamma+1.44\gamma^{2})+\frac{1}{2\times10^{5}}\pi^{3}%
\gamma^{3}(7.4+1.2\gamma)^{3}$

$<1-\frac{1}{200}\pi^{2}\gamma^{2}(54.76+17.76\gamma)+\frac{1}{2\times10^{5}%
}\pi^{3}\gamma^{3}(7.4+\frac{1.2}{100})^{3}=1-\frac{54.76}{200}\pi^{2}%
\gamma^{2}-(\frac{8.88}{100}\pi^{2}-\frac{7.412^{3}}{2\times10^{5}}\pi
^{3})\gamma^{3}$

Hence, we have
$\sinh(\pi(2c\gamma+c\gamma^{2}))\cos(\pi(0.74\gamma+0.12\gamma^{2}))$

$<(\pi(2c\gamma+c\gamma^{2})+\frac{667\times201^{3}}{4\times10^{9}}\pi
^{3}c^{3}\gamma^{3})(1-\frac{54.76}{200}\pi^{2}\gamma^{2}-(\frac{8.88}{100}%
\pi^{2}-\frac{7.412^{3}}{2\times10^{5}}\pi^{3})\gamma^{3})$

$=\pi c\gamma(2+\gamma+\frac{667\times2.01^{3}}{4000}\pi^{2}c^{2}\gamma
^{2}-\frac{54.76}{100}\pi^{2}\gamma^{2}-\frac{54.76}{200}\pi^{2}\gamma
^{3}-2(\frac{8.88}{100}\pi^{2}-\frac{7.412^{3}}{2\times10^{5}}\pi^{3}%
)\gamma^{3}$

$-\frac{667\times54.76\times201^{3}}{8\times10^{11}}\pi^{4}c^{2}\gamma
^{4}-(\frac{8.88}{100}\pi^{2}-\frac{741.2^{3}}{2\times10^{11}}\pi^{3}%
)\gamma^{4}-\frac{667\times201^{3}}{4\times10^{9}}(\frac{8.88}{100}\pi
^{2}-\frac{7.412^{3}}{2\times10^{5}}\pi^{3})\pi^{2}c^{2}\gamma^{5})$

$<\pi c\gamma(2+\gamma+(1.35411021675c^{2}-0.5476)\pi^{2}\gamma^{2})$
which is (\ref{identity3}).

\end{proof}

Now, by using Lemma \ref{Lemmafive}, we have 
\begin{lemma} \label{lemma-F}
(1)$F(x,y,\gamma)<0$ on the edge $0.12\gamma\leq x\leq0.13\gamma,y=\frac{1}%
{2}+0.12\gamma$ of a rectangle $[0.12\gamma,0.13\gamma]\times\lbrack\frac
{1}{2}+0.12\gamma,\frac{1}{2}+0.13\gamma]$ (acutally a square) for small
$\gamma\leq\frac{1}{100}$;

(2) $F(x,y,\gamma)>0$ on the edge $0.12\gamma\leq x\leq0.13\gamma,y=\frac
{1}{2}+0.13\gamma$ of a rectangle $[0.12\gamma,0.13\gamma]\times\lbrack
\frac{1}{2}+0.12\gamma,\frac{1}{2}+0.13\gamma]$ (acutally a square) for small
$\gamma\leq\frac{1}{100}$.
\end{lemma}
\begin{proof}
We only prove (1). The case (2) can be proved similarly. 

   By the definition of $F(x,y,\gamma)$, we have

\begin{align}
 &F(c\gamma,\frac{1}{2}+0.12\gamma,\gamma)\\\nonumber
 &=-\sinh(\pi(6c\gamma+c\gamma^{2}))\cos(\pi(1.22\gamma+0.12\gamma^{2}%
))-\sinh(\pi(6c\gamma-c\gamma^{2}))\cos(\pi(0.22\gamma-0.12\gamma^{2}))\\\nonumber
&+4\sinh(\pi(2c\gamma+c\gamma^{2}))\cos(\pi(0.74\gamma+0.12\gamma^{2}%
))+\sinh(\pi(2c\gamma-c\gamma^{2}))\cos(\pi(0.26\gamma+0.12\gamma^{2}%
))\\\nonumber
&+\sinh(\pi(2c\gamma-3c\gamma^{2}))\cos(\pi(1.26\gamma+0.36\gamma^{2}))
\end{align}
By using the estimates in Lemma \ref{Lemmafive}, we have

\begin{align}
    &F(c\gamma,\frac{1}{2}+0.12\gamma,\gamma)<-\pi c\gamma(6+\gamma+(36c^{2}-4.4814449672-0.002684392992\pi^{2}c^{2}%
)\pi^{2}\gamma^{2})\\\nonumber
&-\pi c\gamma(6-\gamma+(\frac{5.99^{3}}{6}c^{2}%
-0.1452)\pi^{2}\gamma^{2})+4\pi c\gamma(2+\gamma+(1.35411021675c^{2}%
-0.5476)\pi^{2}\gamma^{2})\\\nonumber
&+\pi c\gamma(2-\gamma+(1.334c^{2}-0.0676)\pi^{2}\gamma^{2})+\pi
c\gamma(2-3\gamma+(1.334c^{2}-1.57272192\\\nonumber
&+0.0003050686030133\pi
-0.00010589292\pi^{2}c^{2})\pi^{2}\gamma^{2})\\\nonumber
&=(0.7959230472+0.0003050686030133\pi+(0.002578500072\pi^{2}%
-27.915559133-\frac{5.99^{3}}{6})c^{2})\pi^{3}c\gamma^{3}
\end{align}
which is less than $0$ at $c\in [0.12,0.13]$ by straightforward computations.  
\end{proof}

Similarly, for the imaginary part,  we introduce the function
\begin{align}
   G(x,y,\gamma)&=\cosh((6+\gamma)\pi x)\sin((6+\gamma)\pi y)+\cosh((6-\gamma)\pi
x)\sin((6-\gamma)\pi y)\\\nonumber
&-4\cosh((2+\gamma)\pi x)\sin((2+\gamma)\pi y)\\\nonumber
&-\cosh((2-\gamma)\pi x)\sin((2-\gamma)\pi y)-\cosh((2-3\gamma)\pi
x)\sin((2-3\gamma)\pi y).
\end{align}
we also have
\begin{lemma} \label{lemma-G}
$G(x,y,\gamma)>0$ on the edge $x=0.12\gamma,\frac{1}{2}+0.12\gamma\leq y\leq\frac
{1}{2}+0.13\gamma$ of a rectangle $[0.12\gamma,0.13\gamma]\times\lbrack
\frac{1}{2}+0.12\gamma,\frac{1}{2}+0.13\gamma]$ (acutally a square) for small
$\gamma\leq\frac{1}{100}$;
$G(x,y,\gamma)<0$ on the edge $x=0.13\gamma,\frac{1}{2}+0.12\gamma\leq y\leq\frac
{1}{2}+0.13\gamma$ of a rectangle $[0.12\gamma,0.13\gamma]\times\lbrack
\frac{1}{2}+0.12\gamma,\frac{1}{2}+0.13\gamma]$ (acutally a square) for small
$\gamma\leq\frac{1}{100}$.
\end{lemma}

\begin{proposition} \label{prop-range}
Equations (\ref{equation-realpart}) and (\ref{equation-imaginary}) has a unique solution $(\widehat{\widehat{s}}_{I}(\gamma),\widehat{\widehat
{s}}_{R}(\gamma))$ in the region of $(\widehat{\widehat{s}}_{I}(\gamma
),\widehat{\widehat{s}}_{R}(\gamma))\in\lbrack0.12\gamma,0.13\gamma
]\times\lbrack\frac{1}{2}+0.12\gamma,\frac{1}{2}+0.13\gamma]$ for small
$\gamma\leq\frac{1}{100}$.
\end{proposition}
\begin{proof}
For the existence, by Lemma \ref{lemma-F} and \ref{lemma-G}, we use the Poincare-Miranda Theorem
(generalization of Intermediate value theorem),
thus there exists a solution of $(\widehat{\widehat{s}}_{I}(\gamma
),\widehat{\widehat{s}}_{R}(\gamma))$ in the region of $(\widehat{\widehat{s}%
}_{I}(\gamma),\widehat{\widehat{s}}_{R}(\gamma))\in\lbrack0.12\gamma
,0.13\gamma]\times\lbrack\frac{1}{2}+0.12\gamma,\frac{1}{2}+0.13\gamma]$ for
small $\gamma\leq\frac{1}{100}$. 
From the range of $\widehat{\widehat{s}}_R(\gamma)$, and by Proposition \ref{prop-critcalequation}, it follows that such  solution is unique in this range. 

\end{proof}

\subsection{Proof of Proposition \ref{prop-saddleonedim}}  \label{appendix-onesaddle}
For a fixed constant $c\in \mathbb{R}$, we define the subset 
\begin{align}
    D_0(c)=\{(t,s)\in D_0| s=c\}.
\end{align}
We prove Proposition \ref{prop-saddleonedim} by proving Proposition \ref{prop-saddleonedim1} and Proposition \ref{prop-saddleonedim2} in the following.
\begin{proposition} \label{prop-saddleonedim1}
    For $c_{upper}(p)\leq c<1$ and $n\in \mathbb{Z}$, there exists a constant $C$ independent of $c$, such that
    \begin{align}
        |\int_{D_0(c)} \psi(t,c)\sin(2\pi c)e^{(N+\frac{1}{2})V_N(p,t,c;0,n)}dt|<Ce^{(N+\frac{1}{2})\left(\zeta_{\mathbb{R}}(p)-\epsilon\right)}. 
    \end{align}
\end{proposition}
$D_{0}(c)$ is a slice of the region $D_0$, we will prove Proposition \ref{prop-saddleonedim} by using the saddle point method on $D_0(c)$.
Recall that
\begin{align}
    &V(p,t,s;m,n)=\pi \sqrt{-1}\left((2p+1)s^2-(2p+3+2n)s-(2m+2)t\right)\\\nonumber
    &+\frac{1}{2\pi\sqrt{-1}}\left(\text{Li}_2(e^{2\pi\sqrt{-1}(t+s)})+\text{Li}_2(e^{2\pi\sqrt{-1}(t-s)})-3\text{Li}_2(e^{2\pi\sqrt{-1}t})+\frac{\pi^2}{6}\right),
\end{align}
so we have
\begin{align}
    \frac{\partial V(p,t,s; 0,n)}{\partial t}&=-2\pi\sqrt{-1}+3\log(1-e^{2\pi\sqrt{-1}t})\\\nonumber
    &-\log(1-e^{2\pi\sqrt{-1}(t+s)})-\log(1-e^{2\pi\sqrt{-1}(t-s)}),
\end{align}
and
\begin{align}
    \frac{\partial V(p,t,s; 0,n)}{\partial s}&=(4p+2)\pi\sqrt{-1}s-(2p+3+2n)\pi\sqrt{-1}\\\nonumber
    &-\log(1-e^{2\pi\sqrt{-1}(t+s)})+\log(1-e^{2\pi\sqrt{-1}(t-s)}).
\end{align}

\begin{proposition} \label{prop-critical1}
    Fixing $s=c\in [\frac{1}{2},\frac{3}{4})$, as a function of $t$,  $V(p,t,c;0,n)$ has a unique critical point
    $T_1(c)$ with $t_1(c)=\text{Re}(T_1(c))\in (\frac{1}{2},1)$.     
    \end{proposition} 
\begin{proof}
Consider the equation
\begin{align} \label{formula-equationonedim}
    \frac{dV(p,t,c;0,n)}{dt}&=-2\pi \sqrt{-1}+3\log(1-e^{2\pi\sqrt{-1}t})\\\nonumber
    &-\log(1-e^{2\pi\sqrt{-1}(t+c)})-\log(1-e^{2\pi\sqrt{-1}(t-c)})=0
\end{align}
which gives 
\begin{align}
    x^2-2x+3-C-\frac{1}{C}=0
\end{align}
where $x=e^{2\pi\sqrt{-1}t}$, $C=e^{2\pi\sqrt{-1}c}$.

So we obtain 
\begin{align}
    x=1\pm 2\sqrt{-1}\sin(\pi c). 
\end{align}
Let $T_\pm (c)$ be the solution determined by the equation 
\begin{align} \label{formula-equaitonT-1c}
    e^{2\pi\sqrt{-1}T_\pm (c)}=1\pm 2\sqrt{-1}\sin(\pi c). 
\end{align}
 From (\ref{formula-equaitonT-1c}), we have
\begin{align}
    T_\pm(c)=\frac{\log(1\pm 2\sqrt{-1}\sin(\pi c))}{2\pi\sqrt{-1}} +\mathbb{Z}.
\end{align}
Then 
\begin{align}
    \text{Re}(T_\pm(c))=\frac{\text{arg}(1\pm 2\sqrt{-1}\sin(\pi c))}{2\pi}+\mathbb{Z}. 
\end{align}

By $c\in [\frac{1}{2},1)$, we obtain 
\begin{align}
  0<\text{arg}(1+2\sqrt{-1}\sin(\pi c))<\arctan(2)<1.2, \\\nonumber
  -1.2<-\arctan(2)<\text{arg}(1-2\sqrt{-1}\sin(\pi c))<0. 
\end{align}
Therefore, one can see that only the solution $T_-(c)=\frac{\log(1-2\sqrt{-1}\sin(\pi c))}{2\pi\sqrt{-1}}+1$ satisfies that 
$\text{Re}(T_-(c))\in (\frac{1}{2},1)$. Moreover, by the following Lemma \ref{lemma-T1(c)iscritical}, we know that $T_-(c)$ satisfies the equation (\ref{formula-equationonedim}), so $T_-(c)$ is indeed a critical point of $V(p,t,c;0,n)$.  
In the following, we will denote $T_-(c)$ by $T_1(c)=t_1(c)+\sqrt{-1}X_1(c)$  for convenience.   
\end{proof}

\begin{lemma} \label{lemma-T1(c)iscritical}
$T_-(c)=\frac{\log(1-2\sqrt{-1}\sin(\pi c))}{2\pi\sqrt{-1}}+1$ satisfies the equation (\ref{formula-equationonedim}).
\end{lemma}
\begin{proof}
   The equation (\ref{formula-equationonedim}) is equivalent to 
   \begin{equation} 
\left\{ \begin{aligned}
        &x^2-2x+3-C-\frac{1}{C}=0, \\
                  &3\text{arg}(1-x)-\text{arg}(1-Cx)-\text{arg}(1-C^{-1}x)=2\pi,
                          \end{aligned} \right.
                          \end{equation}
where $x=e^{2\pi\sqrt{-1}t}$, $C=e^{2\pi\sqrt{-1}c}$.

Clearly, $x_0=e^{2\pi\sqrt{-1}T_-(c)}=1-2\sqrt{-1}\sin(\pi c)$ satisfies the first equation, and we have the equation
\begin{align} \label{formula-arguementequ}
    3\text{arg}(1-x_0)-\text{arg}(1-Cx_0)-\text{arg}(1-C^{-1}x_0)=2k\pi
\end{align}
for some $k\in \mathbb{Z}$.  In the following, we show $k=1$. Indeed, for $c\in [\frac{1}{2},1)$, we have
\begin{align}
    3\text{arg}(1-x_0)&=3\text{arg}(2\sqrt{-1}\sin (\pi c))=\frac{3}{2}\pi.
    \end{align}
Since   
    \begin{align}
        &1-Cx_0\\\nonumber
        &=1-2\sin(\pi c)\sin(2\pi c)-\cos(2\pi c)+\sqrt{-1}(2\sin(\pi c)\cos(2\pi c)-\sin(2\pi c))\\\nonumber
    &=2\sin^2(\pi c)(1-2\cos(\pi c))+\sqrt{-1}(2\sin(\pi c)(\cos(2\pi c)-\cos(\pi c))),
    \end{align}
and $2\sin^2(\pi c)(1-2\cos(\pi c))>0$ and $\cos(\pi c)-\sin(\pi c)\leq \cos(2\pi c)-\sin(2\pi c)$ for $c\in [\frac{1}{2},1)$, it follows that 
\begin{align}
\text{arg}(1-Cx_0)\in [-\frac{\pi }{4}, \frac{\pi}{2}).    
\end{align}    
Since    \begin{align}
    &(1-C^{-1}x_0)\\\nonumber
    &=2\sin^2(\pi c)(2\cos(\pi c)+1)+\sqrt{-1}(2\sin(\pi c)(\cos(2\pi c)+\cos(\pi c)))
\end{align}
and $2\sin(\pi c)(\cos(2\pi c)+\cos(\pi c))<0$ for $c\in [\frac{1}{2},1)$, it follows that 
\begin{align}
\text{arg}(1-C^{-1}x_0)\in (-\pi,0).     
\end{align}
Therefore,  we obtain  
\begin{align}
    3\text{arg}(1-x_0)-\text{arg}(1-Cx_0)-\text{arg}(1-C^{-1}x_0)\in (\pi,\frac{11}{4}\pi), 
\end{align}
which implies  $k=1$ in formula (\ref{formula-arguementequ}).  
\end{proof}

\begin{lemma}
We have the following identities:
\begin{align}
    \text{Re}\left(\log\left(1-e^{2\pi\sqrt{-1}(T_1(c)+c)}\right)\right)&=\log\left(4\sin(\pi c)\sin\left(\frac{\pi c}{2}\right)\right), \label{formula-iden1}\\
     \text{Re}\left(\log\left(1-e^{2\pi\sqrt{-1}(T_1(c)-c)}\right)\right)&=\log\left(4\sin(\pi c)\cos\left(\frac{\pi c}{2}\right)\right), \label{formula-iden2}\\
     \text{Re}\left(-\log\left(1-e^{2\pi\sqrt{-1}(T_1(c)+c)}\right)\right.&\left.+\log\left(1-e^{2\pi\sqrt{-1}(T_1(c)-c)}\right)\right)&=\log\left(\cot\left(\frac{\pi c}{2}\right)\right). \label{formula-iden3}
\end{align}
\end{lemma}
\begin{proof}
By straightforward computations, we obtain
\begin{align*}
    &\text{Re}\left(\log\left(1-e^{2\pi\sqrt{-1}(T_1(c)+c)}\right)\right)\\\nonumber
    &=\text{Re}\left(\log\left(1-(1- 2\sqrt{-1}\sin(\pi c))e^{2\pi\sqrt{-1}c}\right)\right)\\\nonumber
    &=\text{Re}\left(\log(1-2\sin(\pi c)\sin(2\pi c)-\cos(2\pi c)+\sqrt{-1}(2\sin(\pi c)\cos(2\pi c))-\sin(2\pi c))\right)\\\nonumber
    &=\frac{1}{2}\log\left(1-2\sin(\pi c)\sin(2\pi c)-\cos(2\pi c)^2+(2\sin(\pi c)\cos(2\pi c))-\sin(2\pi c))^2\right)\\\nonumber
    &=\log\left(4\sin(\pi c)\sin\left(\frac{\pi c}{2}\right)\right)
\end{align*}
and
\begin{align*}
    &\text{Re}\left(\log\left(1-e^{2\pi\sqrt{-1}(T_1(c)-c)}\right)\right)\\\nonumber
    &=\text{Re}\left(\log\left(1-(1- 2\sqrt{-1}\sin(\pi c))e^{-2\pi\sqrt{-1}c}\right)\right)\\\nonumber
    &=\text{Re}\left(\log(1+2\sin(\pi c)\sin(2\pi c)-\cos(2\pi c)+\sqrt{-1}(2\sin(\pi c)\cos(2\pi c))+\sin(2\pi c))\right)\\\nonumber
    &=\frac{1}{2}\log\left(1+2\sin(\pi c)\sin(2\pi c)-\cos(2\pi c)^2+(2\sin(\pi c)\cos(2\pi c))+\sin(2\pi c))^2\right)\\\nonumber
    &=\log\left(4\sin(\pi c)\cos\left(\frac{\pi c}{2}\right)\right),
\end{align*}
which prove the identities (\ref{formula-iden1}) and (\ref{formula-iden2}). Then the identity (\ref{formula-iden3}) follows from the identities (\ref{formula-iden1}) and (\ref{formula-iden2})  immediately.
\end{proof}

\begin{lemma} \label{lemma-ReVTc}
    As a function of $c\in [\frac{1}{2},\frac{3}{4})$, $\text{Re} V(p,T_1(c),c;0,n)$ is a decreasing function of $c$. Furthermore, we have 
\begin{align}
    \text{Re} V(p,T_1(c),c;0,n)=2\left(\Lambda\left(\frac{c}{2}\right)+\Lambda\left(\frac{1}{2}-\frac{c}{2}\right)\right).
\end{align}
\end{lemma}

\begin{proof}
From the equation (\ref{formula-equaitonT-1c}), we obtain 
\begin{align}
    \frac{dT_1(c)}{dc}=\frac{-\cos(\pi c)}{e^{2\pi \sqrt{-1}T_1(c)}}=\frac{-\cos(\pi c)}{1-2\sqrt{-1}\sin(\pi c)}.
\end{align}
Then 
\begin{align}
    &\frac{d \text{Re} V(p,T_1(c),c;0,n)}{dc}\\\nonumber
    &=\text{Re}\left(\frac{\partial V(p,T_1(c),c;0,n)}{\partial t}\frac{dT_1(c)}{dc}+\frac{\partial V(p,T_1(c),c;0,n)}{\partial s}\frac{dc}{dc}\right)\\\nonumber
    &=\text{Re}\left(\frac{\partial V(p,T_1(c),c;0,n)}{\partial s}\right)\\\nonumber
    &=\text{Re}\left(-\log(1-e^{2\pi\sqrt{-1}(T_1(c)+c)})+\log\left(1-e^{2\pi\sqrt{-1}(T_1(c)-c)}\right)\right).
\end{align}
By identity (\ref{formula-iden3}), we obtain 
\begin{align}
    \frac{d \text{Re} V(p,T_1(c),c;0,n)}{dc}=\log\left(\cot\left(\frac{\pi c}{2}\right)\right)\leq 0,
\end{align}
since $\cot\left(\frac{\pi c}{2}\right)\leq 1$ for $c\in [\frac{1}{2},\frac{3}{4})$. Hence, $\text{Re} V(p,T_1(c),c;0,n)$ is a decreasing function. 

For $c\geq \frac{1}{2}$, we have
\begin{align}
    &\text{Re} V(p,T_1(c),c;0,n)-\text{Re} V\left(p,T_1\left(\frac{1}{2}\right),\frac{1}{2};0,n\right)\\\nonumber
    &=\int_{\frac{1}{2}}^{c}\log \left(\cot\left(\frac{\pi \tau}{2}\right)\right)d\tau\\\nonumber
    &=\int_{\frac{1}{2}}^{c}\log \left(2\cos\left(\frac{\pi \tau}{2}\right)\right)d\tau-\int_{\frac{1}{2}}^{c}\log \left(2\sin\left(\frac{\pi \tau}{2}\right)\right)d\tau. 
\end{align}
Let $x=\frac{1}{2}(1-\tau)$, we obtain 
\begin{align}
    \int_{\frac{1}{2}}^{c}\log \left(2\cos\left(\frac{\pi \tau}{2}\right)\right)d\tau =-2\int_{\frac{1}{4}}^{\frac{1}{2}(1-c)}\log(2\sin \pi x)dx=2\left(\Lambda\left(\frac{1}{2}-\frac{c}{2}\right)-\Lambda\left(\frac{1}{4}\right)\right).
\end{align}
Let $y=\frac{\tau}{2}$, we obtain
\begin{align}
    \int_{\frac{1}{2}}^{c}\log \left(2\sin\left(\frac{\pi \tau}{2}\right)\right)d\tau=2\int_{\frac{1}{4}}^{\frac{c}{2}}\log(2\sin (\pi y))dy=-2\left(\Lambda\left(\frac{c}{2}\right)-\Lambda\left(\frac{1}{4}\right)\right).
\end{align}
Hence, for $c\in [\frac{1}{2},\frac{3}{4})$, we have 
\begin{align}
     \text{Re} V(p,T_1(c),c;0,n)=2\left(\Lambda\left(\frac{c}{2}\right)+\Lambda\left(\frac{1}{2}-\frac{c}{2}\right)\right).
\end{align}
\end{proof}
Let $(t_0,s_0)=(t_{0R}+X_0\sqrt{-1},s_{0R}+Y_0\sqrt{-1})$ be a critical point of the potential function $V(p,t,s)$ as given in  Proposition \ref{prop-critical}.  

By the proof of Lemma \ref{lemma-volumeestimate} in Appendix \ref{appendix-2}, we have $0<\zeta_{\mathbb{R}}(p)<\frac{v_8}{2\pi}$.  We assume $c(p)$ is a solution to the following equation
\begin{align}
    Re V(p,T_1(c),c)=2 \left(\Lambda\left(\frac{c}{2}\right)+\Lambda\left(\frac{1}{2}-\frac{c}{2}\right)\right)=\zeta_{\mathbb{R}}(p).
\end{align}

\begin{lemma}
We have the following inequality:
\begin{align}  \label{formula-cleqcupp}
c(p)<c_{upper}(p),
\end{align}
where 
\begin{align} 
    c_{upper}(p)=\frac{1}{\pi}(v_8-2\pi\zeta_{\mathbb{R}}(p))^{\frac{1}{2}}+\frac{1}{2}.
\end{align}
\end{lemma}
\begin{proof}
Let $h(c)=\Lambda(\frac{c}{2})+\Lambda(\frac{1}{2}-\frac{c}{2})$, then
$2h(c(p))=\zeta_{\mathbb{R}}(p)$. 
 By Lemma \ref{lemma-convergent}, we have the following convergent power series
\begin{align}
    h(c)=2\Lambda\left(\frac{1}{4}\right)-\frac{\pi}{4}\left(c-\frac{1}{2}\right)^2-\frac{\pi^3}{48}\left(c-\frac{1}{2}\right)^4-\frac{\pi^5}{144}\left(c-\frac{1}{2}\right)^6-\cdots. 
\end{align}
Then we obtain 
\begin{align}
    2h(c_{upper}(p))&=\zeta_{\mathbb{R}}(p)-2\left(\frac{\pi^3}{48}\left(c_{upper}(p)-\frac{1}{2}\right)^4+\frac{\pi^5}{144}\left(c_{upper}(p)-\frac{1}{2}\right)^6+\cdots\right)\\\nonumber
    &<\zeta_{\mathbb{R}}(p)=2 h(c(p)).
\end{align}
Hence, by Lemma \ref{lemma-ReVTc}, we have $c(p)<c_{upper}(p)$. 
\end{proof}

As a consequence of formula (\ref{formula-cleqcupp}), we have
\begin{corollary} \label{coro-ReV1}
   For $c>c_{upper}(p)$, we have 
   \begin{align*}
   \text{Re} V(p,T_1(c),c)<\text{Re} V(p,T_1(c_{upper}(p)),c_{upper}(p))<\text{Re} V(p,T_1(c(p)),c(p))=\zeta_{\mathbb{R}}(p).
   \end{align*}
\end{corollary}

\begin{lemma} \label{lemma-convergent}
For $c\in [\frac{1}{2},\frac{3}{4})$, we have the following convergent power series 
\begin{align}
    h(c)=2\Lambda\left(\frac{1}{4}\right)-\frac{\pi}{4}\left(c-\frac{1}{2}\right)^2-\frac{\pi^3}{48}\left(c-\frac{1}{2}\right)^4-\frac{\pi^5}{144}\left(c-\frac{1}{2}\right)^6-\cdots.
\end{align}
\end{lemma}
\begin{proof}
  We use the following power series expansion for $\sec(x)$,
\begin{align} \label{formula-sec}
    \sec(x)=1+\frac{1}{2}x^2+\frac{5}{24}x^4+\cdots, \ \text{for} \ |x|<\frac{\pi}{2}
    \end{align}
From  $h(c)=\Lambda(\frac{c}{2})+\Lambda(\frac{1}{2}-\frac{c}{2})$, we obtain 
\begin{align}
    h'(c)&=\frac{1}{2}\log \cot\left(\frac{\pi c}{2}\right), \\\nonumber
    h''(c)&=-\frac{\pi }{2\sin(\pi c)}=-\frac{\pi}{2}\sec\left(\pi\left(c-\frac{1}{2}\right)\right). 
\end{align}
    Since $0\leq \left(c-\frac{1}{2}\right)\pi<\frac{\pi}{2} $, by using formula (\ref{formula-sec}), we obtain 
    \begin{align}
        h'(c)&=\int_{\frac{1}{2}}^ch''(t)dt-h'\left(\frac{1}{2}\right)\\\nonumber
        &=-\frac{\pi}{2}\int_{\frac{1}{2}}^c \sec\left(\pi\left(t-\frac{1}{2}\right)\right)dt\\\nonumber
        &=-\frac{1}{2}\int_{0}^{\pi\left(c-\frac{1}{2}\right)}\sec(x) dx\\\nonumber
        &=-\frac{1}{2}\int_{0}^{\pi\left(c-\frac{1}{2}\right)}\left(1+\frac{1}{2}x^2+\frac{5}{24}x^4+\cdots\right)dx\\\nonumber
        &=-\frac{1}{2}\left(\pi\left(c-\frac{1}{2}\right)+\frac{1}{6}\left(\pi(c-\frac{1}{2})\right)^3+\frac{1}{24}\left(\pi(c-\frac{1}{2})\right)^5+\cdots\right)
    \end{align}
    Hence 
    \begin{align}
        h(c)&=\int_{\frac{1}{2}}^c h'(t)dt+h\left(\frac{1}{2}\right)\\\nonumber
        &=-\frac{1}{2}\int_{\frac{1}{2}}^c\left(\pi\left(t-\frac{1}{2}\right)+\frac{\pi^3}{6}\left(t-\frac{1}{2}\right)^3+\frac{\pi^5}{24}\left(t-\frac{1}{2}\right)^5\right)dt+2\Lambda\left(\frac{1}{4}\right)\\\nonumber
        &=2\Lambda\left(\frac{1}{4}\right)-\frac{1}{2\pi}\int_{0}^{\pi\left(c-\frac{1}{2}\right)}\left(x+\frac{1}{6}x^3+\frac{1}{24}x^5+\cdots\right)dx\\\nonumber
        &=2\Lambda\left(\frac{1}{4}\right)-\frac{\pi}{4}\left(c-\frac{1}{2}\right)^2-\frac{\pi^3}{48}\left(c-\frac{1}{2}\right)^4-\frac{\pi^5}{144}\left(c-\frac{1}{2}\right)^6-\cdots.
    \end{align}
\end{proof}

\begin{lemma} \label{lemma-1-dim-hess}
    For $t\in D_0(c)$, $c\in [\frac{1}{2},\frac{3}{4})$ and $n\in \mathbb{Z}$, we have 
    \begin{align}
        \frac{\partial^2 ReV(p,t+X\sqrt{-1},c;0,n)}{\partial X^2}>0.
    \end{align}
\end{lemma}
\begin{proof}
By straightforward computations,  we have
\begin{align}
&\frac{1}{2\pi}\frac{\partial^2 ReV(p,t+X\sqrt{-1},c;0,n)}{\partial X^2}\\\nonumber
&=-3\frac{\sin(2\pi t)}{e^{2\pi X}+e^{-2\pi X}-2\cos(2\pi t)}\\\nonumber
&+\frac{\sin(2\pi (t+c))}{e^{2\pi X}+e^{-2\pi X}-2\cos(2\pi (t+c))}+\frac{\sin(2\pi (t-c))}{e^{2\pi X}+e^{-2\pi X}-2\cos(2\pi (t-c))}.
\end{align}

Clearly, for $t\in D_0(c)$ and $c\in [\frac{1}{2},\frac{3}{4})$, we have $\sin(2\pi t)<0$ and $\cos(2\pi c)<0$ which imply that 
\begin{align}
 -3\frac{\sin(2\pi t)}{e^{2\pi X}+e^{-2\pi X}-2\cos(2\pi t)}>0,   
\end{align}
and 
\begin{align}
    \cos(2\pi c)(e^{2\pi X}+e^{-2\pi X})-2\cos(2\pi t)<2\cos(2\pi c)-2\cos(2\pi t)<0.
\end{align}
Hence
\begin{align}
    &\frac{\sin(2\pi (t+c))}{e^{2\pi X}+e^{-2\pi X}-2\cos(2\pi (t+c))}+\frac{\sin(2\pi (t-c))}{e^{2\pi X}+e^{-2\pi X}-2\cos(2\pi (t-c))}\\\nonumber
    &=2\sin(2\pi t)\frac{\cos(2\pi c)(e^{2\pi X}+e^{-2\pi X})-2\cos(2\pi t)}{(e^{2\pi X}+e^{-2\pi X}-2\cos(2\pi (t+c)))(e^{2\pi X}+e^{-2\pi X}-2\cos(2\pi (t-c)))}>0.
\end{align}
\end{proof}
\begin{lemma} \label{lemma-1-dim-finfty}
    For $t\in D_0(c)$ and $n\in \mathbb{Z}$, we have 
    \begin{align} \label{formula-f-infty}
        ReV(p,t+X\sqrt{-1},c;0,n)\ \text{goes to $\infty$ uniformly}, \ \text{as} \ X^2\rightarrow \infty. 
    \end{align}
\end{lemma}
\begin{proof}
Based on Lemma \ref{lemma-Li2}, we introduce the following  function for $t\in D_0(c)$
\begin{align}
    F(X;n)=\left\{ \begin{aligned}
         &X  &  \ (\text{if} \ X\geq 0), \\
         &\left(\frac{1}{2}-t\right)X & \ (\text{if} \ X<0).
                          \end{aligned} \right.
\end{align}
since $t>\frac{1}{2}$, we have 
\begin{align}
    F(X;n)\rightarrow \infty \ \text{as} \ X^2\rightarrow \infty,
\end{align}
and by Lemma \ref{lemma-Li2}, we obtain  
\begin{align}
    2\pi F(X;n)-C < ReV(p,t+X\sqrt{-1},c;0,n)< 2\pi F(X;n)+C,
\end{align}
which implies formula (\ref{formula-f-infty}). 
\end{proof}

Now, we can finish the proof of Proposition \ref{prop-saddleonedim1}.
\begin{proof}

We show that there exists a homotopy $S^{\delta}$ ($0\leq \delta\leq 1$) with $S^0=D_{0}(c)$ such that 
\begin{align}
    &(T_1(c),c)\in S^{1}, \label{saddle1-1} \\  
    &S^1-\{(T_1(c),c)\}\subset \{t\in \mathbb{C}|\text{Re} V(p,t,c;0,n)<\text{Re} V(p,T_1(c),c)\}, \label{saddle1-2}\\
    &\partial S^\delta \subset \{t\in \mathbb{C}| \text{Re} V(p,t,c;0,n)<\zeta_{\mathbb{R}}(p)-\epsilon\}  \label{saddle1-3}.
\end{align}
In the fiber of the projection $\mathbb{C}\rightarrow \mathbb{R}$ at $(t,c)\in D_0(c)$, we consider the flow from $X=0$ determined by the vector field $-\frac{\partial Re V}{\partial X}$, 
By Lemma \ref{lemma-1-dim-hess} and \ref{lemma-1-dim-finfty}, we obtain that,
    for $t\in D_0(c)$, then $\text{Re} V$ has a unique minimal point, and the flow goes there. We put $g(t,c;n)$ as the minimal point. We define the end of homotopy to be the set of the destinations of the flow
\begin{align}
        S^1=\{t+g(t,c;n)\sqrt{-1}| t\in D_{0}(c)\}. 
    \end{align}
In addition, we define the internal part of the homotopy by setting it along the flows. 

We show $(\ref{saddle1-3})$ as follows, from the definition of $D_{0}(c)$,
\begin{align}
    \partial D_{0}(c)\subset \partial D_0\subset \{t\in \mathbb{C}| \text{Re} V(p,t,c;0,n)< \zeta_{\mathbb{R}}(p)-\epsilon\}.   
\end{align}
Furthermore, by construction of homotopy, $\text{Re} V(p,t,c;0,n)$ monotoincally decreases along the homotopy. Hence $(\ref{saddle1-3})$ holds. 

We show (\ref{saddle1-1}) and (\ref{saddle1-2}) as follows. Consider the following function
\begin{align}
    h(t,c;n)=\text{Re} V(p,t+g(t,c;n)\sqrt{-1},c;0,n).
\end{align}
It is shown from the definition of $g(t,c;n)$ that 
\begin{align}
    \frac{\partial \text{Re} V(p,t+g(t,c;n)\sqrt{-1},c;0,n)}{\partial X}=0 \ \text{at} \ X=g(t,c;n). 
\end{align}
Hence, we have
\begin{align}
    \text{Im} \frac{\partial V}{\partial t}=0 \  \text{at} \ t+g(t,c;n)\sqrt{-1}.  
\end{align}
Furthermore, we also have 
\begin{align}
    \frac{d h}{d t}=\text{Re} \frac{\partial V}{\partial t} \ \text{at} \ t+g(t,c;n)\sqrt{-1}. 
\end{align}
Therefore, if $t+g(t,c;n)\sqrt{-1}$ is a critical point of $\text{Re} V$,  $t$ is a critical point of the function $h(t,c;n)$. By Proposition \ref{prop-critical1}, $h(t,c;n)$ has a unique maximal point at $t=t_1(c)$.  Moreover, by Corollary $\ref{coro-ReV1}$,  the maximal value 
$$
h(t_1(c),c;n)=\text{Re}V(p,T_1(c),c)<\text{Re} V(p,T_1(c_{upper}(p),c_{upper}(p))=\zeta_{\mathbb{R}}(p)-\epsilon,
$$
for some small $\epsilon>0$. 
Therefore, (\ref{saddle1-1}) and (\ref{saddle1-2}) hold. The assumption of the saddle point method in one dimension is verified and we finish the proof of Proposition \ref{prop-saddleonedim1}.  
\end{proof}

\begin{remark}
   By using the same method as illustrated above,  one can prove
\begin{proposition} \label{prop-saddleonedim2}
    For $0<c\leq 1-c_{upper}(p)$ and $n\in\mathbb{Z}$,  there exists a constant $C$ independent of $c$,  such that
    \begin{align}
        |\int_{D_0(c)}\psi(t,c)\sin(2\pi c) e^{(N+\frac{1}{2})V_N(p,t,c;0,n)}dt|<C\left(e^{(N+\frac{1}{2})\left(\zeta_{\mathbb{R}}(p)-\epsilon\right)}\right). 
    \end{align}
\end{proposition}
In fact, Proposition \ref{prop-saddleonedim2} can also be derived from the symmetric property of the function $\text{Re}V(p,t,c;0,n)$ with respect to the line $c=\frac{1}{2}$.  
\end{remark}

Therefore, combining Proposition \ref{prop-saddleonedim1} and Proposition \ref{prop-saddleonedim2} together, we obtain Proposition \ref{prop-saddleonedim}.

\end{document}